\documentclass[11pt]{article}
\usepackage{amsmath}
\usepackage{amssymb,amsmath}
\usepackage{mathrsfs}
\usepackage{amsmath,amsthm,amsfonts,amssymb,bm,euscript}
\usepackage{fancyhdr,amscd}
\usepackage{anysize}
\usepackage{graphicx,epsfig}
\usepackage{amsthm,latexsym,amsmath,amssymb}
\usepackage{amssymb,amsmath}
\usepackage{mathrsfs}
\usepackage{graphicx}
\usepackage{color}
\usepackage{mathrsfs}
\usepackage{cite}
\usepackage{indentfirst}
 \usepackage[titletoc]{appendix}
 \usepackage [latin1]{inputenc}

\usepackage{amsmath, amssymb}
\usepackage{amsthm, amsfonts, mathrsfs}
\usepackage{mathptmx}
\usepackage{fullpage}
\usepackage{amsfonts,graphicx}
\usepackage[colorlinks=true,  linkcolor=blue, citecolor=blue, urlcolor=blue]{hyperref}
\def\C{\mathop{\bf C\kern 0pt}\nolimits}
\def\DD{\mathop{\bf D\kern 0pt}\nolimits}
\def\K{\mathop{\bf K\kern 0pt}\nolimits}
\def\N{\mathop{\bf N\kern 0pt}\nolimits}
\def\Q{\mathop{\bf Q\kern 0pt}\nolimits}
\def\R{\mathop{\bf R\kern 0pt}\nolimits}

\def\V{{\mathbf{V}}}
\def\H{\mathbf{H}}

\def\uu{\textbf{\textit{u}}}
\def\vv{\textbf{\textit{v}}}

\def\C{\mathcal{C}}

\renewcommand{\div}{\mbox{\rm div}\;\!}

\newcommand{\Dv}{{\rm div}}

\newcommand{\beq}{\begin{equation}}
\newcommand{\eeq}{\end{equation}}
\newcommand{\ben}{\begin{eqnarray}}
\newcommand{\een}{\end{eqnarray}}
\newcommand{\beno}{\begin{eqnarray*}}
\newcommand{\eeno}{\end{eqnarray*}}

\newtheorem{theorem}{Theorem}[section]

\newtheorem{lemma}[theorem]{Lemma}

\numberwithin{equation}{section}

\allowdisplaybreaks \numberwithin{equation} {section}

\begin{document}
\title{The unique  solvability   of  strong solution  to the  multi-dimensional  nonhomogeneous incompressible two-phase magnetohydrodynamic  model}
\author{ Lingxin  Jiang \ \     Fuyi  Xu$^{\dag}$\\[2mm]
 { \small   School of Mathematics and  Statistics, Shandong University of Technology,}\\
  { \small Zibo,    255049,  Shandong,    China}
   }
         \date{}
\maketitle
\noindent{\bf Abstract} \ \ The present paper studies the  initial-boundary value problem  of the  nonhomogeneous incompressible two-phase magnetohydrodynamic model with  Landau potential in a bounded smooth domain in $\mathbb{R}^d$($d = 2, 3$). More precisely, we  construct  the existence of  local in time   in three dimension and  the global existence of strong solution  in two  dimension with  arbitrary large data and bounded  density. In addition,  the uniqueness of the solution is proved through the weighted energy estimates, the shift of integrability method and Lagrangian approach.
 \vskip   0.2cm \noindent{\bf Key words: } nonhomogeneous incompressible two-phase magnetohydrodynamic model; well-posedness;  large initial data; bounded  density.
\vskip   0.2cm \footnotetext[1]{$^\dag$Corresponding author.}
\vskip   0.2cm \footnotetext[2]{E-mail addresses: Jianglingxin2025@163.com(L. Jiang), \  zbxufuyi@163.com(F. Xu).}

\setlength{\baselineskip}{20pt}

\vskip .2in
\section{Introduction and Main Results}
\label{pro}
The inhomogeneous  incompressible magnetohydrodynamics (MHD) equations usually describe  the motion of several conducting incompressible immiscible fluids (without surface tension) in presence of a magnetic field.
Mathematically, the model consists of the inhomogeneous  incompressible
Navier-Stokes equations of fluid dynamics and the Maxwell equations of
electromagnetism (see \cite{JCT,KL}). Specifically, the governing equations of nonhomogeneous incompressible MHD can be stated as follows
\begin{align} \label{1.0MHD}
\left\{
\begin{aligned}
&\partial_{t}\rho+\Dv (\rho\mathbf{u})=0,\\
&\rho\partial_{t}\mathbf{u}+\rho(\mathbf{u}\cdot\nabla)\mathbf{u}
-\Dv(\nu D\mathbf{u})+\nabla P
=\mathbf{b}\cdot\nabla \mathbf{b},\\
&\partial_{t}\mathbf{b}+\mathbf{u}\cdot\nabla \mathbf{b}-\mathbf{b}\cdot\nabla\mathbf{u}-\eta\Delta \mathbf{b}=0, \\
&\Dv \mathbf{u}=0,\\
&\Dv \mathbf{b}=0,
\end{aligned}
\right.
\end{align}
where $\rho = \rho(t, x)$, $\boldsymbol{u} = \boldsymbol{u}(t, x)$, $\boldsymbol{b} = \boldsymbol{b}(t, x)$, $ D \mathbf{u} = \frac{1}{2} ( \nabla \mathbf{u} + ( \nabla \mathbf{u} )^t )$,  and $P = P(t, x)$ represent the density, velocity, the magnetic field,  the deformation tensor  and the pressure, respectively.  $\nu\geq0$ is the kinematic viscosity and $\eta\geq0$ is the magnetic diffusivity.

The magnetohydrodynamics (MHD) equations accurately describe the macroscopic motion laws of conducting fluids in electromagnetic fields, and have irreplaceable theoretical value and application prospects in fields such as astrophysics, magnetic confinement fusion, and liquid metal engineering(see \cite{NPX,UM,xu2}). However, based on the continuum hypothesis, the traditional MHD equations face significant challenges in the well-posedness analysis of their mathematical models when dealing with multiphase conducting fluid problems involving phase interface evolution and density discontinuities\cite{MSBKT}, so that people have to study  the interaction of electromagnetic fields with two incompressible, immiscible and electrically conducting fluids, i.e., a two-phase MHD problem. For the phase field $\varphi$, the free energy of two-phase fluids is
$$E(\varphi)=\int_\Omega\big(\frac{1}{2}|\nabla\varphi|^2+\frac{1}{\varepsilon^2}\Psi(\varphi)\big)\mathrm{d}x,$$
where $\Psi(\varphi)$ models the immiscibility of the fluid components.
To preserve the mass conservation, i.e., $\frac{\mathrm{d}}{\mathrm{d}t}\int_{\Omega}\varphi(x,t)\mathrm{d}x=0$,  one  considers the following  Cahn-Hilliard equations
\begin{align}
\left\{
\begin{aligned}
\varphi_t&=\mathrm{div}\left(\kappa\nabla\frac{\partial E}{\partial\varphi}\right)=\kappa\Delta \mu,\\
\mu&=\frac{\partial E}{\partial\varphi}=-\Delta\varphi+\frac{1}{\varepsilon^2}f(\varphi),
\end{aligned}
\right.
\end{align}
where $\mu$ represents the chemical potential which is given by the variational derivative of the energy $E$ with respect
to $\varphi$, $f(\varphi)=\Psi^{\prime}(\varphi)$, and $\kappa$, $\varepsilon$ denote the mobility of the mixture and width of the interfacial layer, respectively.

Then combining the physics of MHD fluids and the phase field approach, researchers \cite{SZ,YMHYH} established a novel model for comprehensively and accurately describing the complex flow behaviors of two incompressible, immiscible, electrically conducting fluids with different viscosities and electrical conductivities, which  is called  the nonhomogeneous incompressible  two-phase magnetohydrodynamic model. The governing equation consists of the Cahn-Hilliard equations, nonhomogeneous incompressible Navier-Stokes equations, and the Maxwell's equations, which are coupled through convection,
stresses, and Lorentz forces. More precisely,  the nonhomogeneous incompressible two-phase magnetohydrodynamic model reads as follows
\begin{align} \label{1.1}
\left\{
\begin{aligned}
&\partial_{t}\rho+\Dv (\rho \mathbf{u})=0,\\
&\rho\partial_{t}\mathbf{u}+\rho(\mathbf{u}\cdot\nabla)\mathbf{u}
-\div(\nu(\phi)D\mathbf{u})+\nabla P
=(\mathbf{b}\cdot\nabla)\mathbf{b}-\div(\nabla\phi\otimes\nabla\phi),\\
&\partial_{t}\mathbf{b}+(\mathbf{u}\cdot\nabla)\mathbf{b}-(\mathbf{b}\cdot\nabla)\mathbf{u}
-\Delta \mathbf{b}=0, \\
&\rho\partial_{t}\phi+\rho\mathbf{u}\cdot\nabla\phi=\Delta\mu,\\
&\rho\mu=-\Delta\phi+\rho\Psi'(\phi),\\
&\Dv \mathbf{u}=0,\, \Dv \mathbf{b}=0,
\end{aligned}
\right.
\end{align}
subject to the boundary and initial conditions
\begin{equation}\label{1.2MHD}
\begin{cases}
&\mathbf{u}|_{\partial\Omega}=0, \quad \mathbf{b}|_{\partial\Omega}=0, \quad\frac{\partial\mu}{\partial n}
|_{\partial\Omega}=0, \quad \frac{\partial\phi}{\partial n}
|_{\partial\Omega}=0, \quad \text{on } \partial\Omega \times (0, T), \\
&\rho(\cdot, 0)=\rho_0, \quad \mathbf{u}(\cdot, 0)=\mathbf{u}_0, \quad \mathbf{b}(\cdot, 0)=\mathbf{b}_0, \quad \phi(\cdot, 0)=\phi_0, \quad \text{in } \Omega,
\end{cases}
\end{equation}
where the surface tension $\operatorname{div}(\nabla\varphi\otimes\nabla\varphi)$ is often called the Korteweg force, and $\lambda$ is capillary coefficient, $\Dv \boldsymbol{u}_0 = \Dv \boldsymbol{b}_0 = 0$ and $n$ is the outward unit normal to the boundary $\partial\Omega$.

The system \eqref{1.1} is a nonlinear coupled system, which has extensive application prospects in the fields of nuclear fusion,
metallurgy, liquid metal magnetic pumps, aluminum electrolysis and so on \cite{HR,JCT,MSBKT}. We want to point out that the system \eqref{1.1} includes several important models as special cases. When $b=0$, it is the well-known nonhomogeneous incompressible Navier-Stokes-Cahn-Hilliard
 system, which has been studied by many researchers  (see, e.g., \cite{A, ADG, AMH, AAR,SGM}) and references therein.
When $\varphi$ is  absence,  the system \eqref{1.1} turns into the classical incompressible nonhomogeneous MHD equations \eqref{1.0MHD}. Due to its mathematical challenges and broad physical applications, there are extensive literature on the well-posedness of solutions for this system (see, e.g.,\cite{CCM3,CCM4,CW,wu3,wu4,xu2,XQF,XQF1}).
At present, the study on the   nonhomogeneous incompressible two-phase magnetohydrodynamic model  can be traced back to \cite{SZ,YMHYH,wu4}(homogeneous case, i.e., $\rho=$constant). However, to the best of our knowledge,  so far there is no
result regarding on the mathematical analysis of the  system \eqref{1.1}.

\ \ \ Throughout this work we will assume that the viscosity \(\nu = \nu(s) \in W^{1,\infty}(\mathbb{R})\) is such that \(0 < \nu_* \leq \nu(s) \leq \nu^*\) for all \(s \in \mathbb{R}\). For the potential \(\Psi(s)\), we will consider the following  Landau potential
\begin{equation}\label{wei1.4}
\Psi(s) = \frac{1}{4} (s^2 - 1)^2 \quad \forall s \in \mathbb{R}.
\end{equation}

Our first main result on the local well-posedness of strong solutions for the three-dimensional case then reads  as follows.
\begin{theorem}\label{1.2} When $d =3$. Let $\Omega$ be a bounded domain of class $C^3$ in $\mathbb{R}^3$.
	Suppose the initial data satisfy $\rho_0 \in L^{\infty}(\Omega)$,
	$\mathbf{u}_0, \mathbf{b}_0\in \mathbf{V}_{\sigma}(\Omega)$, and $\phi_0 \in H^2(\Omega)$ with
	$0<\rho_* \leq \rho_0 \leq \rho^*, \quad
	\partial_{\mathbf{n}} \phi_0 = 0 \ \text{on } \partial \Omega, \quad
	\mathbf{u}_0 = -\frac{\Delta \phi_0}{\rho_0} + \Psi'(\phi_0) \in H^1(\Omega).$
	Then there exists $T>0$, depending on the size of the initial data,
	the initial--boundary value problem
	\eqref{1.1}-\eqref{1.2MHD} the initial-boundary value problem
	\eqref{1.1}-\eqref{1.2MHD} with Landau potential \eqref{wei1.4} admits  the  unique  local strong solution
	$(\rho, \mathbf{u}, P, \phi, \mu)$ such that
	\begin{align*}
	0&<\rho_* \leq \rho(x, t) \leq \rho^*,\\
	\rho &\in \mathcal{C}([0,T]; L^r(\Omega)) \cap L^\infty(\Omega \times (0,T)) \cap L^\infty(0,T; H^{-1}(\Omega)), \,  r \in [2,\infty),\\
	\mathbf{u} &\in \mathcal{C}([0,T]; \mathbf{V}_{\sigma}) \cap L^2(0,T; H^2(\Omega)) \cap H^1(0,T; \mathbf{H}_{\sigma}), \\
\mathbf{b} &\in \mathcal{C}([0,T]; \mathbf{V}_{\sigma}) \cap L^2(0,T; H^2(\Omega)), \\
		P &\in L^2(0,T; H^1(\Omega)), \\
		\phi &\in \mathcal{C}([0,T]; (W^{2,6}(\Omega))_w) \cap H^1(0,T; H^1(\Omega)), \\
		\mu &\in L^\infty(0,T; H^1(\Omega)) \cap L^2(0,T; W^{2,6}(\Omega)).
	\end{align*}
\end{theorem}

Our second main result on  the global well-posedness of strong solutions for the two-dimensional case is the following theorem.
\begin{theorem}\label{1.3} When $d=2$. Let $\Omega$ be a bounded domain of class $C^3$ in $\mathbb{R}^2$.
	Suppose the initial data satisfy $\rho_0 \in L^{\infty}(\Omega)$,
	$\mathbf{u}_0, \mathbf{b}_0\in \mathbf{V}_{\sigma}(\Omega)$, and $\phi_0 \in H^2(\Omega)$ with
	$0<\rho_* \leq \rho_0 \leq \rho^*, \quad
	\partial_{\mathbf{n}} \phi_0 = 0 \ \text{on } \partial \Omega, \quad
	\mathbf{u}_0 = -\frac{\Delta \phi_0}{\rho_0} + \Psi'(\phi_0) \in H^1(\Omega).$
	Then,  for any $T>0$, the initial-boundary value problem
	\eqref{1.1}-\eqref{1.2MHD} with Landau potential \eqref{wei1.4} admits the  unique global  strong solution
	$(\rho, \mathbf{u}, P, \phi, \mu)$ such that
	\begin{align*}
0&<\rho_* \leq \rho(x, t) \leq \rho^*,\\
\rho &\in \mathcal{C}([0,T]; L^r(\Omega)) \cap L^\infty(\Omega \times (0,T)) \cap L^\infty(0,T; H^{-1}(\Omega)),\, r \in [2,\infty), \\
		\mathbf{u} &\in \mathcal{C}([0,T]; \mathbf{V}_\sigma) \cap L^2(0,T; H^2(\Omega)) \cap H^1(0,T; \mathbf{H}_\sigma), \\
\mathbf{b} &\in \mathcal{C}([0,T]; \mathbf{V}_{\sigma}) \cap L^2(0,T; H^2(\Omega)) , \\		P &\in L^2(0,T; H^1(\Omega)), \\
		\phi &\in \mathcal{C}([0,T]; (W^{2,q}(\Omega))_w) \cap H^1(0,T; H^1(\Omega)), \\
		\mu &\in L^\infty(0,T; H^1(\Omega)) \cap L^2(0,T; W^{2,q}(\Omega)), q \in [2,\infty).
	\end{align*}
\end{theorem}

Before getting into the details in the proof of our main result, we first outline the key components, challenges, and strategies involved. First,
we  need  deal with some  difficulties caused by the more complex nonlinear terms such as  $H\cdot \nabla H$, $u\cdot \nabla H$  and  $H\cdot \nabla u$,    and the hyperbolic-parabolic  coupling effect among the density,  the velocity field,  the magnetic field  and the phase field in the  system \eqref{1.1}. As for the uniqueness of solutions,  authors in \cite{AR1}  remained  an interesting open  issue for the inhomogeneous incompressible Navier-Stokes-Cahn-Hilliard system with   bounded  density. Here, it should be emphasized that, the boundedness of density field is an essential difficulty. In particular, it seems impossible to prove the uniqueness of the strong  solution in the Eulerian coordinates as in \cite{CHK,GT1,Lijinkai}. Indeed, let $(\rho_{1},\mathbf{u}_{1},P_{1},\phi_{1},\mu_{1})$ and $(\rho_{2},\mathbf{u}_{2},P_{2},\phi_{2},\mu_{2})$ be two different solutions of the  system \eqref{1.1}. Then
 $\delta \rho=\rho_1-\rho_2$ satisfies
 \begin{equation*}
\delta \rho_t + \mathbf{u}_1 \cdot \nabla \delta \rho = - (\mathbf{u}_1 - \mathbf{u}_2) \cdot \nabla \rho_2.
\end{equation*}
Without extra assumptions about the regularity of these solutions, the term $(\mathbf{u}_1 - \mathbf{u}_2) \cdot \nabla \rho_2$ cannot be handled by the energy method because the usual technique to prove uniqueness via Gronwall's inequality  cannot be applied here. And the uniqueness result of Germain \cite{PG} cannot be applied here either, which requires the density function satisfying
 $\nabla \rho \in L^{\infty}\big(0,T; L^{d}(\Omega)\big)$. Consequently, the uniqueness issue is non-trivial due to the roughness of the density and the hyperbolic nature of the continuity equation. To address this problem, we shall use the Lagrangian coordinates defined by the stream lines, which is motivated by \cite{RB1,RB,PZZ}. 
 According to the pioneering work by D. Hoff in \cite{DH} or to the recent papers \cite{RB1,RB,PZZ}, in most evolutionary fluid mechanics models, the condition  $\nabla \mathbf{u} \in L^{1}\big(0,T; L^{\infty}(\Omega)\big)$   seems to be the minimal requirement in order to get
uniqueness. However,  when the density is rough, propagate enough regularity for the velocity is the main difficulty. In order to  bound the quantity $\int_{0}^{T}\|\nabla \mathbf{u}\|_{L^{\infty}}d\tau$,  we first  exploit some  extra time-weighted energy estimates for the velocity field. Combining with these time-weighted estimates, interpolation results, classical Sobolev embedding, and shift of integrability from the time variable to the space variable, we eventually get the Lipschitz control of the velocity field.
\section{Preliminaries} This section reviews various tools, including some  functional settings,  important inequalities,  and useful lemmas that will be referenced throughout the paper.

Let $\Omega \subset \mathbb{R}^d$ ($d=2$ or $3$) be a bounded domain with smooth boundary $\partial \Omega$.
For $k \in \mathbb{N}$ and $1 \leq p \leq \infty$, we write $W^{k,p}(\Omega)$ for the Sobolev space of functions in $L^p(\Omega)$ with distributional derivatives up to order $k$ in $L^p(\Omega)$, with norm $\| \cdot \|_{W^{k,p}(\Omega)}$.
The Hilbert spaces $W^{k,2}(\Omega)$ are denoted by $H^k(\Omega)$ with norm $\| \cdot \|_{H^k(\Omega)}$.
We denote by $H_0^1(\Omega)$ the closure of $\mathcal{C}_0^\infty(\Omega)$ in $H^1(\Omega)$, and by $H^{-1}(\Omega)$ its dual.

We set $H=L^2(\Omega)$, with inner product $(\cdot,\cdot)$ and norm $\|\cdot\|$.
Let $V=H^1(\Omega)$ with norm $\|\cdot\|_V$ and dual space $V'$ with norm $\|\cdot\|_{V'}$.
The duality pairing between $V'$ and $V$ is denoted by $\langle \cdot , \cdot \rangle$.
For $u \in V'$, we denote by $\overline{u}$ the mean value of $u$ over $\Omega$, i.e., $\overline{u}=|\Omega|^{-1}\langle u,1 \rangle$.
By the generalized Poincare inequality (see \cite[Chapter II, Section 1.4]{R}), we have
\begin{equation}
u \mapsto \Big( \|\nabla u\|_{L^2(\Omega)}^2 + \big| \frac{1}{|\Omega|} \int_\Omega u dx \big|^2 \Big)^{\frac{1}{2}} \quad \text{and} \quad u \mapsto \Big( \|\nabla u\|_{L^2(\Omega)}^2 + \big| \int_\Omega \eta u dx \big|^2 \Big)^{\frac{1}{2}}, \label{eq:3.1}
\end{equation}
where \( \eta \in L^\infty(\Omega) \) is such that \( 0 < \eta_* \leq \eta(x) \leq \eta^* \) for almost every \( x \in \Omega \), are norms on \( H^1(\Omega) \) equivalent to \( \|u\|_{H^1(\Omega)} \). In particular, there exists a positive constant \( C = C(\Omega, \eta_*, \eta^*) \) such that
\begin{equation}
\|u\|_{H^1(\Omega)} \leq C \Big( \|\nabla u\|_{L^2(\Omega)}^2 + \big| \int_\Omega \eta u dx \big|^2 \Big)^{\frac{1}{2}}, \quad \forall u \in H^1(\Omega). \label{eq:3.2}
\end{equation}
We now introduce the solenoidal function spaces.
Let $\mathcal{C}_{0,\sigma}^\infty(\Omega)$ denote the set of smooth, compactly supported, divergence-free vector fields.
We define
\[
\H_\sigma = \overline{\mathcal{C}_{0,\sigma}^\infty(\Omega)}^{\mathbf{L}^2},
\qquad
\V_\sigma = \overline{\mathcal{C}_{0,\sigma}^\infty(\Omega)}^{\mathbf{H}_0^1},
\]
endowed with the inner product and norm inherited from $\H$ and $H_0^1(\Omega)$, respectively.
On $\V_\sigma$, we also use the equivalent inner product and norm
\[
(\uu,\vv)_{\V_\sigma}=(\nabla \uu,\nabla \vv),
\qquad
\|\uu\|_{\V_\sigma}=\|\nabla \uu\|.
\]
We denote by $\V_\sigma'$ its dual space.

We recall that  Korn's inequality entails
\begin{equation}
\| \nabla \mathbf{u} \|_{L^2(\Omega)} \leq \sqrt{2} \| D \mathbf{u} \|_{L^2(\Omega)} \leq \sqrt{2} \| \nabla \mathbf{u} \|_{L^2(\Omega)}, \quad \forall \mathbf{u} \in \mathbf{V}_\sigma. \label{eq:3.8}
\end{equation}

We  recall the following Gagliardo-Nirenberg and Agmon inequalities.
\begin{lemma}\label{chazhi}\cite{Temam}
\begin{align}
\|u\|_{L^4(\Omega)} &\leq C \|u\|_{L^2(\Omega)}^{\frac{1}{2}} \|u\|_{H^1(\Omega)}^{\frac{1}{2}}, \quad \forall u \in H^1(\Omega), \quad \text{if } d = 2, \label{eq:3.3} \\
\|u\|_{L^3(\Omega)} &\leq C \|u\|_{L^2(\Omega)}^{\frac{1}{2}} \|u\|_{H^1(\Omega)}^{\frac{1}{2}}, \quad \forall u \in H^1(\Omega), \quad \text{if } d = 3, \label{eq:3.4} \\
\|u\|_{L^\infty(\Omega)} &\leq C \|u\|_{L^2(\Omega)}^{\frac{1}{2}} \|u\|_{H^2(\Omega)}^{\frac{1}{2}}, \quad \forall u \in H^2(\Omega), \quad \text{if } d = 2, \label{eq:3.5} \\
\|u\|_{L^\infty(\Omega)} &\leq C \|u\|_{H^1(\Omega)}^{\frac{1}{2}} \|u\|_{H^2(\Omega)}^{\frac{1}{2}}, \quad \forall u \in H^2(\Omega), \quad \text{if } d = 3. \label{eq:3.6}
\end{align}
\end{lemma}
We here  present  the following logarithmic estimate of the product, which plays a key role in the proof of the global existence for two dimensional case.
\begin{lemma}\label{log}\cite{AAR} When $d=2$, $\forall u, v \in H^1(\Omega),$ then
\begin{equation}\label{eq:3.7}
\|uv\|_{L^2(\Omega)} \leq C \|u\|_{H^1(\Omega)} \|v\|_{L^2(\Omega)} \ln^{\frac{1}{2}} \Big( e \frac{\|v\|_{H^1(\Omega)}}{\|v\|_{L^2(\Omega)}} \Big).
\end{equation}
\end{lemma}

We also state a useful lemma, which plays a key role in the proof of uniqueness.
\begin{lemma}[\cite{DM1,RB}]\label{lemma:4.1}
	Let \( A \) be a matrix-valued function on \( [0,T] \times \Omega \) satisfying
	\begin{align*}
	\det A \equiv 1.
	\end{align*}
	There exists a constant \( c>0 \), depending only on \( d \), such that if
	\begin{align*}
	\| \mathrm{Id} - A \|_{L^\infty(0,T; L^\infty)} + \| A_t \|_{L^2(0,T; L^6)} \leq c,
	\end{align*}
	then for any function \( R : [0,T] \times \Omega \to \mathbb{R}^d \) such that
	\begin{align*}
	\operatorname{div} R \in L^2(0,T \times \Omega), \quad R \in L^4(0,T; L^2), \quad R_t \in L^{4/3}(0,T; L^{3/2}), \quad R \cdot \mathbf{n} \equiv 0 \ \text{on} \ (0,T)\times \partial \Omega,
	\end{align*}
	the equation
	\begin{align*}
	\operatorname{div}(Av) = \operatorname{div} R =: g \quad \text{in} \quad [0,T] \times \Omega
	\end{align*}
	admits a solution \( v \) in the space
	\begin{align*}
	X_T := \left\{ v \in L^2(0,T; H_0^1(\Omega)), \, v \in L^4(0,T; L^2(\Omega)), \, v_t \in L^{4/3}(0,T; L^{3/2}(\Omega)) \right\},
	\end{align*}
	which satisfies the following estimates for some constant \( C=C(d) \):
	\begin{equation}\label{4.9}
		\begin{split}
			&\| v \|_{L^4(0,T; L^2)} \leq C \| R \|_{L^4(0,T; L^2)},
			\quad \| \nabla v \|_{L^2(0,T; L^2)} \leq C \| g \|_{L^2(0,T; L^2)},  \\[6pt]
			& \| v_t \|_{L^{4/3}(0,T; L^{3/2})} \leq C \| R \|_{L^4(0,T; L^2)} + C \| R_t \|_{L^{4/3}(0,T; L^{3/2})}.
		\end{split}
	\end{equation}
\end{lemma}

\section{Existence of strong solutions}
In this section,  our central task is to prove  the existence of strong solutions to the initial-boundary value problem
	\eqref{1.1}-\eqref{1.2MHD} with Landau potential \eqref{wei1.4}.

\subsection{The three  dimensional case}
Here,  we shall prove the existence of local in time  strong solution  to the initial-boundary value problem
	\eqref{1.1}-\eqref{1.2MHD} with   arbitrary
large data. The proof is divided in several steps.
\subsubsection*{Step 1 The total energy balance.}
Multiplying \eqref{1.1}$_{2}$ by \(\mathbf{u}\) and integrating over \(\Omega\), we have
\begin{equation}
\int_{\Omega} \rho\partial_t \Big(\frac{\vert\mathbf{u}\vert^2}{2}\Big)\,\mathrm{d}x + \int_{\Omega} \rho\mathbf{u}\cdot\nabla \Big(\frac{\vert\mathbf{u}\vert^2}{2}\Big)\,\mathrm{d}x + \int_{\Omega} \nu(\phi)\vert D\mathbf{u}\vert^2\,\mathrm{d}x = \int \big((\mathbf{b}\cdot\nabla) \mathbf{b}\big)\cdot \mathbf{u}-\int_{\Omega} \text{div}(\nabla\phi\otimes\nabla\phi)\cdot\mathbf{u}\,\mathrm{d}x.
\end{equation}
Noting that
\begin{equation}
\int_{\Omega} \rho\partial_t \Big(\frac{1}{2}\vert\mathbf{u}\vert^2\Big)\,\mathrm{d}x + \int_{\Omega} \rho\mathbf{u}\cdot\nabla \Big(\frac{1}{2}\vert\mathbf{u}\vert^2\Big)\,\mathrm{d}x= \frac{\mathrm{d}}{\mathrm{d}t}\int_{\Omega} \frac{1}{2}\rho\vert\mathbf{u}\vert^2\,\mathrm{d}x - \int_{\Omega} (\partial_t\rho + \text{div}(\rho\mathbf{u}))\frac{\vert\mathbf{u}\vert^2}{2}\,\mathrm{d}x,
\end{equation}
and  using the integration by parts, we deduce
\begin{align*}
\frac{1}{2} \frac{d}{dt} \int_{\Omega} \rho |\mathbf{u}|^2 \, dx + \int_{\Omega} \nu(\phi) |D \mathbf{u}|^2 \, dx = \int_{\Omega} - \operatorname{div}(\nabla \phi \otimes \nabla \phi) \cdot\mathbf{u} \, dx.
\end{align*}
Moreover, in light of the  following relations
\begin{align*}
- \operatorname{div}(\nabla \phi \otimes \nabla \phi) &= - \Delta \phi \nabla \phi - \nabla \big( \frac{1}{2} |\nabla \phi|^2 \big) \\
&= \rho \mu \nabla \phi - \rho \Psi'(\phi) \nabla \phi - \nabla \left( \frac{1}{2} |\nabla \phi|^2 \right) \\
&= \rho \mu \nabla \phi - \rho \nabla \Psi(\phi) - \nabla \left( \frac{1}{2} |\nabla \phi|^2 \right),
\end{align*}
we conclude  that
\begin{equation}\label{3.3}
\frac{1}{2} \frac{\mathrm{d}}{\mathrm{d}t} \int_{\Omega} \rho |\mathbf{u}|^2 \, dx + \int_{\Omega} \nu(\phi) | \mathbf{u}|^2 \, dx = \int\big(( \mathbf{b}\cdot\nabla)  \mathbf{b}\big)\cdot \mathbf{u}+\int_{\Omega} \rho \mu \nabla \phi \cdot\mathbf{u} \, dx - \int_{\Omega} \rho \mathbf{u} \cdot \nabla \Psi(\phi) \, dx.
\end{equation}
Next, multiplying \eqref{1.1}$_{4}$ by \(\mu\) and integrating over \(\Omega\), we find that
\begin{align*}
\int_{\Omega} \rho \partial_t \phi \mu \, dx + \int_{\Omega} \rho  \mathbf{u} \cdot \nabla \phi \mu \, dx + \int_{\Omega} |\nabla \mu|^2 \, dx = 0.
\end{align*}
Employing \eqref{1.1}$_{4}$ and \eqref{1.1}$_{5}$, we have
\begin{align*}
\int_{\Omega} \rho \partial_t \phi \mu \, dx &= - \int_{\Omega} \partial_t \phi \Delta \phi \, dx + \int_{\Omega} \rho \Psi'(\phi) \partial_t \phi \, dx \\
&= \frac{1}{2} \frac{d}{dt} \int_{\Omega} |\nabla \phi|^2 \, dx + \int_{\Omega} \rho \partial_t \Psi(\phi) \, dx \\
&= \frac{d}{dt} \int_{\Omega} \left( \frac{1}{2} |\nabla \phi|^2 + \rho \Psi(\phi) \right) dx - \int_{\Omega} \partial_t \rho \Psi(\phi) \, dx.
\end{align*}
Then, we get
\begin{equation}\label{3.4}
\frac{d}{dt} \int_{\Omega} \left( \frac{1}{2} |\nabla \phi|^2 + \rho \Psi(\phi) \right) dx + \int_{\Omega} \rho \mathbf{u} \cdot \nabla \phi \mu \, dx + \int_{\Omega} |\nabla \mu|^2 \, dx = \int_{\Omega} \partial_t \rho \Psi(\phi) \, dx.
\end{equation}
Multiplying  \eqref{1.1}$_{3}$ by \(\mathbf{b}\) and integrating over \(\Omega\), we have
\begin{equation}\label{3.5}
\int_{\Omega} \partial_t \left(\frac{\vert\mathbf{b}\vert^2}{2}\right)\,\mathrm{d}x+\int_{\Omega} \big((\mathbf{u}\cdot\nabla) \mathbf{b}\big)\cdot \mathbf{b}dx-\int_{\Omega}\big((\mathbf{b}\cdot\nabla) \mathbf{u}\big)\cdot \mathbf{b}dx+\int_{\Omega}|\nabla \mathbf{b}|^2 dx=0.
\end{equation}
By using $\text{div}\mathbf{u}=0$ and
\begin{equation*}
\int_{\Omega}\big((\mathbf{b}\cdot\nabla) \mathbf{u}\big)\cdot \mathbf{b}dx=-\int_{\Omega}\big((\mathbf{b}\cdot\nabla) \mathbf{b}\big)\cdot \mathbf{u}dx,\notag
\end{equation*}
and then  summing \eqref{3.3}-\eqref{3.5}, we arrive at
\begin{align*}
\frac{\mathrm{d}}{\mathrm{d}t} \int_{\Omega}\big( \frac{1}{2} &\rho |\mathbf{u}|^2 + \frac{1}{2} |\mathbf{b}|^2 + \frac{1}{2} |\nabla \phi|^2 + \rho \Psi(\phi)\big) \, \mathrm{d}x + \int_{\Omega} \nu(\phi) |D \mathbf{u}|^2\\
&+|\nabla\mathbf{b}|^2 \, \mathrm{d}x + \int_{\Omega} |\nabla \mu|^2 \, \mathrm{d}x = \int_{\Omega} \partial_t \rho \Psi(\phi) \, \mathrm{d}x - \int_{\Omega} \rho \mathbf{u} \cdot \nabla \Psi(\phi) \, \mathrm{d}x.
\end{align*}
After integrating by parts, and owing to  \eqref{1.1}$_{1}$, we have
\begin{equation*}
\int_{\Omega} \partial_t \rho \Psi(\phi) \, \mathrm{d}x - \int_{\Omega} \rho \mathbf{u} \cdot \nabla \Psi(\phi) \, \mathrm{d}x = \int_{\Omega} (\partial_t \rho + \mathrm{div}(\rho \mathbf{u})) \Psi(\phi) \, \mathrm{d}x = 0.
\end{equation*}
Therefore, we infer that
\begin{equation}\label{3.6A}
\frac{\mathrm{d}}{\mathrm{d}t} \int_{\Omega}\big( \frac{1}{2} \rho |\mathbf{u}|^2 + \frac{1}{2} |\mathbf{b}|^2 + \frac{1}{2} |\nabla \phi|^2 + \rho \Psi(\phi)\big) \, \mathrm{d}x + \int_{\Omega} \nu(\phi) |D \mathbf{u}|^2 \, \mathrm{d}x + \int_{\Omega} |\nabla \mu|^2 \, \mathrm{d}+\int_{\Omega}|\nabla\mathbf{b}|^2 \, \mathrm{d}x  = 0.
\end{equation}
Denoting
\begin{equation*}
E(\rho, \mathbf{u},\mathbf{b},\phi) = \int_{\Omega}\big( \frac{1}{2} \rho |\mathbf{u}|^2 + \frac{1}{2} |\mathbf{b}|^2 + \frac{1}{2} |\nabla \phi|^2 + \rho \Psi(\phi)\big) \, \mathrm{d}x ,
\end{equation*}
and integrating \eqref{3.6A}  in time,  for all \( t \in [0,T] \), we obtain
\begin{equation}\label{3.7AAA}
E(\rho(t), \mathbf{u}(t),\mathbf{b}(t), \phi(t)) + \int_{0}^{t} \int_{\Omega} \nu(\phi) |D \mathbf{u}|^2 +|\nabla\mathbf{b}|^2 + |\nabla \mu|^2 \, \mathrm{d}x \, \mathrm{d}\tau = E(\rho_0, \mathbf{u}_0,\mathbf{b}_0, \phi_0).
\end{equation}
Next, integrating \(\eqref{1.1}_1\) over \(\Omega\), and using the boundary condition of \(\mathbf{u}\), we have for all \( t \in [0,T] \)
\begin{equation}\label{3.8AA}
\int_{\Omega} \rho(t) \, \mathrm{d}x = \int_{\Omega} \rho_0 \, \mathrm{d}x.
\end{equation}
and  integrating \(\eqref{1.1}_4\) over \(\Omega\) and using the boundary condition of \(\mu\), we also have
\[\frac{\mathrm{d}}{\mathrm{d}t} \int_{\Omega} \rho \phi \, \mathrm{d}x - \int_{\Omega} (\partial_t \rho + \mathrm{div}(\rho \mathbf{u})) \phi \, \mathrm{d}x = 0.\]
Thus, by \(\eqref{1.1}_1\), for all \( t \in [0,T] \),  we have
\begin{equation}\label{3.9AA}
\int_{\Omega} \rho(t) \phi(t) \, \mathrm{d}x = \int_{\Omega} \rho_0 \phi_0 \, \mathrm{d}x.
\end{equation}
Hence
\begin{equation*}
\big|\int_{\Omega}\rho(t)\phi(t)\mathrm{d}x\big|=\big|\int_{\Omega}\rho_{0}\phi_{0}\mathrm{d}x\big|\leq C,
\end{equation*}
which  together with \eqref{eq:3.2} yields  that
\begin{equation}\label{3.10AA}
\|\phi\|_{L^{\infty}(0,T;H^1(\Omega))} \leq C.
\end{equation}
Moreover, taking the $L^2$-scalar product of \eqref{1.1}$_{5}$ with $\mu$, and  using integration by
parts formula and \eqref{3.8AA}, we have
\begin{equation*}\begin{split}
& \rho_* \|\mu\|_{L^2(\Omega)}^2 \leq \|\nabla\phi\|_{L^2(\Omega)} \|\nabla\mu\|_{L^2(\Omega)} + \rho^* \|\Psi'(\phi)\|_{L^2(\Omega)} \|\mu\|_{L^2(\Omega)}  \\
& \hspace{2em} \leq \|\nabla\phi\|_{L^2(\Omega)} \|\nabla\mu\|_{L^2(\Omega)} + C(1 + \|\phi\|_{L^6(\Omega)}^3) \|\mu\|_{L^2(\Omega)}.
\end{split}\end{equation*}
According to \eqref{3.10AA}, we get
\begin{equation}\label{4.30AA}
\|\mu\|_{L^2(\Omega)}^2 \leq C(1 + \|\nabla\mu\|_{L^2(\Omega)}).
\end{equation}
At last, taking the $L^2$-scalar product of \eqref{1.1}$_{5}$ with $\Delta\phi$, and  using integration by
parts formula  and \eqref{3.8AA}, we obtain
\begin{equation}\label{4.32AA}\begin{split}
\|\Delta\phi\|_{L^2(\Omega)}^2 &= -\int_{\Omega} \rho \mu \Delta\phi \mathrm{d}x + \int_{\Omega} \rho \Psi'(\phi) \Delta\phi \mathrm{d}x \\
& \leq \rho^* \|\mu\|_{L^2(\Omega)} \|\Delta\phi\|_{L^2(\Omega)} + \rho^{\ast} \|\Psi'(\phi)\|_{L^2(\Omega)} \|\Delta\phi\|_{L^2(\Omega)}\\
& \leq C(1 + \|\mu\|_{L^2(\Omega)} + \|\phi\|_{L^6(\Omega)}^3) \|\Delta\phi\|_{L^2(\Omega)}.
\end{split}\end{equation}
It follows from   \eqref{4.30AA} and \eqref{4.32AA}  that
\begin{equation}\label{4.24AA}
\|\mu\|_{H^1(\Omega)}\leq C(1 + \|\nabla\mu\|_{L^2(\Omega)}),\quad  \|\phi\|_{H^2(\Omega)}\leq C(1 + \|\nabla\mu\|_{L^2(\Omega)}^{\frac{1}{2}}). \end{equation}

\subsubsection*{Step 3 Higher-order energy estimates.}
\noindent
To begin with, taking the $L^2$-scalar product of \eqref{1.1}$_{2}$ with $\partial_t\mathbf{u}$, we have
\begin{equation}
\begin{split}\label{4.11}
\frac{\mathrm{d}}{\mathrm{d}t}&\int_{\Omega}\frac{\nu(\phi)}{2}|D
\mathbf{u}|^2\mathrm{d}x+\int_{\Omega}\rho|\partial_t\mathbf{u}|^2\mathrm{d}x+
\int_{\Omega}\rho(\mathbf{u}\cdot\nabla)\mathbf{u}
\cdot\partial_t\mathbf{u}\mathrm{d}x
+\int_{\Omega}(\mathbf{b}\cdot\nabla)\mathbf{b}\cdot\partial_t
\mathbf{u}\mathrm{d}x\\
&=\int_{\Omega}\rho\mu\nabla\phi\cdot\partial_t\mathbf{u}\mathrm{d}x
-\int_{\Omega}\rho\nabla\Psi(\phi)\cdot\partial_t\mathbf{u}\mathrm{d}x
+\int_{\Omega}\nu'(\phi)\partial_t\phi|\mathbb{D}\mathbf{u}|^2\mathrm{d}x.
\end{split}
\end{equation}
Applying the operator $\nabla$ to \eqref{1.1}$_{3}$, then taking the $L^2$-scalar product of with $\nabla\mathbf{b}$, and using integration by parts formula, we have
\begin{equation}\label{4.12}
\frac{\mathrm{d}}{\mathrm{d}t}\int_{\Omega}\frac{1}{2}|\nabla\mathbf{b}|^2\mathrm{d}x
+\int_{\Omega}|\Delta\mathbf{b}|^2\mathrm{d}x=
-\int_{\Omega}(\mathbf{u}\cdot\nabla)\mathbf{b}\cdot
\Delta\mathbf{b}\mathrm{d}x+
\int_{\Omega}(\mathbf{b}\cdot\nabla)\mathbf{u}\cdot
\Delta\mathbf{b}\mathrm{d}x.
\end{equation}
Taking the $L^2$-scalar product of \eqref{1.1}$_{4}$ with $\partial_t\mu$, we obtain
\begin{equation}\label{4.13}
\frac{1}{2}\frac{\mathrm{d}}{\mathrm{d}t}\|\nabla\mu\|_{L^2(\Omega)}^2
+\int_{\Omega}\rho\partial_t\phi \partial_t\mu\mathrm{d}x+\int_{\Omega}\rho\mathbf{u}\cdot\nabla\phi
\partial_t\mu\mathrm{d}x=0.
\end{equation}
By using  \eqref{1.1}$_{1}$ and  \eqref{1.1}$_{5}$, we find
\begin{equation}
\begin{split}\label{4.14}
\int_{\Omega}\rho\partial_t\phi\partial_t\mu\mathrm{d}x&
=\int_{\Omega}\partial_t(\rho\mu)\partial_t
\phi\mathrm{d}x
-\int_{\Omega}\mu\partial_t\rho\partial_t\phi\mathrm{d}x\\
&=\int_{\Omega}\Delta\phi\partial_t\phi\mathrm{d}x
+\int_{\Omega}\partial_t(\rho\Psi'(\phi))\partial_t\phi\mathrm{d}x
+\int_{\Omega}\text{div}(\rho\mathbf{u})\mu\partial_t\phi\mathrm{d}x\\
&=\int_{\Omega}|\nabla\partial_t\phi|^2\mathrm{d}x
+\int_{\Omega}\partial_t\rho\Psi'(\phi)\partial_t\phi\mathrm{d}x\\
&\quad+\int_{\Omega}\rho\Psi''(\phi)|\partial_t\phi|^2\mathrm{d}x
-\int_{\Omega}\rho\mathbf{u}\cdot\nabla(\mu\partial_t\phi)\mathrm{d}x\\
&=\int_{\Omega}|\nabla\partial_t\phi|^2\mathrm{d}x
-\int_{\Omega}\text{div}(\rho\mathbf{u})\Psi'(\phi)\partial_t\phi\mathrm{d}x\\
&\quad+\int_{\Omega}\rho\Psi''(\phi)|\partial_t\phi|^2\mathrm{d}x
-\int_{\Omega}\rho\partial_t\phi\mathbf{u}\cdot\nabla\mu\mathrm{d}x
-\int_{\Omega}\rho\mu\mathbf{u}\cdot\nabla\partial_t\phi\mathrm{d}x\\
&=\int_{\Omega}|\nabla\partial_t\phi|^2\mathrm{d}x
+\int_{\Omega}\rho\Psi''(\phi)\partial_t\phi\mathbf{u}\cdot
\nabla\phi\mathrm{d}x\\
&\quad+\int_{\Omega}\rho\Psi'(\phi)\mathbf{u}
\cdot\nabla\partial_t\phi\mathrm{d}x+\int_{\Omega}\rho\Psi''(\phi)
|\partial_t\phi|^2\mathrm{d}x\\
&\quad-\int_{\Omega}\rho\partial_t\phi\mathbf{u}\cdot\nabla\mu\mathrm{d}x - \int_{\Omega}\rho\mu\mathbf{u}\cdot\nabla\partial_t\phi\mathrm{d}x.
\end{split}
\end{equation}
Besides, using once again \eqref{1.1}$_{1}$, we have
\begin{equation}
\begin{split}\label{4.15}
\int_{\Omega}\rho\mathbf{u}\cdot\nabla\phi\partial_t\mu\mathrm{d}x
=&\frac{\mathrm{d}}{\mathrm{d}t}\Big(\int_{\Omega}\rho\mu\mathbf{u}
\cdot\nabla\phi\mathrm{d}x\Big)-\int_{\Omega}\partial_t\rho\mu
\mathbf{u}\cdot\nabla\phi\mathrm{d}x \\
&-\int_{\Omega}\rho\mu\partial_t\mathbf{u}\cdot\nabla\phi\mathrm{d}x
-\int_{\Omega}\rho\mu\mathbf{u}\cdot\nabla\partial_t\phi\mathrm{d}x \\
=&\frac{\mathrm{d}}{\mathrm{d}t}\Big(\int_{\Omega}\rho\mu
\mathbf{u}\cdot\nabla\phi\mathrm{d}x\Big)
+\int_{\Omega}\text{div}(\rho\mathbf{u})\mu
\mathbf{u}\cdot\nabla\phi\mathrm{d}x \\
&-\int_{\Omega}\rho\mu\partial_t\mathbf{u}
\cdot\nabla\phi\mathrm{d}x-\int_{\Omega}\rho\mu
\mathbf{u}\cdot\nabla\partial_t\phi\mathrm{d}x \\
=&\frac{\mathrm{d}}{\mathrm{d}t}\Big(\int_{\Omega}
\rho\mu\mathbf{u}\cdot\nabla\phi\mathrm{d}x\Big)
-\int_{\Omega}\rho(\mathbf{u}\cdot\nabla\mu)
(\mathbf{u}\cdot\nabla\phi)\mathrm{d}x \\
&-\int_{\Omega}\rho\mu\mathbf{u}\cdot\nabla
(\mathbf{u}\cdot\nabla\phi)\mathrm{d}x
-\int_{\Omega}\rho\mu\partial_t\mathbf{u}\cdot\nabla\phi\mathrm{d}x \\
&-\int_{\Omega}\rho\mu\mathbf{u}\cdot\nabla\partial_t\phi\mathrm{d}x.
\end{split}
\end{equation}
Substituting \eqref{4.14} and \eqref{4.15} into \eqref{4.13}, we conclude that
\begin{equation}
\begin{split}\label{4.16}
\frac{\mathrm{d}}{\mathrm{d}t}&\Big(\frac{1}{2}\|\nabla\mu\|_{L^2(\Omega)}^2 + \int_{\Omega}\rho\mu\mathbf{u}\cdot\nabla\phi\mathrm{d}x\Big) + \int_{\Omega}|\nabla\partial_t\phi|^2\mathrm{d}x
 + \int_{\Omega}\rho\Psi''(\phi)|\partial_t\phi|^2\mathrm{d}x \\&= - \int_{\Omega}\rho\Psi''(\phi)\partial_t\phi\mathbf{u}\cdot
\nabla\phi\mathrm{d}x -\int_{\Omega}\rho\Psi'(\phi)\mathbf{u}
\cdot\nabla\partial_t\phi\mathrm{d}x + \int_{\Omega}\rho\partial_t\phi\mathbf{u}\cdot\nabla\mu\mathrm{d}x \\
&\quad +\int_{\Omega}\rho(\mathbf{u}\cdot\nabla\mu)
(\mathbf{u}\cdot\nabla\phi)\mathrm{d}x + \int_{\Omega}\rho\mu\mathbf{u}
\cdot\nabla(\mathbf{u}\cdot\nabla\phi)\mathrm{d}x \\
&\quad + \int_{\Omega}\rho\mu\partial_t\mathbf{u}\cdot\nabla\phi\mathrm{d}x + 2\int_{\Omega}\rho\mu\mathbf{u}\cdot\nabla\partial_t\phi\mathrm{d}x.
\end{split}
\end{equation}
Taking the $L^2$-scalar product of \eqref{1.1}$_{4}$ with $\partial_t\phi$, we have
\begin{equation}\label{4.17}
\int_{\Omega}\rho|\partial\phi|^2\mathrm{d}x = - \int_{\Omega}\rho\mathbf{u}\cdot\nabla\phi\partial_t\phi\mathrm{d}x - \int_{\Omega}\nabla\mu\cdot\nabla\partial_t\phi\mathrm{d}x.
\end{equation}
Thanks to $\Psi''(s) = 3s^2 - 1\geq - 1$, multiplying \eqref{4.17} by 2 and adding it to \eqref{4.16}, we get
\begin{equation}
\begin{split}\label{4.18}
&\frac{\mathrm{d}}{\mathrm{d}t}\Big(\frac{1}{2}\|\nabla\mu\|_{L^2(\Omega)}^2 + \int_{\Omega}\rho\mu\mathbf{u}\cdot\nabla\phi\mathrm{d}x\Big) + \alpha_0\|\partial_t\phi\|_{H^1(\Omega)}^2 \\
&\leq - \int_{\Omega}\rho\Psi''(\phi)\partial_t\phi
\mathbf{u}\cdot\nabla\phi\mathrm{d}x - \int_{\Omega}\rho\Psi'(\phi)\mathbf{u}\cdot\nabla\partial_t\phi\mathrm{d}x \\
&\quad + \int_{\Omega}\rho\partial_t\phi\mathbf{u}\cdot\nabla\mu\mathrm{d}x +\int_{\Omega}\rho(\mathbf{u}\cdot\nabla\mu)
(\boldsymbol{u}\cdot\nabla\phi)\mathrm{d}x \\
&\quad + \int_{\Omega}\rho\mu\mathbf{u}\cdot\nabla(\mathbf{u}\cdot\nabla\phi)\mathrm{d}x + \int_{\Omega}\rho\mu\partial_t\mathbf{u}\cdot\nabla\phi\mathrm{d}x \\
&\quad + 2\int_{\Omega}\rho\mu\mathbf{u}\cdot\nabla\partial_t\phi\mathrm{d}x - 2\int_{\Omega}\rho\mathbf{u}\cdot\nabla\phi\partial_t\phi\mathrm{d}x - 2\int_{\Omega}\nabla\mu\cdot\nabla\partial_t\phi\mathrm{d}x ,
\end{split}
\end{equation}
where $\alpha_0 = \min\{1,\rho_*\}>0$. Summing \eqref{4.11}, \eqref{4.12} and \eqref{4.16}, we arrive at
\begin{equation}
\begin{split}\label{4.19}
&\frac{\mathrm{d}}{\mathrm{d}t}\Big(\int_{\Omega}\frac{\nu(\phi)}{2}
|D\mathbf{u}|^2\mathrm{d}x+\int_{\Omega}\rho\mu\mathbf{u}\cdot
\nabla\phi\mathrm{d}x+\frac{1}{2}\|\nabla\mathbf{b}\|^2\mathrm{d}x
+\frac{1}{2}\|\nabla\mu\|_{L^2(\Omega)}^2\Big)\\
&\quad+\rho_*\|\partial_t\mathbf{u}\|_{L^2(\Omega)}^2
+\alpha_0\|\partial_t\phi\|_{H^1(\Omega)}^2 +
\alpha_0\|\Delta\mathbf{b}\|_{L^2(\Omega)}^2\\
&\leq \int_{\Omega}-\rho(\mathbf{u}\cdot\nabla)\mathbf{u}
\cdot\partial_t\mathbf{u}\mathrm{d}x
+2\int_{\Omega}\rho\mu\nabla\phi\cdot\partial_t\mathbf{u}
\mathrm{d}x\\
&\quad+\int_{\Omega}-\rho\nabla\Psi(\phi)\cdot\partial_t
\mathbf{u}\mathrm{d}x+\int_{\Omega}\nu'(\phi)
\partial_t\phi|\mathbb{D}\mathbf{u}|^2\mathrm{d}x \\
&\quad+\int_{\Omega}-\rho\Psi''(\phi)\partial_t\phi
\mathbf{u}\cdot\nabla\phi\mathrm{d}x+\int_{\Omega}
-\rho\Psi'(\phi)\mathbf{u}\cdot\nabla\partial_t\phi\mathrm{d}x \\
&\quad+\int_{\Omega}\rho\partial_t\phi
\mathbf{u}\cdot\nabla\mu\mathrm{d}x+\int_{\Omega}\rho
(\mathbf{u}\cdot\nabla\mu)(\mathbf{u}\cdot\nabla\phi)\mathrm{d}x \\
&\quad+\int_{\Omega}\rho\mu\mathbf{u}
\cdot\nabla(\mathbf{u}\cdot\nabla\phi)\mathrm{d}x+2\int_{\Omega}\rho\mu\mathbf{u}\cdot\nabla\partial_t\phi\mathrm{d}x\\
&\quad+2\int_{\Omega}-\rho\mathbf{u}\cdot\nabla\phi\partial_t
\phi\mathrm{d}x+2\int_{\Omega}-\nabla\mu
\cdot\nabla\partial_t\phi\mathrm{d}x \\
&+\int_{\Omega}(\mathbf{b}\cdot\nabla)\mathbf{b}\cdot\partial_t
\mathbf{u}\mathrm{d}x
+\int_{\Omega}(\mathbf{u}\cdot\nabla)\mathbf{b}\cdot\Delta
\mathbf{b}\mathrm{d}x
-\int_{\Omega}(\mathbf{b}\cdot\nabla)\mathbf{u}\cdot
\Delta\mathbf{b}\mathrm{d}x\\&\triangleq\sum_{i=1}^{15} \text{II}_i.
\end{split}
\end{equation}
In what follows, we  bound term by term  from the right hand side of \eqref{4.19}. By \eqref{eq:3.4}, \eqref{eq:3.8}, \eqref{4.12} and \eqref{4.24AA}, we have
\begin{equation}\label{4.25}
\begin{split}
|\text{II}_1|&\leq \rho^*\|\mathbf{u}\|_{L^6(\Omega)}\|\nabla\mathbf{u}\|_{L^3(\Omega)}
\|\partial_t\mathbf{u}\|_{L^2(\Omega)} \\
&\leq C\|\nabla\mathbf{u}\|_{L^2(\Omega)}^{\frac{3}{2}}\|\Delta
\mathbf{u}\|_{L^2(\Omega)}^{\frac{1}{2}}\|\partial_t\mathbf{u}\|_{L^2(\Omega)}\\
&\leq \frac{\rho_*}{8}\|\partial_t\mathbf{u}\|_{L^2(\Omega)}^2
+\varepsilon\|\Delta\mathbf{u}\|_{L^2(\Omega)}^2 + C\|\nabla\mathbf{u}\|_{L^2(\Omega)}^6,
\end{split}
\end{equation}
\begin{equation}\label{4.26}
\begin{split}
|\text{II}_2|&\leq 2\rho^*\|\mu\|_{L^6(\Omega)}\|\nabla\phi\|_{L^3(\Omega)}
\|\partial_t\mathbf{u}\|_{L^2(\Omega)} \\
&\leq C\|\mu\|_{H^1(\Omega)}\|\nabla\phi\|_{L^2(\Omega)}^{\frac{1}{2}}
\|\phi\|_{H^2(\Omega)}^{\frac{1}{2}}\|\partial_t\mathbf{u}\|_{L^2(\Omega)}\\
&\leq \frac{\rho_*}{8}\|\partial_t\mathbf{u}\|_{L^2(\Omega)}^2 + C\|\phi\|_{H^2(\Omega)}(1 + \|\nabla\mu\|_{L^2(\Omega)}) \\
&\leq \frac{\rho_*}{8}\|\partial_t\mathbf{u}\|_{L^2(\Omega)}^2 + C(1 + \|\nabla\mu\|_{L^2(\Omega)}^{\frac{5}{2}}),
\end{split}
\end{equation}
\begin{equation}\label{4.27}
\begin{split}
|\text{II}_3|&\leq \rho^*\|\Psi_0'(\phi)\|_{L^{\infty}(\Omega)}\|\nabla\phi\|_{L^2(\Omega)}
\|\partial_t\mathbf{u}\|_{L^2(\Omega)} \\
&\leq \frac{\rho_*}{8}\|\partial_t\mathbf{u}\|_{L^2(\Omega)}^2 + C(\|\phi\|_{L^{\infty}(\Omega)}^2 + \|\phi\|_{L^{\infty}(\Omega)}^6) \\
&\leq \frac{\rho_*}{8}\|\partial_t\mathbf{u}\|_{L^2(\Omega)}^2 + C(\|\phi\|_{H^1(\Omega)}\|\phi\|_{H^2(\Omega)} + \|\phi\|_{H^1(\Omega)}^3\|\phi\|_{H^2(\Omega)}^3)\\
&\leq \frac{\rho_*}{8}\|\partial_t\mathbf{u}\|_{L^2(\Omega)}^2 + C(\|\phi\|_{H^2(\Omega)} + \|\phi\|_{H^2(\Omega)}^3) \\
&\leq \frac{\rho_*}{8}\|\partial_t\mathbf{u}\|_{L^2(\Omega)}^2 + C(1 + \|\nabla\mu\|_{L^2(\Omega)}^{\frac{3}{2}}) ,
\end{split}
\end{equation}
\begin{equation}\label{4.28}
\begin{split}
|\text{II}_4|&\leq C\|\partial_t\phi\|_{L^6(\Omega)}\|\nabla\mathbf{u}\|_{L^3(\Omega)}
\|\nabla\mathbf{u}\|_{L^2(\Omega)} \\
&\leq \frac{\alpha_0}{14}\|\partial_t\phi\|_{H^1(\Omega)}^2 + C\|\nabla\mathbf{u}\|_{L^2(\Omega)}^3\|\Delta\mathbf{u}\|_{L^2(\Omega)} \\
&\leq \frac{\alpha_0}{14}\|\partial_t\phi\|_{H^1(\Omega)}^2 + \varepsilon\|\Delta\mathbf{u}\|_{L^2(\Omega)}^2 + C\|\nabla\mathbf{u}\|_{L^2(\Omega)}^6,
\end{split}
\end{equation}
\begin{equation}\label{4.29}
\begin{split}
|\text{II}_5|&\leq \rho^*\|\Psi''(\phi)\|_{L^3(\Omega)}\|\partial_t\phi\|_{L^6(\Omega)}
\|\mathbf{u}\|_{L^{\infty}(\Omega)}\|\nabla\phi\|_{L^2(\Omega)} \\
&\leq C\|\partial_t\phi\|_{H^1(\Omega)}
\|\nabla\mathbf{u}\|_{L^2(\Omega)}^{\frac{3}{2}}
\|\Delta\mathbf{u}\|_{L^2(\Omega)}^{\frac{1}{2}} \\
&\leq \frac{\alpha_0}{14}\|\partial_t\phi\|_{H^1(\Omega)}^2 + \varepsilon\|\Delta\mathbf{u}\|_{L^2(\Omega)}^2 + C\|\nabla\mathbf{u}\|_{L^2(\Omega)}^2 ,
\end{split}
\end{equation}
\begin{equation}\label{4.30}
\begin{split}
|\text{II}_6|&\leq \rho^*\|\Psi'(\phi)\|_{L^2(\Omega)}\|\mathbf{u}\|_{L^{\infty}(\Omega)}
\|\partial_t\phi\|_{H^1(\Omega)}\\
&\leq \frac{\alpha_0}{14}\|\partial_t\phi\|_{H^1(\Omega)}^2 + C\|\nabla\mathbf{u}\|_{L^2(\Omega)}\|\mathbb{A}\mathbf{u}\|_{L^2(\Omega)} \\
&\leq \frac{\alpha_0}{14}\|\partial_t\phi\|_{H^1(\Omega)}^2 + \varepsilon\|\Delta\mathbf{u}\|_{L^2(\Omega)}^2 + C\|\nabla\mathbf{u}\|_{L^2(\Omega)}^2 ,
\end{split}
\end{equation}
\begin{equation}\label{4.31}
\begin{split}
|\text{II}_7|&\leq \rho^*\|\partial_t\phi\|_{L^6(\Omega)}\|\mathbf{u}\|_{L^3(\Omega)}
\|\nabla\mu\|_{L^2(\Omega)} \\
&\leq \frac{\alpha_0}{14}\|\partial_t\phi\|_{H^1(\Omega)}^2 + C\|\nabla\mathbf{u}\|_{L^2(\Omega)}^2\|\nabla\mu\|_{L^2(\Omega)}^2 ,
\end{split}
\end{equation}
\begin{equation}\label{4.32}
\begin{split}
|\text{II}_8|&\leq \rho^*\|\mathbf{u}\cdot\nabla\mu\|_{L^2(\Omega)}\|\mathbf{u}\cdot
\nabla\phi\|_{L^2(\Omega)} \\
&\leq C\|\mathbf{u}\|_{L^{\infty}(\Omega)}^2\|\nabla\mu\|_{L^2(\Omega)}
\|\nabla\phi\|_{L^2(\Omega)}\\
&\leq C\|\nabla\mathbf{u}\|_{L^2(\Omega)}^2\|\nabla\mu\|_{L^2(\Omega)}(1 + \|\nabla\mu\|_{L^2(\Omega)}^{\frac{1}{2}}) \\
&\leq C(\|\nabla\mathbf{u}\|_{L^2(\Omega)}^2\|\nabla\mu\|_{L^2(\Omega)} + \|\nabla\mathbf{u}\|_{L^2(\Omega)}^2\|\nabla\mu\|_{L^2(\Omega)}^{\frac{3}{2}}) ,
\end{split}
\end{equation}
\begin{equation}\label{4.33}
\begin{split}
|\text{II}_9|&\leq \rho^*\|\mu\mathbf{u}\cdot\nabla(\mathbf{u}\cdot\nabla\phi)\|_{L^1(\Omega)} \\
&\leq C\|\mu\|_{L^6(\Omega)}\|\mathbf{u}\|_{L^6(\Omega)}(\|\nabla\mathbf{u}\|_{L^2(\Omega)}
\|\phi\|_{L^6(\Omega)}+\|\mathbf{u}\|_{L^2(\Omega)}
\|\nabla\phi\|_{L^6(\Omega)}) \\
&\leq C(1+\|\nabla\mu\|_{L^2(\Omega)})\|\nabla\mathbf{u}\|_{L^2(\Omega)}^2
\|\phi\|_{H^2(\Omega)} \\
&\leq C(1+\|\nabla\mu\|_{L^2(\Omega)}^{\frac{3}{2}}) \|\nabla\mathbf{u}\|_{L^2(\Omega)}^2 ,
\end{split}
\end{equation}
\begin{equation}\label{4.34}
\begin{split}
|\text{II}_{10}|&\leq 2\rho^*\|\mu\mathbf{u}\cdot\nabla\partial_t\phi\|_{L^1(\Omega)} \\
&\leq 2C\|\mu\|_{L^6(\Omega)}\|\mathbf{u}\|_{L^3(\Omega)}
\|\nabla\partial_t\phi\|_{L^2(\Omega)} \\
&\leq 2C(1+\|\nabla\mu\|_{L^2(\Omega)})\|\nabla\mathbf{u}\|_{L^2(\Omega)}
\|\partial_t\phi\|_{H^1(\Omega)}\\
&\leq \frac{\alpha_0}{14}\|\partial_t\phi\|_{H^1(\Omega)}^2 + C\|\nabla\mathbf{u}\|_{L^2(\Omega)}^2(1+\|\nabla\mu\|_{L^2(\Omega)}^2) ,
\end{split}
\end{equation}
\begin{equation}\label{4.35}
\begin{split}
|\text{II}_{11}|&\leq 2\rho^*\|\mathbf{u}\cdot\nabla\phi\partial_t\phi\|_{L^1(\Omega)} \\
&\leq 2C\|\mathbf{u}\|_{L^3(\Omega)}\|\nabla\phi\|_{L^2(\Omega)}
\|\partial_t\phi\|_{L^6(\Omega)} \\
&\leq 2C(E_0)\|\nabla\mathbf{u}\|_{L^2(\Omega)}
\|\nabla\phi\|_{L^2(\Omega)}
\|\partial_t\phi\|_{H^1(\Omega)} \\
&\leq \frac{\alpha_0}{14}\|\partial_t\phi\|_{H^1(\Omega)}^2 + C\|\nabla\mathbf{u}\|_{L^2(\Omega)}^2,
\end{split}
\end{equation}
\begin{equation}\label{4.36}
\begin{split}
|\text{II}_{12}|&\leq 2\|\nabla\mu\cdot\nabla\partial_t\phi\|_{L^1(\Omega)} \\
&\leq 2\|\nabla\mu\|_{L^2(\Omega)}\|\nabla\partial_t\phi\|_{L^2(\Omega)} \\
&\leq \frac{\alpha_0}{14}\|\partial_t\phi\|_{H^1(\Omega)}^2 + C\|\nabla\mu\|_{L^2(\Omega)}^2 ,
\end{split}
\end{equation}
\begin{equation}\label{4.37}
\begin{split}
|\text{II}_{13}|&\leq \|\mathbf{b}\|_{L^6(\Omega)}\|\nabla\mathbf{b}\|_{L^3(\Omega)}
\|\partial_t\mathbf{u}\|_{L^2(\Omega)} \\
&\leq C\|\nabla\mathbf{b}\|_{L^2(\Omega)}^{\frac{3}{2}}
\|\Delta\mathbf{b}\|_{L^2(\Omega)}^{\frac{1}{2}}
\|\partial_t\mathbf{u}\|_{L^2(\Omega)}\\
&\leq \frac{\rho_*}{8}\|\partial_t\mathbf{u}\|_{L^2(\Omega)}^2
+\varepsilon\|\Delta\mathbf{b}\|_{L^2(\Omega)}^2 + C\|\nabla\mathbf{b}\|_{L^2(\Omega)}^6,
\end{split}
\end{equation}
\begin{equation}\label{4.38}
\begin{split}
|\text{II}_{14}|&\leq \|\mathbf{u}\|_{L^6(\Omega)}\|\nabla\mathbf{b}\|_{L^3(\Omega)}
\|\Delta\mathbf{b}\|_{L^2(\Omega)} \\
&\leq
\|\nabla\mathbf{u}\|_{L^2(\Omega)}\|\nabla\mathbf{b}\|_{L^2(\Omega)}^{\frac{1}{2}}
\|\Delta\mathbf{b}\|_{L^2(\Omega)}^{\frac{3}{2}} \\
&\leq
C\|\nabla\mathbf{u}\|_{L^2(\Omega)}^4+C\|\nabla\mathbf{b}\|_{L^2(\Omega)}^2
+\frac{\alpha_0}{4}\|\Delta\mathbf{b}\|_{L^2(\Omega)}^3,
\end{split}
\end{equation}
\begin{equation}\label{4.39}
\begin{split}
|\text{II}_{15}|&\leq \|\mathbf{b}\|_{L^6(\Omega)}\|\nabla\mathbf{u}\|_{L^3(\Omega)}
\|\Delta\mathbf{b}\|_{L^2(\Omega)} \\
&\leq
\|\nabla\mathbf{b}\|_{L^2(\Omega)}\|\nabla\mathbf{u}\|_{L^2(\Omega)}^{\frac{1}{2}}
\|\mathbb{A}\mathbf{u}\|_{L^2(\Omega)}^{\frac{1}{2}}
\|\Delta\mathbf{b}\|_{L^2(\Omega)} \\
&\leq
C\|\nabla\mathbf{b}\|_{L^2(\Omega)}^4\|\nabla\mathbf{u}\|_{L^2(\Omega)}^2
+\varepsilon\|\Delta\mathbf{u}\|_{L^2(\Omega)}^2
+\frac{\alpha_0}{4}\|\Delta\mathbf{b}\|_{L^2(\Omega)} ^{2} \\
&\leq C\|\nabla\mathbf{b}\|_{L^2(\Omega)}^6+C\|\nabla\mathbf{u}\|_{L^2(\Omega)}^6
+\varepsilon\|\Delta\mathbf{u}\|_{L^2(\Omega)}^2
+\frac{\alpha_0}{4}\|\Delta\mathbf{b}\|_{L^2(\Omega)} ^{2}.
\end{split}
\end{equation}
 Collecting \eqref{4.25}-\eqref{4.39} all together, we obtain  the following differential inequality
\begin{equation}\label{4.40}
\begin{split}
\frac{\mathrm{d}}{\mathrm{d}t}&\Big(\int_{\Omega}\frac{\nu(\phi)}{2}
|D\mathbf{u}|^2\mathrm{d}x+\int_{\Omega}\rho\mu\mathbf{u}\cdot
\nabla\phi\mathrm{d}x+\frac{1}{2}\|\nabla\mathbf{b}\|^2\mathrm{d}x
+\frac{1}{2}\|\nabla\mu\|_{L^2(\Omega)}^2\Big)\\
&\quad+\rho_*\|\partial_t\mathbf{u}\|_{L^2(\Omega)}^2
+\alpha_0\|\partial_t\phi\|_{H^1(\Omega)}^2 +
\alpha_0\|\Delta\mathbf{b}\|_{L^2(\Omega)}^2\\
&\leq \varepsilon\|\Delta\mathbf{u}\|_{L^2(\Omega)}^2 + C\|\nabla\mathbf{u}\|_{L^2(\Omega)}^6 +C\|\nabla\mathbf{b}\|_{L^2(\Omega)}^6 \\
&\quad + C(1 + \|\nabla\mu\|_{L^2(\Omega)})\big(\|\nabla\mathbf{u}\|_{L^2(\Omega)}^2 + \|\nabla\mu\|_{L^2(\Omega)}^2\big).
\end{split}
\end{equation}
In order to bound the second derivatives of $\mathbf{u}$, let us take the  $L^2$-scalar product  of  system  \eqref{1.1}$_{2}$  with $-\Delta \mathbf{u}$
  and obtain
\begin{equation}\label{4.41}
\begin{split}
\int_{\Omega}\text{div}(\nu(\phi)D\mathbf{u})\cdot\Delta
\mathbf{u}\mathrm{d}x&=(\rho\partial_t\mathbf{u},\Delta\mathbf{u})+
(\rho(\mathbf{u}\cdot\nabla)\mathbf{u},\Delta\mathbf{u}) \\
&\quad - ((\mathbf{b}\cdot\nabla)\mathbf{b},\Delta\mathbf{u})-(\rho\mu\nabla\phi,\Delta\mathbf{u})+(\rho\nabla\Psi(\phi),\Delta
\mathbf{u}).
\end{split}
\end{equation}
Then, we rewrite \eqref{4.41} as follows
\begin{equation}\label{4.43}
\begin{split}
\int_{\Omega}\frac{\nu(\phi)}{2}|\Delta\mathbf{u}|^2\mathrm{d}x
&=(\rho\partial_t\mathbf{u},\Delta\mathbf{u})
+(\rho(\mathbf{u}\cdot\nabla)\mathbf{u},\Delta
\mathbf{u}) \\
&\quad-(\rho\mu\nabla\phi,\Delta\mathbf{u})
+(\rho\nabla\Psi(\phi),\Delta\mathbf{u}) +(\nu'(\phi)\mathbb{D}\mathbf{u}\nabla\phi,\Delta\mathbf{u})
\\
&\quad -((\mathbf{b}\cdot\nabla)\mathbf{b},\Delta\mathbf{u})
\\&\triangleq\sum_{i=1}^{6} \tilde{\text{II}}_i.
\end{split}
\end{equation}
It follows from \eqref{chazhi}, H\"{o}lder's and Young's inequalities, that
\begin{equation*}\label{4.44}
\begin{split}
| \tilde{\text{II}}_1|&\leq \rho^*\|\partial_t\mathbf{u}\|_{L^2(\Omega)}\|\Delta\mathbf{u}\|_{L^2(\Omega)}
\\&\leq\eta
\|\Delta\mathbf{u}\|_{L^2(\Omega)}^2+C_1
\|\partial_t\mathbf{u}\|_{L^2(\Omega)}^2,
\end{split}
\end{equation*}
\begin{equation*}\label{4.45}
\begin{split}
| \tilde{\text{II}}_2|&\leq \rho^*\|\mathbf{u}\|_{L^{\infty}(\Omega)}\|\nabla\mathbf{u}\|_{L^2(\Omega)}
\|\Delta\mathbf{u}\|_{L^2(\Omega)} \\
&\leq C_1\|\nabla\mathbf{u}\|_{L^2(\Omega)}^{\frac{3}{2}}
\|\Delta\mathbf{u}\|_{L^2(\Omega)}^{\frac{3}{2}}\\
&\leq
\eta\|\Delta\mathbf{u}\|_{L^2(\Omega)}^2
+C_1\|\nabla\mathbf{u}\|_{L^2(\Omega)}^6,
\end{split}
\end{equation*}
\begin{equation*}\label{4.46}
\begin{split}
| \tilde{\text{II}}_3|&\leq \rho^*\|\mu\|_{L^6(\Omega)}\|\nabla\phi\|_{L^3(\Omega)}
\|\Delta\mathbf{u}\|_{L^2(\Omega)} \\
&\leq\eta\|\Delta\mathbf{u}\|_{L^2(\Omega)}^2
+C_1\|\mu\|_{H^1(\Omega)}^2\|\phi\|_{H^2(\Omega)}^2,
\end{split}
\end{equation*}
\begin{equation*}\label{4.47}
\begin{split}
| \tilde{\text{II}}_4|&\leq \rho^*\|\Psi_0'(\phi)\|_{L^{\infty}(\Omega)}
\|\nabla\phi\|_{L^2(\Omega)}\|\Delta\mathbf{u}\|_{L^2(\Omega)} \\
&\leq C_1(1 + \|\phi\|_{L^{\infty}(\Omega)}^3)\|\Delta\mathbf{u}\|_{L^2(\Omega)} \\
&\leq C_1(1 + \|\phi\|_{H^1(\Omega)}^{\frac{3}{2}}\|\phi\|_{H^2(\Omega)}^{\frac{3}{2}})
\|\Delta\mathbf{u}\|_{L^2(\Omega)} \\
&\leq \eta\|\Delta\mathbf{u}\|_{L^2(\Omega)}^2 + C_1(1 + \|\phi\|_{H^2(\Omega)}^3),
\end{split}
\end{equation*}
\begin{equation*}\label{4.48}
\begin{split}
| \tilde{\text{II}}_5|&\leq C\|\nabla\mathbf{u}\|_{L^3(\Omega)}\|\nabla\phi\|_{L^6(\Omega)}\|\mathbb{A}
\mathbf{u}\|_{L^2(\Omega)} \\
&\leq C_1\|\nabla\mathbf{u}\|_{L^2(\Omega)}^{\frac{1}{2}}
\|\Delta\mathbf{u}\|_{L^2(\Omega)}^{\frac{3}{2}}\|\phi\|_{H^2(\Omega)} \\
&\leq \eta\|\Delta\mathbf{u}\|_{L^2(\Omega)}^2 + C_1\|\nabla\mathbf{u}\|_{L^2(\Omega)}^2\|\phi\|_{H^2(\Omega)}^4,
\end{split}
\end{equation*}
\begin{equation*}\label{4.50}
\begin{split}
| \tilde{\text{II}}_6|&\leq \|\mathbf{b}\|_{L^{\infty}(\Omega)}\|\nabla\mathbf{b}\|_{L^2(\Omega)}
\|\Delta\mathbf{u}\|_{L^2(\Omega)} \\
&\leq C_1\|\nabla\mathbf{b}\|_{L^2(\Omega)}^{\frac{3}{2}}
\|\Delta\mathbf{b}\|_{L^2(\Omega)}^{\frac{1}{2}}
\|\Delta\mathbf{u}\|_{L^2(\Omega)}
\\&\leq C_1\|\nabla\mathbf{b}\|_{L^2(\Omega)}^{6}+\eta
\|\Delta\mathbf{b}\|_{L^2(\Omega)}^{2}
+\eta\|\Delta\mathbf{u}\|_{L^2(\Omega)}^2.
\end{split}
\end{equation*}
Then inserting the above estimates into \eqref{4.43}, we have
\begin{equation}\label{4.43AA}
\begin{split}
\int_{\Omega}\frac{\nu(\phi)}{2}|\Delta\mathbf{u}|^2\mathrm{d}x
&\leq C_1\|\partial_t\mathbf{u}\|_{L^2(\Omega)}^2+C_1\|\nabla\mathbf{u}\|_{L^2(\Omega)}^6+C_1\|\mu\|_{H^1(\Omega)}^2\|\phi\|_{H^2(\Omega)}^2
 +C_1(1 + \|\phi\|_{H^2(\Omega)}^3)\\&\quad+ C_1\|\nabla\mathbf{u}\|_{L^2(\Omega)}^2\|\phi\|_{H^2(\Omega)}^4+C_1\|\nabla\mathbf{b}\|_{L^2(\Omega)}^{6}+\eta
\|(\Delta\mathbf{b},\Delta\mathbf{b})\|_{L^2(\Omega)}^{2}.
\end{split}
\end{equation}
At last, we deal with  the gradient of the pressure. Taking the divergence of system  \eqref{1.1}$_{2}$  gives rise to
\begin{equation*}\Delta P=\Dv\big(-\rho\partial_t\mathbf{u}-\rho(\mathbf{u}\cdot\nabla)\mathbf{u}+(\mathbf{b}\cdot\nabla)\mathbf{b}+\nu'(\phi)D\mathbf{u}\nabla\phi+\rho\mu\nabla\phi-\rho\nabla\Psi(\phi)\big),
\end{equation*} and consequently the pressure $p$ may be recovered by
\begin{equation*} \nabla P =\nabla\Delta^{-1}\Dv\big(-\rho\partial_t\mathbf{u}-\rho(\mathbf{u}\cdot\nabla)\mathbf{u}+(\mathbf{b}\cdot\nabla)\mathbf{b}+\nu'(\phi)D\mathbf{u}\nabla\phi+\rho\mu\nabla\phi-\rho\nabla\Psi(\phi)\big),
\end{equation*}
which together with  the bounded of Riesz's  operator yields that
\begin{equation}\begin{split}\label{3.6-2} \|\nabla P\|_{L^2(\Omega)}^2&\leq C_2\|\partial_t\mathbf{u}\|_{L^2(\Omega)}^2+C_1\|\nabla\mathbf{u}\|_{L^2(\Omega)}^6+C_2\|\mu\|_{H^1(\Omega)}^2\|\phi\|_{H^2(\Omega)}^2
 \\&\quad+C_2(1 + \|\phi\|_{H^2(\Omega)}^3)+ C_2\|\nabla\mathbf{u}\|_{L^2(\Omega)}^2\|\phi\|_{H^2(\Omega)}^4+C\|\nabla\mathbf{b}\|_{L^2(\Omega)}^{6}+\eta
\|\Delta\mathbf{b}\|_{L^2(\Omega)}^{2}.
\end{split}\end{equation}
Denoting
\begin{equation*}
\begin{split}\label{4.20}
H(\rho,\mathbf{u},\mathbf{b},\phi,\mu)&
=\int_{\Omega}\frac{\nu(\phi)}{2}|D\mathbf{u}|^2\mathrm{d}x
+\frac{1}{2}\|\nabla\mathbf{b}\|^2\mathrm{d}x+\frac{1}{2}\|\nabla\mu\|_{L^2(\Omega)}^2
\\
&\quad+\int_{\Omega}\rho\mu\mathbf{u}\cdot\nabla\phi\mathrm{d}x,
\end{split}
\end{equation*}
\begin{align*}
F=F(\mathbf{u},\mathbf{b},\phi,P)=\|\Delta \mathbf{u}\|_{L^2(\Omega)}^2+\|\Delta \mathbf{b}\|_{L^2(\Omega)}^2+\|\partial_t \mathbf{u}\|_{L^2(\Omega)}^2 +\|\partial_t\phi\|_{H^1(\Omega)}^2+\|\nabla P\|_{L^2(\Omega)}^2.
\end{align*}
Thanks to \eqref{eq:3.4} and \eqref{4.24AA}, we observe that
\begin{equation}
\begin{split}\label{4.21}
\Big|\int_{\Omega}\rho\mu\mathbf{u}\cdot\nabla\phi\mathrm{d}x \Big|&\leq \rho^*\|\mu\|_{L^6(\Omega)}\|\mathbf{u}\|_{L^3(\Omega)}\|\nabla\phi\|_{L^2(\Omega)}\\
&\leq C(\rho^*)(1 + \|\nabla\mu\|_{L^2(\Omega)})\|\mathbf{u}\|_{L^2(\Omega)}^{\frac{1}{2}}
\|\nabla\mathbf{u}\|_{L^2(\Omega)}^{\frac{1}{2}} \\
&\leq \int_{\Omega}\frac{\nu(\phi)}{4}|D\mathbf{u}|^2\mathrm{d}x
+\frac{1}{4}\|\nabla\mu\|_{L^2(\Omega)}^2 +C(\rho^*),
\end{split}
\end{equation}
which implies that
\begin{equation*}
\begin{split}\label{4.20}
H(\rho,\mathbf{u},\mathbf{b},\phi,\mu)&
\approx\int_{\Omega}\frac{\nu(\phi)}{2}|D\mathbf{u}|^2\mathrm{d}x
+\frac{1}{2}\|\nabla\mathbf{b}\|^2\mathrm{d}x+\frac{1}{2}\|\nabla\mu\|_{L^2(\Omega)}^2+C(\rho^*).
\end{split}
\end{equation*}
Adding \eqref{4.40}, \eqref{4.43AA}$\times \frac{\rho^*}{8C_1}$ and  \eqref{3.6-2}$\times \frac{\rho^*}{8C_2}$, and then
taking  small enough $\varepsilon,\eta$,
we finally conclude that
\begin{equation}\label{4.58}
\frac{\mathrm{d}}{\mathrm{d}t}H + F \leq C(1+H^3).
\end{equation}
Define $X(t)=1+H + \int_{0}^{t}F(\tau)d\tau.$ Hence, whenever $T$ satisfies  $CX^{2}(0)t<\frac{1}{2}$ for  $t\in [0,T]$, we obtain
\begin{equation*}\label{3.10-1-A} X^{2}(t)\leq\frac{X^{2}(0)}{1-2CX^{2}(0)t}< \infty\quad \text{for all }\  t\in [0,T].\end{equation*}


\subsubsection*{Step 4 Approximated regular solutions.}
As in \cite{AR1}, we use energy methods and Galerkin approximations. We consider the family of
eigenfunctions $\{ w_j\}_{j\geq 1}$ of the operator $A_1=-\Delta+I$ with homogeneous Neumann boundary condition,    and the family of eigenfunctions
$\{\boldsymbol{w}_j\}_{j\geq 1}$ of the Stokes operator $\boldsymbol{A}$  with homogeneous with Dirichlet boundary condition.
For any integer $m\geq 1$, we define the finite-dimensional subspaces of $V$ and  $\boldsymbol{V}_\sigma$ respectively, by
$ V_m=\text{span}\{ w_1,...,w_m\} $ and  $\boldsymbol{V}_m= \text{span}\{ \boldsymbol{w}_1,...,\boldsymbol{w}_m\}$.
We denote by $\Pi_m$ and $\mathbb{P}_m$ the orthogonal projections on $V_m$ and $\boldsymbol{V}_m$
with respect to the inner product in $H$ and $\boldsymbol{H}_\sigma$, respectively. 
For any $m\in \mathbb{N}$, we consider $(\rho_{m0},\boldsymbol{u}_{m0}, \boldsymbol{b}_{m0}, \phi_{m0})$, where
$\mathbf{u}_m(\cdot,0)=\mathbb{P}_m\mathbf{u}_0$,
$\mathbf{b}_m(\cdot,0)=\mathbb{P}_m\mathbf{b}_0$,
$\phi_m(\cdot,0)=\Pi_m\phi_0$ satisfy
\begin{equation}\label{4.12A}
\boldsymbol{u}_{m0}\to \mathbf{u}_0,\, \boldsymbol{b}_{m0}\to \mathbf{b}_0, \, \text{strongly in}\,  L^{2}(\Omega), \phi_{m0}\to \phi_{0}, \, \text{strongly in}\,  H^{1}(\Omega), \, \text{as}\,  m\to\infty.
\end{equation}
Then, we consider the following  approximate solutions  through the semi-Galerkin method
 \begin{equation*}
\rho_m\in\mathcal{C}^1(\overline{Q_T}),\, \mathbf{u}_m\in\mathcal{C}^1([0,T];\mathbf{V}_m),\,
\mathbf{b}_m\in\mathcal{C}^1([0,T];\mathbf{V}_m),\,
\phi_m\in\mathcal{C}^1([0,T];V_m),\, \mu_m\in\mathcal{C}([0,T];V_m)
\end{equation*}
with $Q_T = \Omega\times(0,T)$. Moreover, recalling that $\mathbf{u}_m(\cdot,0)\to\mathbf{u}_0$ in $\mathbf{V}_{\sigma}$,
$\mathbf{b}_m(\cdot,0)\to\mathbf{b}_0$ in $\mathbf{V}_{\sigma}$
and $\phi_m(\cdot,0)\to\phi_0^k$ in $H^3(\Omega)$,   it then follows from Step1-Step4, that
\begin{equation}\label{4.8}
\rho_*\leq\rho_m(x,t)\leq\rho^*, \quad \forall (x,t)\in\overline{Q_T},
\end{equation}
\begin{equation}\label{4.9}
\|\mathbf{u}_m\|_{L^{\infty}(0,T;\mathbf{H}_{\sigma})}\leq C, \|\mathbf{u}_m\|_{L^2(0,T;\mathbf{V}_{\sigma})}\leq C, \\
\|\mathbf{b}_m\|_{L^{\infty}(0,T;\mathbf{H}_{\sigma})}\leq C,  \|\mathbf{b}_m\|_{L^2(0,T;\mathbf{V}_{\sigma})}\leq C,
\end{equation}
and
\begin{equation}\label{4.10}
\|\phi_m\|_{L^{\infty}(0,T;H^1(\Omega))}\leq C, \quad \|\phi_m\|_{L^2(0,T;H^2(\Omega))}\leq C, \quad \|\mu_m\|_{L^2(0,T;H^1(\Omega))}\leq C,
\end{equation}
where $C$ is independent of $m$. Based on  the Aubin-Lions lemma,  we can  prove  that $(\rho_m,\mathbf{u}_m,\phi_m,\mu_m)$ (up to
a subsequence) converges to the limit $(\rho,\mathbf{u}, \phi,\mu)^\top$ satisfying the initial--boundary value problem \eqref{1.1}-\eqref{1.2MHD} as stated in Theorem \ref{1.2} as $m\to\infty$.

\subsection{The two  dimensional case}
This part is devoted to  proving the global existence of strong solution for the initial-boundary value problem
	\eqref{1.1}-\eqref{1.2MHD}   with arbitrary large data. Compared with the three dimensional case, the only difference consists in the estimates of the terms $\text{II}_1$, $\text{II}_4$, $\text{II}_{13}$, $\text{II}_{14}$, $\text{II}_{15}$ and $ \tilde{\text{II}}_2$, $ \tilde{\text{II}}_5$, $ \tilde{\text{II}}_6$. By exploiting \eqref{eq:3.3}, \eqref{eq:3.7}, \eqref{4.24AA} and  H\"{o}lder's and Young's inequalities,  we deduce that
\begin{equation}\label{4.59}
\begin{split}
|\text{II}_1|&\leq \rho^*\|\mathbf{u}\|_{L^4(\Omega)}\|\nabla \mathbf{u}\|_{L^4(\Omega)}\|\partial_t \mathbf{u}\|_{L^2(\Omega)}\\
&\leq C_{2}\|\nabla \mathbf{u}\|_{L^2(\Omega)}\|\Delta \mathbf{u}\|_{L^2(\Omega)}^{\frac{1}{2}}\|\partial_t \mathbf{u}\|_{L^2(\Omega)}\\
&\leq \frac{\rho_*}{8}\|\partial_t \mathbf{u}\|_{L^2(\Omega)}^2+\varepsilon\|\Delta \mathbf{u}\|_{L^2(\Omega)}^2+C_{2}\|\nabla \mathbf{u}\|_{L^2(\Omega)}^4,
\end{split}
\end{equation}
\begin{equation}\label{4.60}
\begin{split}
|\text{II}_4|&\leq C_{2}\|\partial_t\phi\|_{H^1(\Omega)}\|\nabla \mathbf{u}\|_{L^2(\Omega)}^2\ln^{\frac{1}{2}}\Big(C_{2}\frac{\|\Delta \mathbf{u}\|_{L^2(\Omega)}}{\|\nabla \mathbf{u}\|_{L^2(\Omega)}}\Big)\\
&\leq \frac{\alpha_0}{14}\|\partial_t\phi\|_{H^1(\Omega)}^2+C_{2}\|\nabla \mathbf{u}\|_{L^2(\Omega)}^4\ln\Big(C_{2}\frac{\|\Delta \mathbf{u}\|_{L^2(\Omega)}}{\|\nabla \mathbf{u}\|_{L^2(\Omega)}}\Big),
\end{split}
\end{equation}
\begin{equation}\label{4.61}
\begin{split}
|\text{II}_{13}|&\leq \rho^*\|\mathbf{b}\|_{L^4(\Omega)}\|\nabla\mathbf{b}\|_{L^4(\Omega)}
\|\partial_t\mathbf{u}\|_{L^2(\Omega)} \\
&\leq C_{2}\|\nabla\mathbf{b}\|_{L^2(\Omega)}
\|\Delta\mathbf{b}\|_{L^2(\Omega)}^{\frac{1}{2}}
\|\partial_t\mathbf{u}\|_{L^2(\Omega)}\\
&\leq \frac{\rho_*}{8}\|\partial_t\mathbf{u}\|_{L^2(\Omega)}^2
+\frac{\alpha_0}{6}\|\Delta\mathbf{b}\|_{L^2(\Omega)}^2 + C_{2}\|\nabla\mathbf{b}\|_{L^2(\Omega)}^4,
\end{split}
\end{equation}
\begin{equation}\label{4.62}
\begin{split}
|\text{II}_{14}|&\leq \rho^*\|\mathbf{u}\|_{L^4(\Omega)}\|\nabla\mathbf{b}\|_{L^4(\Omega)}
\Delta\mathbf{b}\|_{L^2(\Omega)} \\
&\leq
\rho^*\|\nabla\mathbf{u}\|_{L^2(\Omega)}^{\frac{1}{2}}
\|\nabla\mathbf{b}\|_{L^2(\Omega)}^{\frac{1}{2}}
\|\Delta\mathbf{b}\|_{L^2(\Omega)}^{\frac{1}{2}}
\|\Delta\mathbf{b}\|_{L^2(\Omega)} \\
&\leq
C_{2}\|\nabla\mathbf{u}\|_{L^2(\Omega)}^2\|\nabla\mathbf{b}\|_{L^2(\Omega)}^2
+\frac{\alpha_0}{6}\|\Delta\mathbf{b}\|_{L^2(\Omega)}^2 ,
\end{split}
\end{equation}
\begin{equation}\label{4.63}
\begin{split}
|\text{II}_{15}|&\leq \|\mathbf{b}\|_{L^4(\Omega)}
\|\nabla\mathbf{u}\|_{L^4(\Omega)}\|\Delta\mathbf{b}\|_{L^2(\Omega)} \\
&\leq
\|\nabla\mathbf{b}\|_{L^2(\Omega)}^{\frac{1}{2}}
\|\nabla\mathbf{u}\|_{L^2(\Omega)}^{\frac{1}{2}}
\|\Delta\mathbf{u}\|_{L^2(\Omega)}^{\frac{1}{2}}
\|\Delta\mathbf{b}\|_{L^2(\Omega)}  \\
&\leq
C_{2}\|\nabla\mathbf{b}\|_{L^2(\Omega)}^2\|\nabla\mathbf{u}\|_{L^2(\Omega)}^2
+\varepsilon\|\Delta\mathbf{u}\|_{L^2(\Omega)}^2
+\frac{\alpha_0}{6}\|\Delta\mathbf{b}\|_{L^2(\Omega)} ^{2}.
\end{split}
\end{equation}
Collecting \eqref{4.26}-\eqref{4.27}, \eqref{4.29}-\eqref{4.36}, and \eqref{4.59}-\eqref{4.63} all together, and using \eqref{4.24AA}, we infer that
\begin{equation}\label{4.64}
\begin{split}
\frac{\mathrm{d}}{\mathrm{d}t}H&+\frac{\rho_*}{2}\|\partial_t \mathbf{u}\|_{L^2(\Omega)}^2+\frac{\alpha_0}{2}\|\partial_t\phi\|_{H^1(\Omega)}^2
+\frac{\alpha_0}{2}\|\Delta\mathbf{b}\|_{L^2(\Omega)}^2\\
&\leq \varepsilon\|\Delta \mathbf{u}\|_{L^2(\Omega)}^2+C_2\|\nabla \mathbf{u}\|_{L^2(\Omega)}^4\ln\Big(C_2\frac{\|\mathbf{A} \mathbf{u}\|_{L^2(\Omega)}}{\|\nabla \mathbf{u}\|_{L^2(\Omega)}}\Big )
\\&\quad+C_2\big(\|\nabla \mathbf{u}\|_{L^2(\Omega)}^2+\|\nabla \mathbf{b}\|_{L^2(\Omega)}^2+\|\nabla\mu\|_{L^2(\Omega)}^2\big)H.
\end{split}
\end{equation}
In what follows,  we estimate $ \tilde{\text{II}}_2$, $\ \tilde{\text{II}}_5$ and $ \tilde{\text{II}}_6$ of \eqref{4.43}. It follows from  \eqref{eq:3.3}, \eqref{4.9} and \eqref{4.10},  H\"{o}lder's and Young's inequalities, that
\begin{equation}\label{4.65}
\begin{split}
| \tilde{\text{II}}_2|&\leq \rho^*\|\mathbf{u}\|_{L^4(\Omega)}\|\nabla \mathbf{u}\|_{L^4(\Omega)}\|\Delta \mathbf{u}\|_{L^2(\Omega)}\\
&\leq C_{2}\|\nabla \mathbf{u}\|_{L^2(\Omega)}\|\Delta \mathbf{u}\|_{L^2(\Omega)}^{\frac{3}{2}}\\&
\leq \eta\|\Delta \mathbf{u}\|_{L^2(\Omega)}^2+C_{2}\|\nabla \mathbf{u}\|_{L^2(\Omega)}^4,
\end{split}
\end{equation}
\begin{equation}\label{4.66}
\begin{split}
| \tilde{\text{II}}_5|&\leq C\|\nabla \mathbf{u}\|_{L^4(\Omega)}\|\nabla\phi\|_{L^4(\Omega)}\|\Delta \mathbf{u}\|_{L^2(\Omega)}\\
&\leq C_{2}\|\nabla \mathbf{u}\|_{L^2(\Omega)}^{\frac{1}{2}}\|\mathbf{A} \mathbf{u}\|_{L^2(\Omega)}^{\frac{3}{2}}\|\nabla\phi\|_{L^2(\Omega)}^{\frac{1}{2}}
\|\phi\|_{H^2(\Omega)}^{\frac{1}{2}}\\
&\leq \eta\|\Delta \mathbf{u}\|_{L^2(\Omega)}^2+C_{2}\|\nabla \mathbf{u}\|_{L^2(\Omega)}^2\|\phi\|_{H^2(\Omega)}^2,
\end{split}
\end{equation}
\begin{equation}\label{4.68}
\begin{split}
| \tilde{\text{II}}_6|&\leq \|\mathbf{b}\|_{L^4(\Omega)}\|\nabla \mathbf{b}\|_{L^4(\Omega)}\|\Delta \mathbf{u}\|_{L^2(\Omega)}\\
&\leq C_{2}\|\nabla \mathbf{b}\|_{L^2(\Omega)}\|\Delta \mathbf{b}\|_{L^2(\Omega)}^{\frac{1}{2}}\|\Delta \mathbf{u}\|_{L^2(\Omega)}\\
&\leq \eta\|\Delta \mathbf{u}\|_{L^2(\Omega)}^2+\eta\|\Delta \mathbf{b}\|_{L^2(\Omega)}^2+C_{2}\|\nabla \mathbf{b}\|_{L^2(\Omega)}^4.
\end{split}
\end{equation}
Therefore,  combining with \eqref{4.64},  \eqref{4.65}-\eqref{4.68}   \eqref{4.43AA} and  \eqref{3.6-2},
 we are lead to
\begin{equation}\label{4.72}
\begin{split}
\frac{\mathrm{d}}{\mathrm{d}t}H+F
&\leq C_2(\|\nabla \mathbf{u}\|_{L^2(\Omega)}^2 +\|\nabla \mathbf{b}\|_{L^2(\Omega)}^2 +\|\nabla \mu\|_{L^2(\Omega)}^2)H + C_2\|\nabla \mathbf{u}\|_{L^2(\Omega)}^4 \ln\Big(C_2\frac{\| \Delta\mathbf{u}\|_{L^2(\Omega)}}{\|\nabla \mathbf{u}\|_{L^2(\Omega)}}\Big)\\&\quad+\eta\|\Delta \mathbf{u}\|_{L^2(\Omega)}^2+\varepsilon\|\Delta \mathbf{u}\|_{L^2(\Omega)}^2+\eta\|\Delta \mathbf{b}\|_{L^2(\Omega)}^2.
\end{split}
\end{equation}
Moreover,  it follows from  the basic inequality
\begin{align*}
&x\ln(Cy)\leq y + x\ln(Cx)\quad \forall x,y > 0,
\end{align*}
with $x = \|\nabla\mathbf{u}\|_{L^2(\Omega)}^2$, $y = \frac{\|\Delta\mathbf{u}\|_{L^2(\Omega)}}
{\|\nabla\mathbf{u}\|_{L^2(\Omega)}}$,  that
\begin{align*}
C_2\|\nabla\mathbf{u}\|_{L^2(\Omega)}^4
\ln\Big(C_2\frac{\|\Delta\mathbf{u}\|_{L^2(\Omega)}}
{\|\nabla\mathbf{u}\|_{L^2(\Omega)}}\Big)
&\leq C_2\|\nabla\mathbf{u}\|_{L^2(\Omega)}\|\Delta\mathbf{u}\|_{L^2(\Omega)}+ C_2\|\nabla\mathbf{u}\|_{L^2(\Omega)}^4
\ln\Big(C_2\|\nabla\mathbf{u}\|_{L^2(\Omega)}^2\Big)\\
&\leq \varepsilon\|\Delta\mathbf{u}\|_{L^2(\Omega)}^2 + C_2\|\nabla\mathbf{u}\|_{L^2(\Omega)}^2+ C_2\|\nabla\mathbf{u}\|_{L^2(\Omega)}^4\ln\Big(C_2(e + \|\nabla\mathbf{u}\|_{L^2(\Omega)}^2)\Big).
\end{align*}
Then taking  small enough $\varepsilon,\eta$, we eventually arrive at
\begin{equation}\label{4.73}
\frac{\mathrm{d}}{\mathrm{d}t}H + F\leq C_2(\|\nabla\mathbf{u}\|_{L^2(\Omega)}^2 +\|\nabla\mathbf{b}\|_{L^2(\Omega)}^2 + \|\nabla\mu\|_{L^2(\Omega)}^2)(e+H)\ln\big(e+H\big),
\end{equation}
which together with Gronwall's inequality and \eqref{3.7AAA} yields
\begin{equation*}\label{4.73}
H + \int_{0}^{t}F(\tau)d\tau \leq C \, \quad \text{for all }\,   t\geq0.
\end{equation*}

\section{The proof of uniqueness}
The last section is devoted   to presenting the proof to the uniqueness part in  both
Theorems \ref{1.2} and \ref{1.3}, which mainly rely on the weighted energy estimates,  the shift of integrability method and  Lagrangian
 coordinates.
\subsection{Weighted energy estimates }
\label{extra}
In this section, our main objective is to transfer integrability from the time variable to the spatial variables.
To achieve this, we derive time-weighted estimates, such as
\begin{align*}
&(\sqrt{\rho t}\,\mathbf{u}_t,\sqrt{t}\mathbf{b}_t, \ \sqrt{\rho t}\,\phi_t, \ \sqrt{t}\,\nabla \phi_t) \in L^{\infty}([0,T]; L^2)\\
&(\sqrt{t}\,\nabla \mathbf{u}_t,\nabla \mathbf{b}_t, \ \sqrt{t}\,\nabla \mu_t) \in L^2([0,T]; L^2),
\quad \text{and} \quad
\sqrt{t\rho}\,\mu_t \in L^2([0,T]; L^2).
\end{align*}
These results are established through the following two lemmas, corresponding to the two and three dimensional cases, respectively.
\begin{lemma}\label{lemma3.2}
Assume \(d = 3\), and that a strong solution $(\rho,\mathbf{u},\mathbf{b},P,\phi,\mu)$ to the  initial--boundary value problem \eqref{1.1}-\eqref{1.2MHD}. Then for all \(0\leq t\leq T\), we have
\begin{align}\label{3D}
\|\sqrt{\rho t}\mathbf{u}_t\|_2^2&+\|\sqrt{t}\mathbf{b}_t\|_2^2+\|\sqrt{\rho t}\phi_t\phi\|_2^2+\|\sqrt{ t}\nabla\phi_t\|_2^2+\|\sqrt{\rho t}\phi_t\|_2^2
+\int_{0}^{t}\nu_*\tau\|\nabla \mathbf{u}_t\|_2^2d\tau\\ \notag
&+\int_{0}^{t}\tau\|\nabla \mathbf{b}_t\|_2^2d\tau+\int_{0}^{t}\tau\|\nabla\mu_t\|_2^2d\tau
+\int_{0}^{t}\tau\rho\|\mu_t\|_2^2d\tau
\leq\exp\big(\int_{0}^{t}h_1(\tau)d\tau\big)-1,
\end{align}
where
\begin{align*}
&h_1(t)=
(1 + C + C_{\rho^*} + C_{\rho^*\rho_*T} + C_T)
\Big(\big( \|\phi\|_{12}^3 + \|\phi\|_4 \big)^2
\big( \|\nabla^2 \phi\|_4^2 + \|\nabla \phi\|_4^2 + \|\mathbf{u}\|_4^2 \big) \\
&+ \big( \|\nabla \phi\|_6 + \|\phi\|_\infty^2 \|\nabla\phi\|_6 \big)^2
\big( \|\mathbf{u}\|_6^2 + \|\nabla\phi\|_4^2  + \|\mathbf{u}\|_4^2 \big) + \|\mathbf{u}\|_4^2 \big( \|\nabla^2 \phi\|_6 + \|\nabla \phi\|_{12}^2 \|\phi\|_\infty + \|\nabla^2 \phi\|_6 \|\phi\|_\infty^2 \big)^2 \\
&+ \|\nabla \mathbf{u}\|_2^2 + \| \mu\|_4^2 \|\nabla \phi\|_4^2 + \|\nabla \phi\|_2^2 + \|\phi\|_\infty^4 + \|\nabla \mu\|_2^2 + \|\nabla \phi\|_\infty^2
 + \|\nabla^2 \phi\|_2^2 + \|\nabla \phi\|_4^2 + \|\phi\|_8^4 \|\nabla\phi\|_4^2 \\
&+ \|\nabla \phi\|_4^8 + \big( \|\phi\|_\infty^2 + 1 \big)^2 \|\nabla \phi\|_4^2 + \|\mathbf{u}\|_4^4 + \|\mathbf{u}\|_6^4 + \|\phi\|_6^4 + \|\nabla \mathbf{u}\|_2^2  + \|\mathbf{u}\|_4 \|\nabla \phi\|_4 \|\phi\|_6 + \|\nabla \mu\|_2^2 + \|\nabla \phi\|_4^2\\& + \|\phi\|_\infty^4
+ \|\nabla \phi\|_4^2 + \|\mathbf{u}\|_6^2 + \|\mathbf{u}\|_2^2 + \|\nabla \mathbf{u}\|_2^2 + \|\nabla \mathbf{u}\|_2^{\frac{7}{2}} \\
&+ \|\mathbf{u}\|_4^2 \big( \|\nabla \phi\|_4^2 + \|\nabla^2 \phi\|_4^2 + \|\mu\|_4^2 \big)
+ \|\nabla^2 \phi\|_4^2 \|\mu\|_4^2 + \|\nabla \phi\|_8^2 \|\phi\|_8^2
+ \big( \|\phi\|_\infty^3 + \|\phi\|_\infty\big)^2 \Bigr)\\
&\times\Big( \|\mathbf{u}\|_{\infty}^2 +\|\mathbf{b}\|_{\infty}^2 + \|\mathbf{u}\|_{\infty}^4 + \|\mu\|_{\infty}^2 + \|\phi_t\|_2^2 + \| \phi_t\|_4^2 + \| \phi_t\|_6^2 + \|\phi_t\|_2 + \|\mathbf{u}_t\|_2^2 + \|\nabla \phi_t\|_2^2 \\
&+  \|\nabla^2 \mathbf{u}\|_2^2 + \|\nabla \mathbf{u}\|_3^2 +\|\nabla \mathbf{b}\|_3^2 + \|\nabla \mu\|_4^2 + \|\nabla \mu\|_6^2 + \|\nabla^2 \mu\|_6^2 + \|\nabla^2 \mu\|_4^2 + \|\nabla \mu\|_{w^{1,6}}^2 \Big)
\end{align*}
is $ L^1_{loc}(\mathbb{R}^+) $ only depends on, \( \rho^* \), \( \rho_* \), $T$ and the initial value.
\end{lemma}
\begin{proof}
Differentiating \eqref{1.1}$_2$ with respect to $t$, respectively, multiplying by $\sqrt t$, and taking the inner product with $\sqrt t \mathbf{u}_t$, we have
\begin{align*}
\frac{1}{2}\frac{d}{dt}&\int_\Omega\rho t|\mathbf{u}_t|^{2}dx-\int_\Omega\frac{d}{dt}\bigl(\Dv(\nu(\phi)D\mathbf{u}\bigr)\sqrt t\cdot\sqrt t \mathbf{u}_tdx+\int_\Omega\nabla P_t\cdot t\mathbf{u}_tdx\\
=&-\int_\Omega\frac{d}{dt}\bigl(\Dv\nabla\phi\otimes\nabla\phi\bigr)\sqrt t\cdot \sqrt t \mathbf{u}_tdx+\frac{1}{2}\int_\Omega\rho |\mathbf{u}_t|^{2}dx-\frac{1}{2}\int t\rho_t|\mathbf{u}_t|^{2}dx-\int_\Omega\sqrt t \mathbf{b}_t\cdot\nabla \mathbf{b}\cdot \sqrt t \mathbf{u}_tdx\\
&-\int_\Omega\sqrt t \mathbf{b}\cdot\nabla \mathbf{b}_t\cdot \sqrt t \mathbf{u}_tdx-\int_\Omega\sqrt t\rho_t \mathbf{u}\cdot\nabla \mathbf{u}\cdot\sqrt t \mathbf{u}_tdx-\int_\Omega\sqrt\rho\mathbf{u}_t\cdot\nabla \mathbf{u}\cdot \sqrt t \mathbf{u}_tdx-\int_\Omega\sqrt t \rho \mathbf{u}\cdot\nabla \mathbf{u}_t\cdot\sqrt t \mathbf{u}_tdx.
\end{align*}
On the other hand,
\begin{align*}
-&\int_\Omega\frac{d}{dt}\Dv\bigl(\nu(\phi)D\mathbf{u}\bigr)\sqrt t\cdot\sqrt t \mathbf{u}_tdx
\ \\
&=\int_\Omega Dv\bigl(\nu'(\phi)\phi_tDu\bigr)\cdot t \mathbf{u}_tdx+\int_\Omega \Dv\bigl(\nu(\phi )D\mathbf{u}_t\bigr)\cdot t \mathbf{u}_tdx\\
&=\int_\Omega \nu'(\phi)\phi_tD\mathbf{u}\cdot \nabla(t \mathbf{u}_t)dx+\int_\Omega \nu(\phi )D\mathbf{u}_t\cdot \nabla (t \mathbf{u}_t)dx,
\end{align*}
which together with  \eqref{1.1}$_{5}$,  yields that
\begin{align*}
-&\int_\Omega\frac{d}{dt}\big(\Dv\nabla\phi\otimes\nabla\phi\big)\sqrt t\cdot \sqrt t \mathbf{u}_tdx\\&=\int_\Omega\frac{d}{dt}\bigl(\rho\mu\cdot\nabla\phi-\rho\nabla\Phi(\phi)
-\nabla(\frac{1}{2}|\nabla\phi|^2)\bigr)\sqrt t\cdot\sqrt t \mathbf{u}_tdx\\
&=\int_\Omega\rho_t\mu\cdot\nabla\phi\cdot t \mathbf{u}_tdx+\int_\Omega t\rho\mu_t\cdot\nabla\phi\cdot \mathbf{u}_tdx+
\int_\Omega t\rho\mu\cdot\nabla \phi_t\cdot\mathbf{u}_tdx\\
&\quad-\int_\Omega t\rho_t\Phi^{'}(\phi)\cdot\nabla\phi \cdot\mathbf{u}_tdx
- \int_\Omega t \rho{\Phi}''(\phi) \phi_t\cdot \nabla \phi\cdot \mathbf{u}_tdx- \int_{\Omega} t \rho \Phi'(\phi)\cdot \nabla \phi_t\cdot \mathbf{u}_tdx.
\end{align*}
Observing that
\begin{align*}
-&\int_{\Omega} \frac{d}{dt} \big( \nabla  \frac{1}{2} |\nabla \phi|^2\big)\cdot  t \mathbf{u}_tdx
\\&= -\int_{\Omega} \frac{d}{dt} \big( \nabla  \frac{1}{2} |\nabla \phi|^2 t \cdot \mathbf{u}_t\big) dx
+\int_{\Omega}\nabla( \frac{1}{2} |\nabla \phi|^2) \cdot\mathbf{u}_tdx
+\int_{\Omega} \nabla( \frac{1}{2} |\nabla \phi|^2 ) \cdot t\mathbf{u}_{tt}dx \\
&= -\frac{d}{dt}  \int_{\Omega}\nabla  (\frac{1}{2} |\nabla \phi|^2 )\cdot t \mathbf{u}_tdx
- \int_{\Omega}  ( \frac{1}{2} |\nabla \phi|^2 )\cdot\text{div} \, \mathbf{u}_tdx
- \int_{\Omega}   \frac{1}{2} |\nabla \phi|^2 \cdot\text{div} \, t \mathbf{u}_{tt}dx  \\
&=-\frac{d}{dt} \int_{\Omega}  \frac{1}{2} |\nabla \phi|^2 \cdot\text{div} \, t \mathbf{u}_t dx= 0,
\end{align*}
and employing  Korn's inequality, we conclude that
\begin{equation}
\begin{split}
&\frac{1}{2} \frac{d}{dt} \int_{\Omega} t\rho |\mathbf{u}_t|^2 dx + \frac{1}{2} \int_{\Omega} \nu(\phi) t|\nabla \mathbf{u}_t|^2 dx\label{eq-31} \\
&\leq \frac{1}{2} \int_{\Omega} \rho |\mathbf{u}_t|^2dx - \frac{1}{2} \int_{\Omega} t\rho_t \, |\mathbf{u}_t|^2dx - \int_{\Omega} \left( \sqrt{t} \rho_t \mathbf{u} \cdot \nabla \mathbf{u} \right)\cdot \left( \sqrt{t} \mathbf{u}_t \right) dx- \int_{\Omega} \left( \sqrt{t} \rho \mathbf{u}_t \cdot \nabla \mathbf{u} \right) \cdot\left( \sqrt{t} \mathbf{u}_t \right) dx \\
&\quad- \int_{\Omega} \sqrt{t} \rho \mathbf{u} \cdot \nabla \mathbf{u}_t \cdot\left( \sqrt{t} \mathbf{u}_t \right) dx - \int_{\Omega}\nu'(\phi) \, \phi_t D\mathbf{u} \cdot \nabla t \mathbf{u}_tdx + \int_{\Omega} \left( \sqrt{t} \rho_t \mu \nabla \phi \right) \cdot\left( \sqrt{t} \mathbf{u}_t \right) dx \\
&\quad+ \int_{\Omega} \left( \sqrt{t} \rho \mu_t \nabla \phi \right)\cdot \left( \sqrt{t} \mathbf{u}_t \right) dx + \int_{\Omega} \left( \sqrt{t} \rho \mu \nabla \phi_t \right) \cdot\left( \sqrt{t} \mathbf{u}_t \right) dx - \int_{\Omega} \sqrt{t} \rho_t \Phi'(\phi) \nabla \phi\cdot \left( \sqrt{t} \mathbf{u}_t \right) dx \\
&\quad- \int_{\Omega} \sqrt{t}  \rho \Phi''(\phi) \, \phi_t \nabla \phi\cdot \left( \sqrt{t}\mathbf{u}_t \right) dx - \int_{\Omega} \left( \sqrt{t}  \rho \Phi'(\phi) \nabla \phi_t \right)\cdot \left( \sqrt{t} \mathbf{u}_t \right)dx\\
&\quad-\int_\Omega\left(\sqrt t \mathbf{b}_t\cdot\nabla \mathbf{b} \right)\cdot \left(\sqrt t \mathbf{u}_t \right)dx
-\int_\Omega\left(\sqrt t \mathbf{b}\cdot\nabla \mathbf{b}_t\right)\cdot \left(\sqrt t \mathbf{u}_t\right)dx
\\&\triangleq\sum_{i=1}^{14} I_i .
\end{split}
\end{equation}
Differentiating \eqref{1.1}$_3$ with respect to $t$, respectively, multiplying by $\sqrt t$, and taking the inner product with $\sqrt t \mathbf{b}_t$, yield that
\begin{equation}
\begin{split}\label{eq-43}
\frac{1}{2}&\frac{d}{dt}\int_\Omega t|\mathbf{b}_t|^{2}dx+\int_\Omega t|\nabla\mathbf{b}_t|^{2}dx\\
&\leq\frac{1}{2}\int_\Omega|\mathbf{b}_t|^{2}dx-\int_\Omega \mathbf{u}_t \cdot\nabla \mathbf{b}\cdot t \mathbf{b}_tdx-\int_\Omega \mathbf{u} \cdot\nabla \mathbf{b}_t\cdot t \mathbf{b}_tdx+\int_\Omega \mathbf{b}_t \cdot\nabla \mathbf{u}\cdot t \mathbf{b}_tdx+\int_\Omega \mathbf{b} \cdot\nabla \mathbf{u}_t\cdot t \mathbf{b}_tdx
\\&\triangleq\sum_{i=1}^{5} H_i .
\end{split}
\end{equation}
Differentiating \eqref{1.1}$_4$ with respect to $t$, respectively, multiplying by $\sqrt t$, and taking the inner product with $\sqrt t \mu_t$, yield that
\begin{align*}-\int_{\Omega} |\nabla \sqrt{t} \mu_t|^2 \, dx&=\int_{\Omega} (\sqrt{t} \rho_t \phi_t)(\sqrt{t} \mu_t) \, dx + \int_{\Omega} \rho (\sqrt{t} \phi_t)_t \sqrt{t} \mu_t \, dx - \frac{1}{2} \int_{\Omega} \rho \phi_t \mu_t \, dx \\
&+ \int_{\Omega} (\sqrt{t} \rho_t \mathbf{u}\cdot \nabla\phi)(\sqrt{t} \mu_t) \, dx+ \int_{\Omega} (\sqrt{t} \rho \mathbf{u}_t\cdot \nabla\phi)(\sqrt{t} \mu_t) \, dx \\
&+ \int_{\Omega} (\sqrt{t} \rho \mathbf{u}\cdot \nabla\phi_t)(\sqrt{t} \mu_t) \, dx.
\end{align*}
By  virtue of \eqref{1.1}$_5$,  we obtain
\begin{align*}
\int_{\Omega} \rho (\sqrt{t} \phi_t)_t \sqrt{t} \mu_t \, dx &= \int_{\Omega} (\sqrt{t} \phi_t)_t \sqrt{t} \left(-\Delta \phi_t + \rho_t (\phi^3 - \phi) + 3\rho \phi^2 \phi_t - \rho \phi_t - \rho_t \mu \right) dx \\
&= \frac{1}{2} \frac{d}{dt} \int_{\Omega} |\nabla \sqrt{t} \phi_t|^2 dx + \int_{\Omega} (\sqrt{t} \phi_t)_t \sqrt{t} \rho_t \phi^3 dx - \int_{\Omega} (\sqrt{t} \phi_t)_t \sqrt{t} \rho_t \phi dx \\
&+ 3 \int_{\Omega} (\sqrt{t} \phi_t)_t \sqrt{t} \rho \phi^2 \phi_t \, dx - \int_{\Omega} (\sqrt{t} \phi_t)_t \sqrt{t} \rho \phi_t \, dx - \int_{\Omega} (\sqrt{t} \phi_t)_t \sqrt{t} \rho_t \mu dx.
\end{align*}
Then we have
\begin{equation}\label{eq-32}
\begin{split}
&\int_{\Omega} |\nabla \sqrt{t} \mu_t|^2 dx + \frac{1}{2} \frac{d}{dt} \int_{\Omega} |\nabla \sqrt{t} \phi_t|^2 dx\\& = -\int_{\Omega} \sqrt{t} \rho_t \phi_t \sqrt{t} \mu_t dx - \int_{\Omega} (\sqrt{t} \phi_t)_t \sqrt{t} \rho_t \phi^3 dx + \int_{\Omega} (\sqrt{t} \phi_t)_t \sqrt{t} \rho_t \phi \, dx\\
 &\quad- 3\int_{\Omega} (\sqrt{t} \phi_t)_t \sqrt{t} \rho \phi^2 \phi_t dx + \int_{\Omega}
 (\sqrt{t} \phi_t)_t \sqrt{t} \rho \phi_t dx + \int_{\Omega} (\sqrt{t} \phi_t)_t \sqrt{t} \rho_t \mu dx + \frac{1}{2} \int_{\Omega} \rho\phi_t \mu_t dx \\
&\quad- \int_{\Omega} (\sqrt{t} \rho_t \mathbf{u}\cdot \nabla\phi)(\sqrt{t} \mu_t) dx
- \int_{\Omega} (\sqrt{t} \rho \mathbf{u}_t \cdot\nabla\phi)(\sqrt{t} \mu_t) dx + \int_{\Omega} (\sqrt{t} \rho \mathbf{u}\cdot \nabla\phi_t)(\sqrt{t} \mu_t) \\&\triangleq\sum_{i=1}^{10} K_i.
\end{split}
\end{equation}
Furthermore, differentiating \eqref{1.1}$_5$, \eqref{1.1})$_4$ with respect to $t$, respectively, multiplying by $\sqrt t$, then taking the inner product with $\sqrt t \phi_t$, $\sqrt t \mu_t$ and summing them together,  we conclude that
\begin{equation}\label{eq-33}
\begin{split}
&\frac{1}{2} \frac{d}{dt} \int_{\Omega}  \rho t|\phi_t|^2 dx + \int_{\Omega}t\rho |\mu_t|^2 dx
\\&\leq \frac{1}{2} \int_{\Omega} \rho|\phi_t|^2 dx-\frac{1}{2} \int_{\Omega} t\rho_t|\phi_t|^2 dx - \int_{\Omega} (\sqrt{t} \rho_t \mathbf{u}\cdot \nabla\phi)(\sqrt{t} \phi_t) dx \\
&\quad- \int_{\Omega} (\sqrt{t} \rho \mathbf{u}_t\cdot \nabla\phi)(\sqrt{t} \phi_t) dx - \int_{\Omega} (\sqrt{t} \rho \mathbf{u}\cdot \nabla\phi_t)(\sqrt{t} \phi_t) dx - \int_{\Omega} (\sqrt{t} \rho_t \mu)(\sqrt{t} \mu_t) dx \\
&\quad+ \int_{\Omega} \bigl(\sqrt{t} \rho_t (\phi^3 - \phi)\bigr)(\sqrt{t} \mu_t) dx + \int_{\Omega} \sqrt{t} \rho (3\phi^2 \phi_t - \phi_t)(\sqrt{t} \mu_t)
\\&\triangleq\sum_{i=1}^{8} L_i.
\end{split}
\end{equation}
In what follows, let us bound these terms in the right hand sides of  \eqref{eq-32}-\eqref{eq-33}.  For \( I_2 \), thanks to \( \rho_t = - \mathbf{u} \cdot \nabla \rho \), we have
\begin{equation}
\begin{split}
|I_2 |&\leq C \big| \int_{\Omega} t \mathrm{div}(\rho \mathbf{u}) |\mathbf{u}_t|^2 dx \big|\\
&\leq C \int_{\Omega} t \rho |\mathbf{u}| |\nabla \mathbf{u}_t| |u_t| dx  \\
&\leq C\big( \int_{\Omega} \rho t |\mathbf{u}_t|^2 dx \big)^{\frac{1}{2}} \big( \int_{\Omega} t \rho |\mathbf{u}|^2 |\nabla \mathbf{u}_t|^2 dx \big)^{\frac{1}{2}} \\
&\leq C_\rho*  \| \sqrt{t \rho} \mathbf{u}_t \|_2 \| \mathbf{u} \|_{\infty} \| \sqrt{t} \nabla \mathbf{u}_t \|_2  \\
&\leq \varepsilon \| \sqrt{t} \nabla \mathbf{u}_t \|_2^2 + C_\rho*   \| \mathbf{u} \|_{\infty}^2\| \sqrt{t \rho} \mathbf{u}_t \|_2^2.
\end{split}
\end{equation}
For \( I_3 \), according to \( \rho_t = - \mathbf{u} \cdot \nabla \rho \) and then performing an integration by parts, we get
\begin{equation}\label{3.6666}
\begin{split}
|I_3 |&\leq \big| - \int_{\Omega} (\sqrt{t} \rho_t \mathbf{u} \cdot \nabla \mathbf{u}) \cdot (\sqrt{t} \mathbf{u}_t) dx \big|\\
&\leq \big| - \int_{\Omega} t \rho \mathbf{u} \cdot \nabla \big[ (\mathbf{u} \cdot \nabla) \mathbf{u} \cdot \mathbf{u}_t \big] dx \big| \\
&\leq \int_{\Omega} t \rho |\mathbf{u}| \big( |\nabla \mathbf{u}|^2 |\mathbf{u}_t| + |\mathbf{u}| |\nabla^2 \mathbf{u}| |\mathbf{u}_t| + |\mathbf{u}| |\nabla \mathbf{u}| |\nabla \mathbf{u}_t|)dx  \\
&\triangleq \sum_{i=1}^6 I_{3i}.
\end{split}
\end{equation}
Now we deal with the terms $I_{3i}$$(i=1,2,3,4,5,6)$ in the above inequality \eqref{3.6666}.  For $I_{31}$, it follows from H\"{o}lder's and Young's inequalities and $\dot{H}^1\hookrightarrow L^6(\Omega)$ that
\begin{align*}
I_{31} &\leq \sqrt{\rho^* T} \| \sqrt{\rho t} \mathbf{u}_t \|_4 \| \mathbf{u} \|_6 \| \nabla \mathbf{u} \|_{24/7}^2 \\
&\leq \sqrt{\rho^* T} \| \sqrt{\rho t} \mathbf{u}_t \|_2^{1/4} \| \sqrt{\rho t} \mathbf{u}_t \|_6^{3/4} \| u \|_6 \| \nabla \mathbf{u} \|_{24/7}^2 \\
&\leq \varepsilon \| \nabla \sqrt{t} u_t \|_2^2 + C_{T,\rho^*} \| \sqrt{\rho t} \mathbf{u}_t \|_2^{2/5} \| \nabla \mathbf{u} \|_{24/7}^{16/5} \| \nabla \mathbf{u} \|_2^{8/5}.
\end{align*}
Due to
\[
\| \nabla \mathbf{u} \|_{24/7}^{16/5} \leq C \| \nabla \mathbf{u} \|_2^{6/5} \| \nabla^2 \mathbf{u} \|_2^2,
\]
thus
\begin{align*}
I_{31} &\leq \varepsilon \| \nabla \sqrt{t} \mathbf{u}_t \|_2^2 + C_{T,\rho^*} \| \sqrt{\rho t} \mathbf{u}_t \|_2^{2/5} \| \nabla \mathbf{u} \|_2^{14/5} \| \nabla^2 \mathbf{u} \|_2^2 \\
&\leq \varepsilon \| \nabla \sqrt{t} \mathbf{u}_t \|_2^2 + C_{T,\rho^*} \big( \| \sqrt{\rho t} \mathbf{u}_t \|_2^2 + \| \nabla \mathbf{u}\|_2^{7/2} \big) \| \nabla^2\mathbf{u} \|_2^2.
\end{align*}
Similarly, we also  get
\begin{align*}
I_{32} &= \int_{\Omega} t \rho |\mathbf{u}|^2 | \nabla^2 \mathbf{u} | |\mathbf{u}_t| dx \leq  C _{T \rho^{\ast} } \| \nabla^2\mathbf{u} \|_2^2 + C\| \mathbf{u} \|_{\infty}^4 \| \sqrt{\rho t} \mathbf{u}_t \|_2^2,\\
I_{33} &= \int_{\Omega} t \rho |\mathbf{u}|^2 \|\nabla \mathbf{u}\| \|\nabla \mathbf{u}_t\| dx \\
&\leq \varepsilon \int_{\Omega} |\nabla \sqrt{t} \mathbf{u}_t|^2 dx + C \int_{\Omega} t \rho^2 |\mathbf{u}|^4 \|\nabla \mathbf{u}\|^2 dx \\
&\leq \varepsilon |\nabla \sqrt{t} \|\mathbf{u}_t\|_2^2  + C_{T\rho*} \| \mathbf{u} \|_{\infty}^4 \|\nabla \mathbf{u}\|_2^2.
\end{align*}
For $I_4$ and $I_5$, we conclude from  $\dot{H}^1\hookrightarrow L^6(\Omega)$, that
\begin{equation}
\begin{split}
|I_4| &\leq \big| -\int _{\Omega}\sqrt{t} \rho \mathbf{u}_t \cdot \nabla \mathbf{u} \cdot \sqrt{t} \mathbf{u}_t dx\big|
\\&\leq C_{\rho*}\| \nabla \mathbf{u} \|_3 \| \sqrt{\rho t} \mathbf{u}_t \|_2 \| \sqrt{t} \mathbf{u}_t \|_6\\
&\leq C_{\rho*} \| \nabla \mathbf{u} \|_3 \| \sqrt{\rho t} \mathbf{u}_t \|_2 \| \sqrt{t} \nabla \mathbf{u}_t \|_2
 \\&\leq \varepsilon \| \sqrt{t} \nabla \mathbf{u}_t \|_2^2 + C_{\rho*} \| \nabla \mathbf{u} \|_3^2 \| \sqrt{\rho t} \mathbf{u}_t \|_2^2 ,
 \end{split}
\end{equation}
\begin{equation}
\begin{split}
|I_5| &\leq \big| -\int _{\Omega}\sqrt{t} \rho \mathbf{u} \cdot \nabla \mathbf{u}_t \cdot \sqrt{t} \mathbf{u}_t dx\big|
 \\&\leq C_{\rho*} \| \sqrt{\rho t} \mathbf{u}_t \|_2 \| \mathbf{u} \|_{\infty} \| \sqrt{t} \nabla \mathbf{u}_t \|_2 \\
&\leq \varepsilon \| \sqrt{t} \nabla \mathbf{u}_t \|_2^2 + C_{\rho*} \| \mathbf{u}\|_{\infty}^2 \| \sqrt{\rho t} \mathbf{u}_t \|_2^2 .
\end{split}
\end{equation}
For $I_6$, due to $\nu = \nu(s) \in W^{1,\infty}(\mathbb{R})$, $\dot{H}^1\hookrightarrow L^6(\Omega)$, we get
\begin{equation}
\begin{split}
|I_6| &\leq \big|- \int _{\Omega}\nu'(\phi) \phi_t D \mathbf{u} \cdot \nabla t\mathbf{u}_t dx \big|
\\& \leq C \| \sqrt{t} \nabla \mathbf{u}_t \|_2 \|\sqrt t \phi_t \|_6 \| D \mathbf{u} \|_3\\
&\leq \varepsilon \| \sqrt{t} \nabla \mathbf{u}_t \|_2^2 + C\| \nabla \mathbf{u} \|_3^2 \|\nabla\sqrt t\phi_t \|_2^2.
\end{split}
\end{equation}
For \( I_7\), according to \( \rho_t = - \mathbf{u}\cdot \nabla \rho \) and then performing an integration by parts, we get
\begin{equation}
\begin{split}\label{3.1000}
|I_7| &\leq \big| -\int_{\Omega} (\sqrt{t} \, \text{div} \rho \mathbf{u}) \cdot(\mu \nabla\phi\cdot \sqrt{t} \, \mathbf{u}_t) \, dx \big| \\&
\leq\big| \int_{\Omega} t \, \rho \mathbf{u} \cdot \nabla (\mu \nabla\phi\cdot u_t) \, dx \big| \\
&\leq \big| \int_{\Omega} t \, \rho \mathbf{u} \cdot \nabla \mu \cdot\nabla \phi\cdot \mathbf{u}_t \, dx \big| + \big| \int_{\Omega} t \, \rho \mathbf{u} \,\cdot \mu \nabla^2 \phi \cdot \mathbf{u}_t \, dx \big| + \big| \int_{\Omega} t \, \rho \mathbf{u} \,\cdot \mu \nabla \phi \cdot \nabla \mathbf{u}_t \, dx \big|
\\&\triangleq \sum_{i=1}^3 I_{7i}.
\end{split}
\end{equation}
We deal with  the terms $I_{7i}$$(i=1,2,3)$ in the above inequality \eqref{3.1000}. It follows from H\"{o}lder's and Young's inequalities that
\begin{align*}
I_{71} &\leq C_{T} \rho * \, \|\sqrt{\rho t} \, \mathbf{u}_t\|_2 \, \|\mathbf{u}\|_{\infty} \, \|\nabla \mu\|_4 \, \|\nabla \phi\|_4 \leq  C_{T} \rho *\|\mathbf{u}\|_{\infty}^2 \, \|\sqrt{\rho t} \, \mathbf{u}_t\|_2^2 + C\|\nabla \mu\|_4^2 \, \|\nabla \phi\|_4^2, \\
I_{72} &\leq C_{T} \rho * \, \|\sqrt{\rho t} \, \mathbf{u}_t\|_2 \, \|\mathbf{u}\|_{\infty} \, \|\nabla^2 \phi\|_4 \, \| \mu\|_4 \leq  C_{T} \rho *\|\mathbf{u}\|_{\infty}^2 \, \|\sqrt{\rho t} \, \mathbf{u}_t\|_2^2 + C\|\nabla^2 \phi\|_4^2 \, \| \mu\|_4^2, \\
I_{73} &\leq C_{T} \rho *  \, \|\nabla \sqrt t\mathbf{u}_t\|_2 \, \|\mathbf{u}\|_{\infty} \, \|\mu\|_4 \, \|\nabla \phi\|_4 \leq \varepsilon \|\nabla\, \sqrt t \mathbf{u}_t\|_2^2 + C_{\rho*T} \|\mathbf{u}\|_{\infty}^2 \, \|\mu\|_4^2 \, \|\nabla \phi\|_4^2.
\end{align*}
For \( I_8\) and \( I_9\),  according to H\"{o}lder's and Young's inequalities,  we have
\begin{align}
|I_8| &\leq \big| \int_{\Omega} ( \sqrt{t}\rho \, \mu_t \, \nabla \phi) \cdot\, (\sqrt{t} \, \mathbf{u}_t) \, dx \big|\notag
\\& \leq C_{\rho*} \, \|\sqrt{\rho t} \, \mu_t\|_2 \, \|\nabla \, \phi\|_4 \, \|\sqrt{t}\mathbf{u}_t\|_4\\ \notag
&\leq C_{\rho*T} \, \|\sqrt{\rho t} \, \mu_t\|_2 \, \|\nabla \, \phi\|_4 \, \|\nabla\sqrt{t}\mathbf{u}_t\|_2^\frac{3}{4}\|\mathbf{u}_t\|_2^\frac{1}{4}\notag\\
 &\leq \varepsilon\|\nabla \sqrt{t} \, \mathbf{u}_t\|_2^2 +\varepsilon\|\sqrt{\rho t} \, \mu_t\|_2^2 + C_{\rho*T} \, \|\nabla \phi\|_4^8 \, \|\mathbf{u}_t\|_2^2,\notag\\
|I_9| &\leq \big| \int_{\Omega} (\sqrt{t} \, \rho \, \mu \, \nabla \phi_t) \,\cdot (\sqrt{t} \, \mathbf{u}_t) \, dx \big| \notag\\&\leq C_{\rho*} \, \|\sqrt{\rho t} \, \mathbf{u}_t\|_2 \, \|\nabla \sqrt{t} \, \phi_t\|_2 \, \|\mu\|_{\infty} \\ \notag
&\leq C_{\rho*} \, \|\sqrt{\rho t} \, \mathbf{u}_t\|_2^2 + C\|\mu\|_{\infty}^2 \, \|\nabla \sqrt{t} \, \phi_t\|_2^2 .
\end{align}
For \( I_{10} \), thanks to \( \rho_t = - \mathbf{u} \cdot \nabla \rho \), we deduce that
\begin{equation}
\begin{split}\label{3.1444}
|I_{10}| &\leq \big| \int_{\Omega} \sqrt{t} \, \text{div} \, (\rho \, \mathbf{u}) \cdot\,\bigl((\phi^3 - \phi) \, \nabla \phi \,\cdot \sqrt{t} \, \mathbf{u}_t\bigr) \, dx \big|
\\&\leq\big| \int_{\Omega} \rho \, \mathbf{u} \cdot \nabla \bigl(t(\phi^3 - \phi) \, \cdot \nabla \phi \, \mathbf{u}_t \bigr) \, dx \big|\\
&\leq \big| 3\int_{\Omega} \rho \, \mathbf{u}\cdot \big| \nabla \phi \big|^2|\phi|^2\cdot \, t \, \mathbf{u}_t dx \big| + \big| \int_{\Omega} \rho \, \mathbf{u} \cdot\big| \nabla \phi \big|^2 \,\cdot t \, \mathbf{u}_t dx \big| \\
&\quad + \big| \int_{\Omega} \rho \, \mathbf{u} \,\cdot (\phi^3 - \phi) \, \nabla^2 \phi \, \cdot t \mathbf{u}_t dx \big| + \big| \int_{\Omega} \rho \, \mathbf{u} \, \cdot (\phi^3 - \phi) \, \nabla \phi \, \cdot\nabla t\, \mathbf{u}_t dx \big|
\\&\triangleq \sum_{i = 1}^{4} I_{10(i)}.
\end{split}
\end{equation}
Next, we turn to the estimates of $I_{10(i)}(i=1,2,3,4)$ in the above inequality \eqref{3.1444}. Thanks to H\"{o}lder's and Young's inequalities, we have
\begin{equation*}
\begin{split}
I_{10(1)} &\leq C_{\rho*} \, \|\sqrt{\rho t} \, \mathbf{u}_t\|_2 \, \|\mathbf{u}\|_{\infty} \, \| \left| \nabla \phi \right|^2 \|_4\| \left|\phi \right|^2 \|_4\\& \leq C_{\rho*T} \|\mathbf{u}\|_{\infty}^2 \, \|\sqrt{\rho t} \, \mathbf{u}_t\|_2^2 + C \, \|\nabla \phi\|_8^2\|\phi\|_8^2,\\
I_{10(2)} &\leq C_{\rho*T} \|\sqrt{\rho_t} \mathbf{u}_t\|_{2} \|\mathbf{u}\|_{\infty} \| \nabla\phi \|_{4}^2\\& \leq C_{\rho*T} \|\mathbf{u}\|_{\infty}^2 \|\sqrt{\rho t} \mathbf{u}_t\|_{2}^2 + C \|\nabla \phi \|_{4}^4 ,\\
I_{10(3)} &\leq C_{\rho*T} \|\sqrt{\rho t} \mathbf{u}_t\|_{2} \|\nabla^2 \phi\|_{4} \left( \||\phi|^3\|_{4} + \|\phi\|_{4} \right) \|\mathbf{u}\|_{\infty}\\& \leq C_{\rho*T}\|\mathbf{u}\|_{\infty}^2 \|\sqrt{\rho t} \mathbf{u}_t\|_{2}^2 + C \|\nabla^2 \phi\|_{4}^2 \left( \|\phi\|_{12}^3 + \|\phi\|_{4} \right)^2,\\
I_{10(4)} &\leq C_{\rho*T} \|\nabla\sqrt{t} \mathbf{u}_t\|_{2} \|u\|_{\infty} \left( \||\phi|^3\|_{4} + \|\nabla\phi\|_{4} \right) \|\nabla\phi\|_{4} \\ &\leq \varepsilon\|\nabla\sqrt{t} \mathbf{u}_t\|_{2}^2 + C_{\rho*T} \|\mathbf{u}\|_{\infty}^2 \left( \|\phi\|_{12}^3 + \|\phi\|_{4} \right)^2 \|\nabla\phi\|_{4}^2 .
\end{split}
\end{equation*}
For  $I_{11}$  and $I_{12}$, according to $\Phi_0(s) = \frac{1}{4}(s^2 - 1)^2 \quad \forall s \in \mathbb{R}$, we have
\begin{equation}
\begin{split}
|I_{11}| &\leq \big| -\int_{\Omega} \sqrt{t}\rho(3\phi^2-1) \phi_t \nabla\phi\cdot \sqrt{t} \mathbf{u}_t \, dx \big| \\
&\leq C_{\rho*T} \|\sqrt{\rho t} \mathbf{u}_t\|_{2} (\|\phi\|_{\infty}^2+1) \| \phi_t\|_{4} \|\nabla\phi\|_{4} \\
&\leq  C_{\rho*T}(\|\phi\|_{\infty}^2+1)^2 \|\nabla\phi\|_{4}^2 \| \phi_t\|_{4}^2 + C \|\sqrt{\rho t} \mathbf{u}_t\|_{2}^2 ,\\
|I_{12}| &\leq \big| -\int_{\Omega} \sqrt{t} \rho(\phi^3 - \phi) \nabla\phi_t \cdot\sqrt{t} \mathbf{u}_t \, dx \big| \\&\leq C_{\rho*T} \|\nabla\phi_t\|_{2} \|\sqrt{\rho t} \mathbf{u}_t\|_{2} \big( \|\phi\|_{\infty}^3 + \|\phi\|_{\infty} \big) \\
&\leq C_{\rho*T} \|\nabla\phi_t\|_{2}^2 \|\sqrt{\rho t} \mathbf{u}_t\|_{2}^2 + C\big( \|\phi\|_{\infty}^3 + \|\phi\|_{\infty} \big)^2.
\end{split}
\end{equation}
For $I_{13}$, we conclude from  $\dot{H}^1\hookrightarrow L^6(\Omega)$, that
\begin{equation}
\begin{split}
|I_{13}| &\leq \big|-\int_{\Omega}\sqrt{t} \mathbf{b}_t \cdot \nabla \mathbf{b} \cdot \sqrt{t} \mathbf{u}_t dx\big| \\&\leq \| \nabla \mathbf{b} \|_3 \| \sqrt{t} \mathbf{u}_t \|_6 \| \sqrt{t} \mathbf{b}_t \|_2\\
&\leq \| \nabla \mathbf{b} \|_3 \| \sqrt{t} \mathbf{b}_t \|_2 \| \sqrt{t} \nabla \mathbf{u}_t \|_2
\\& \leq \varepsilon \| \sqrt{t} \nabla \mathbf{u}_t \|_2^2 + C \| \nabla \mathbf{b} \|_3^2 \| \sqrt{t} \mathbf{b}_t \|_2^2.
\end{split}
\end{equation}
For \( I_{14}\),  according to H\"{o}lder's and Young's inequalities,  we have
\begin{equation}
\begin{split}
|I_{14}| &\leq \big|-\int_{\Omega}\sqrt{t} \mathbf{b} \cdot \nabla \mathbf{b}_t \cdot \sqrt{t} \mathbf{u}_t dx\big| \\&\leq  C_{\rho_*}\| \sqrt{\rho t}\mathbf{u}_t \|_2 \| \mathbf{b} \|_\infty \| \nabla\sqrt{t} \mathbf{b}_t \|_2\\
&\leq C_{\rho_*}\| \| \mathbf{b} \|_\infty^2\| \sqrt{\rho t}\mathbf{u}_t \|_2^2 + \varepsilon \| \sqrt{t} \nabla \mathbf{b}_t \|_2^2 .
\end{split}
\end{equation}
For $H_2$ and $H_4$ , we conclude from  $\dot{H}^1\hookrightarrow L^6(\Omega)$, that
\begin{equation}
\begin{split}
|H_2 |&\leq \big|-\int_{\Omega}\sqrt{t} \mathbf{u}_t \cdot \nabla \mathbf{b} \cdot \sqrt{t} \mathbf{b}_t dx\big| \\&\leq \| \nabla \mathbf{b} \|_3 \| \sqrt{t} \mathbf{u}_t \|_6 \| \sqrt{t} \mathbf{b}_t \|_2\\
&\leq \| \nabla \mathbf{b} \|_3 \| \sqrt{t} \mathbf{b}_t \|_2 \| \sqrt{t} \nabla \mathbf{u}_t \|_2 \\&\leq \varepsilon \| \sqrt{t} \nabla \mathbf{u}_t \|_2^2 + C \| \nabla \mathbf{b} \|_3^2 \| \sqrt{t} \mathbf{b}_t \|_2^2,\\
|H_4| &\leq \big|\int_{\Omega}\sqrt{t} \mathbf{b}_t \cdot \nabla \mathbf{u} \cdot \sqrt{t} \mathbf{b}_t dx\big| \\&\leq \| \nabla \mathbf{u} \|_3 \| \sqrt{t} \mathbf{b}_t \|_6 \| \sqrt{t} \mathbf{b}_t \|_2\\
&\leq \| \nabla \mathbf{u} \|_3 \| \sqrt{t} \mathbf{b}_t \|_2 \| \sqrt{t} \nabla \mathbf{b}_t \|_2\\& \leq \varepsilon \| \sqrt{t} \nabla \mathbf{b}_t \|_2^2 + C \| \nabla \mathbf{u} \|_3^2 \| \sqrt{t} \mathbf{b}_t \|_2^2.
\end{split}
\end{equation}
For \( H_3\) and \( H_5\),  according to H\"{o}lder's and Young's inequalities,  we have
\begin{equation}
\begin{split}
|H_3 |&\leq \big|-\int_{\Omega}\sqrt{t} \mathbf{u} \cdot \nabla \mathbf{b}_t \cdot \sqrt{t} \mathbf{b}_t dx\big|\\& \leq \| \sqrt{t}\mathbf{b}_t \|_2 \| \mathbf{u} \|_\infty \| \nabla\sqrt{t} \mathbf{b}_t \|_2\\
&\leq C\| \| \mathbf{u} \|_\infty^2\| \sqrt{t}\mathbf{b}_t \|_2^2 + \varepsilon \| \sqrt{t} \nabla \mathbf{b}_t \|_2^2,\\
|H_5| &\leq \big|\int_{\Omega}\sqrt{t} \mathbf{b} \cdot \nabla \mathbf{u}_t \cdot \sqrt{t} \mathbf{b}_t dx\big|\\& \leq \| \sqrt{t}\mathbf{b}_t \|_2 \| \mathbf{b} \|_\infty \| \nabla\sqrt{t} \mathbf{u}_t \|_2\\
&\leq C\| \| \mathbf{b} \|_\infty^2\| \sqrt{t}\mathbf{b}_t \|_2^2 + \varepsilon \| \sqrt{t} \nabla \mathbf{u}_t \|_2^2.
\end{split}
\end{equation}
For \( K_i\) and \( L_i\), we can refer to \cite{J}.
Summing  \eqref{eq-31}, \eqref{eq-43}, \eqref{eq-32}, and \eqref{eq-33} together, and  combining with  estimates of $ I_i$, $ H_i$, $ K_i$ and $ L_i$, we finally conclude that
\begin{equation}\label{3.30}
\begin{split}
\frac{d}{dt}& \left( \|\sqrt{\rho t}\, \mathbf{u}_t\|_2^2 +\|\sqrt{ t}\, \mathbf{b}_t\|_2^2 + \|\nabla\sqrt{t}\ \phi_t\|_2^2 + \|\sqrt{\rho_t}\, \phi\, \phi_t\|_2^2+ \|\sqrt{\rho t}\phi_t\|_2^2\right)\\
&\quad+ \int_0^t \rho\tau \|\mu_t\|_2^2 \, dt
+ \int_0^t \nu(\phi)\, \tau \|\nabla \mathbf{u}_t\|_2^2 \, dt+ \int_0^t \tau \|\nabla \mathbf{b}_t\|_2^2 \, dt
+ \int_0^t \|\nabla\sqrt{\tau}\, \mu_t\|_2^2 \, dt \\
&\leq h_1(t) \left( 1 + \|\sqrt{\rho t}\, \mathbf{u}_t\|_2^2 + \|\sqrt{t}\, \mathbf{b}_t\|_2^2 + \|\nabla\sqrt{t}\, \phi_t\|_2^2 + \|\sqrt{\rho_t}\, \phi\, \phi_t\|_2^2+\|\sqrt{\rho t}\phi_t\|_2^2 \right),
\end{split}
\end{equation}
where
\begin{align*}
&h_1(t)=
(1 + C + C_{\rho^*} + C_{\rho^*\rho_*T} + C_T)
\Big(\big( \|\phi\|_{12}^3 + \|\phi\|_4 \big)^2
\big( \|\nabla^2 \phi\|_4^2 + \|\nabla \phi\|_4^2 + \|\mathbf{u}\|_4^2 \big) \\
&+ \big( \|\nabla \phi\|_6 + \|\phi\|_\infty^2 \|\nabla\phi\|_6 \big)^2
\big( \|\mathbf{u}\|_6^2 + \|\nabla\phi\|_4^2  + \|\mathbf{u}\|_4^2 \big) + \|\mathbf{u}\|_4^2 \big( \|\nabla^2 \phi\|_6 + \|\nabla \phi\|_{12}^2 \|\phi\|_\infty + \|\nabla^2 \phi\|_6 \|\phi\|_\infty^2 \big)^2 \\
&+ \|\nabla \mathbf{u}\|_2^2 + \| \mu\|_4^2 \|\nabla \phi\|_4^2 + \|\nabla \phi\|_2^2 + \|\phi\|_\infty^4 + \|\nabla \mu\|_2^2 + \|\nabla \phi\|_\infty^2
 + \|\nabla^2 \phi\|_2^2 + \|\nabla \phi\|_4^2 + \|\phi\|_8^4 \|\nabla\phi\|_4^2 \\
&+ \|\nabla \phi\|_4^8 + \big( \|\phi\|_\infty^2 + 1 \big)^2 \|\nabla \phi\|_4^2 + \|\mathbf{u}\|_4^4 + \|\mathbf{u}\|_6^4 + \|\phi\|_6^4 + \|\nabla \mathbf{u}\|_2^2  + \|\mathbf{u}\|_4 \|\nabla \phi\|_4 \|\phi\|_6 + \|\nabla \mu\|_2^2 + \|\nabla \phi\|_4^2\\& + \|\phi\|_\infty^4
+ \|\nabla \phi\|_4^2 + \|\mathbf{u}\|_6^2 + \|\mathbf{u}\|_2^2 + \|\nabla \mathbf{u}\|_2^2 + \|\nabla \mathbf{u}\|_2^{\frac{7}{2}} \\
&+ \|\mathbf{u}\|_4^2 \big( \|\nabla \phi\|_4^2 + \|\nabla^2 \phi\|_4^2 + \|\mu\|_4^2 \big)
+ \|\nabla^2 \phi\|_4^2 \|\mu\|_4^2 + \|\nabla \phi\|_8^2 \|\phi\|_8^2
+ \big( \|\phi\|_\infty^3 + \|\phi\|_\infty\big)^2 \Bigr)\\
&\times\Big( \|\mathbf{u}\|_{\infty}^2 +\|\mathbf{b}\|_{\infty}^2 + \|\mathbf{u}\|_{\infty}^4 + \|\mu\|_{\infty}^2 + \|\phi_t\|_2^2 + \| \phi_t\|_4^2 + \| \phi_t\|_6^2 + \|\phi_t\|_2 + \|\mathbf{u}_t\|_2^2 + \|\nabla \phi_t\|_2^2 \\
&+  \|\nabla^2 \mathbf{u}\|_2^2 + \|\nabla \mathbf{u}\|_3^2 +\|\nabla \mathbf{b}\|_3^2 + \|\nabla \mu\|_4^2 + \|\nabla \mu\|_6^2 + \|\nabla^2 \mu\|_6^2 + \|\nabla^2 \mu\|_4^2 + \|\nabla \mu\|_{w^{1,6}}^2 \Big)
\end{align*}
only depends on, \( \rho^* \), \( \rho_* \), $T$ and the initial value.
From  Theorem \ref{1.1} and the  3-D Gagliardo-Nirenberg interpolation inequality: \( \| u \|_{\infty}^4 \leq \|\nabla u \|_2^2 \| \nabla^2 u \|_2^2 \), we conclude that $ h_1(t)\in L^1_{loc}(\mathbb{R}^+) $. This along with \eqref{3.30} and Gronwall's inequality yields that \eqref{3D}  holds for \( t \geq 0 \).
\end{proof}
\begin{lemma}
Assume \(d = 2\), and that that a strong solution $(\rho, \mathbf{u}, \mathbf{b}, P,\phi, \mu)$ to the initial--boundary value problem \eqref{1.1}-\eqref{1.2MHD}. Then for all \(0\leq t\leq T\), we have
\begin{equation}\label{2D}
\begin{split}
\|\sqrt{\rho t}\mathbf{u}_t\|_2^2&+\|\sqrt{t}\mathbf{b}_t\|_2^2+\|\sqrt{\rho t}\phi_t\phi\|_2^2+\|\sqrt{ t}\nabla\phi_t\|_2^2+\|\sqrt{\rho t}\phi_t\|_2^2+\int_{0}^{t}\nu(\phi)\tau\|\nabla \mathbf{u}_t\|_2^2d\tau \\
&+\int_{0}^{t}\tau\|\nabla\mu_t\|_2^2d\tau+\int_{0}^{t}\tau\|\nabla\mathbf{b}_t\|_2^2d\tau
+\int_{0}^{t}\tau\rho\|\mu_t\|_2^2d\tau
\leq\exp\Big(\int_{0}^{t}h_2(\tau)d\tau\Big)-1,
\end{split}
\end{equation}
where
 \begin{align*}
&h_2(t)=(1 + C + C_{\rho^*} + C_{\rho^*\rho_*T} + C_T)
\Big(\big( \|\phi\|_{12}^3 + \|\phi\|_4 \big)^2
\big( \|\nabla^2 \phi\|_4^2 + \|\nabla \phi\|_4^2 + \|\mathbf{u}\|_4^2 \big) \\
&+ \big( \|\nabla \phi\|_6 + \|\phi\|_\infty^2 \|\nabla\phi\|_6 \big)^2
\big( \|\mathbf{u}\|_6^2 + \|\nabla\phi\|_4^2  + \|\mathbf{u}\|_4^2 \big) + \|\mathbf{u}\|_4^2 \big( \|\nabla^2 \phi\|_6 + \|\nabla \phi\|_{12}^2 \|\phi\|_\infty + \|\nabla^2 \phi\|_6 \|\phi\|_\infty^2 \big)^2 \\
&+ \|\nabla \mathbf{u}\|_2^2 + \| \mu\|_4^2 \|\nabla \phi\|_4^2 + \|\nabla \phi\|_2^2 + \|\phi\|_\infty^4 + \|\nabla \mu\|_2^2 + \|\nabla \phi\|_\infty^2
 + \|\nabla^2 \phi\|_2^2 + \|\nabla \phi\|_4^2 + \|\phi\|_8^4 \|\nabla\phi\|_4^2 \\
&+ \|\nabla \phi\|_4^8 + \big( \|\phi\|_\infty^2 + 1 \big)^2 \|\nabla \phi\|_4^2 + \|\mathbf{u}\|_4^4 + \|\mathbf{u}\|_6^4 + \|\phi\|_6^4 + \|\nabla \mathbf{u}\|_2^2  + \|\mathbf{u}\|_4 \|\nabla \phi\|_4 \|\phi\|_6 + \|\nabla \mu\|_2^2 + \|\nabla \phi\|_4^2 \\
&+ \|\phi\|_\infty^4 + \|\nabla \phi\|_4^2 + \|\mathbf{u}\|_6^2 + \|\mathbf{u}\|_2^2 + \|\nabla \mathbf{u}\|_2^2
+ \|\mathbf{u}\|_4^2 \big( \|\nabla \phi\|_4^2 + \|\nabla^2 \phi\|_4^2 + \|\mu\|_4^2 \big)
\\&+ \|\nabla^2 \phi\|_4^2 \|\mu\|_4^2 + \|\nabla \phi\|_8^2 \|\phi\|_8^2
+ \big( \|\phi\|_\infty^3 + \|\phi\|_\infty\big)^2 \Big)\times\Big( \|\mathbf{u}\|_{\infty}^2 + \|\mathbf{u}\|_{\infty}^4 + \|\mu\|_{\infty}^2 + \|\phi_t\|_2^2 \\
&+ \| \phi_t\|_4^2 + \| \phi_t\|_6^2 + \|\phi_t\|_2 + \|\mathbf{u}_t\|_2^2 + \|\nabla \phi_t\|_2^2 +  \|\nabla^2 \mathbf{u}\|_2^2 + \|\nabla \mathbf{u}\|_3^2 + \|\nabla \mu\|_4^2 + \|\nabla \mu\|_6^2\\
& + \|\nabla^2 \mu\|_6^2 + \|\nabla^2 \mu\|_4^2 + \|\nabla \mu\|_{w^{1,6}}^2 \Big)
\end{align*}
is $ h_2(t) \in L^1_{loc}(\mathbb{R}^+)$  only depends on, \( \rho^* \), \( \rho_* \), $T$ and the initial value.
\end{lemma}
\begin{proof}
Compared with the proof of the 3-D case, we here only show some different parts for \(I_{31}$, $I_4$, $I_6$, $I_8 \), \( K_{43} \) and \( K_{9} \) in what follows. Combining Holder's inequality and Sobolev embedding, we have
\begin{align*}
|I_{31}| &\leq \big|\int_{\mathbb{T}^2} \sqrt{\rho t} |\mathbf{u}| \|\nabla \mathbf{u}\| \|\nabla \mathbf{u}\| \sqrt{\rho t} \mathbf{u}_t dx\big| \\&\leq \|\mathbf{u}\|_{\infty}^2 \| \sqrt{\rho t} \mathbf{u}_t \|_2^2 + C _{T \rho^*} \| \nabla \mathbf{u} \|_4^4,\\
&\leq \|\mathbf{u}\|_{\infty}^2 \| \sqrt{\rho t} \mathbf{u}_t \|_2^2 + C_{ T \rho^*} \| \nabla \mathbf{u} \|_2^2\| \mathbf{u} \|_2^2,\\
|I_4| &\leq C_{\rho*T}\| \nabla \mathbf{u} \|_2 \| \sqrt{t} \mathbf{u}_t \|_4^2
\\& \leq C_{\rho*\rho_*T} \| \nabla \mathbf{u} \|_2 \| \sqrt{\rho t} \mathbf{u}_t \|_2 \| \sqrt{t} \nabla \mathbf{u}_t \|_2 \\ &\leq \varepsilon \| \sqrt{t} \nabla \mathbf{u}_t \|_2^2 + C_{\rho*\rho_*T} \| \nabla \mathbf{u} \|_2^2 \| \sqrt{\rho t} \mathbf{u}_t \|_2^2,\\
|I_6| &\leq \big|- \int _{\Omega}\nu'(\phi) \phi_t D \mathbf{u} \cdot \nabla t\mathbf{u}_t dx \big| \\&\leq C \| \sqrt{t} \nabla \mathbf{u}_t \|_2 \|\sqrt t \phi_t \|_4 \| D \mathbf{u} \|_4\\
&\leq \varepsilon \| \sqrt{t} \nabla \mathbf{u}_t \|_2^2 + C_T\| \nabla \mathbf{u} \|_2\| \nabla^2 \mathbf{u} \|_2 \|\nabla\sqrt t\phi_t \|_2\|\phi_t \|_2\\
&\leq \varepsilon \| \sqrt{t} \nabla \mathbf{u}_t \|_2^2 + C_T\| \nabla \mathbf{u} \|_2^2\| \nabla^2 \mathbf{u} \|_2^2+C \|\nabla\sqrt t\phi_t \|_2^2\|\phi_t \|_2^2\\
|I_8| &\leq \big| \int_{\Omega} ( \|\sqrt{t}\rho \, \mu_t \, \nabla \phi) \, (\sqrt{t} \, \mathbf{u}_t) \, dx \big|\\& \leq C_{\rho*} \, \|\sqrt{\rho t} \, \mu_t\|_2 \, \|\nabla \, \phi\|_4 \, \|\sqrt{t}\mathbf{u}_t\|_4\\
&\leq C_{\rho*T} \, \|\sqrt{\rho t} \, \mu_t\|_2 \, \|\nabla \, \phi\|_4 \, \|\nabla\sqrt{t}u_t\|_2^\frac{1}{2}\|\mathbf{u}_t\|_2^\frac{1}{2}\\& \leq \varepsilon\|\nabla \sqrt{t} \, \mathbf{u}_t\|_2^2 + C_{\rho*T} \, \|\nabla \phi\|_4^2 \, \|\mathbf{u}_t\|_2^2+\varepsilon\|\sqrt{\rho t} \, \mu_t\|_2^2,\\
|I_{13}| &\leq \big|\int_{\Omega} ( \|\sqrt{t}\mathbf{b}_t \cdot\nabla\mathbf{b}) \cdot (\sqrt{t} \, \mathbf{u}_t) \, dx \big|\\& \leq  \|\nabla\mathbf{b}_2 \, \|\sqrt{t}\mathbf{u}_t\|_4\|\sqrt{t}\mathbf{b}_t\|_4\\
&\leq C\|\sqrt{t}\mathbf{u}_t\|_2^{\frac{1}{2}}\|\sqrt{t}\mathbf{u}_t\|_{H^1}^{\frac{1}{2}} \|\sqrt{t}\mathbf{b}_t\|_2^{\frac{1}{2}}\|\sqrt{t}\mathbf{b}_t\|_{H^1}^{\frac{1}{2}} \\ &\leq C\|\sqrt{t}\mathbf{u}_t\|_2\|\sqrt{t}\mathbf{u}_t\|_{H^1}
+C\|\sqrt{t}\mathbf{b}_t\|_2\|\sqrt{t}\mathbf{b}_t\|_{H^1}\\
&\leq C\|\sqrt{t}\mathbf{u}_t\|_2^2+\varepsilon\|\sqrt{t}\mathbf{u}_t\|_{H^1}^2
+C\|\sqrt{t}\mathbf{b}_t\|_2^2+\varepsilon\|\sqrt{t}\mathbf{b}_t\|_{H^1}^2\\
&\leq C\|\sqrt{t}\mathbf{u}_t\|_2^2+\varepsilon\|\sqrt{t}\mathbf{u}_t\|_2^2
+C\|\sqrt{t}\mathbf{b}_t\|_2^2+\varepsilon\|\sqrt{t}\mathbf{b}_t\|_2^2,\\
|H_2 |&\leq \| \nabla \mathbf{b} \|_4 \| \sqrt{t} \mathbf{u}_t \|_4 \| \sqrt{t} \mathbf{b}_t \|_2\\&\leq \| \nabla \mathbf{b} \|_4^2\| \sqrt{t} \mathbf{b}_t \|_2^2+ \| \sqrt{t} \mathbf{u}_t \|_4^2\\
& \leq \| \nabla \mathbf{b} \|_4^2\| \sqrt{t}\mathbf{b}_t \|_2^2 + C \| \sqrt{t} \mathbf{u}_t \|_4^2\\&\leq\| \nabla \mathbf{b} \|_4^2\| \sqrt{t}\mathbf{b}_t \|_2^2 + C \| \sqrt{t} \mathbf{u}_t \|_2 \| \sqrt{t} \mathbf{u}_t \|_H^1\\
&\leq\| \nabla \mathbf{b} \|_4^2\| \sqrt{t}\mathbf{b}_t \|_2^2 + C \| \sqrt{t} \mathbf{u}_t \|_2^2+\varepsilon \| \sqrt{t} \mathbf{u}_t \|_2^2,\\
|H_4|&\leq \big|\int_{\Omega} ( \sqrt{t}\mathbf{b}_t \cdot\nabla\mathbf{u}) \cdot (\sqrt{t} \, \mathbf{b}_t) \, dx \big|\\& \leq  \|\nabla\mathbf{u}\|_2 \, \|\sqrt{t}\mathbf{b}_t\|_4^2
\\&\leq C \|\nabla\mathbf{u}\|_2^2\|\sqrt{t}\mathbf{b}_t\|_2\|\sqrt{t}\mathbf{b}_t\|_{H^1}\\
&\leq C \|\nabla\mathbf{u}\|_2^2\|\sqrt{t}\mathbf{b}_t\|_2^2
+\varepsilon\|\sqrt{t}\mathbf{b}_t\|_{H^1}^2
\\&\leq C\|\nabla\mathbf{u}_2^2\|\sqrt{t}\mathbf{b}_t\|_2^2
+\varepsilon\|\nabla\sqrt{t}\mathbf{b}_t\|_2^2,\\
|K_{43}|&=3\int_{\Omega}|\sqrt{t}\phi_t|^2\rho\phi\phi_t dx
\\&\leq C_{\rho*T}\|\sqrt{\rho t}\phi_t\phi\|_2\|\phi_t\|_4\|\phi_t\|_4\\
&\leq C_{\rho*T}\|\phi_t\|_4^2\|\sqrt{\rho t}\phi_t\phi\|_2^2+ C\|\phi_t\|_4^2,\\
|K_9|& \leq C_{\rho^* T} \|\sqrt{\rho_t} \mathbf{u}_t\|_2 \|\sqrt{t} \mu_t\|_4 \|\nabla \phi\|_4 \\&\leq C_{\rho^* T} \|\sqrt{\rho_t} \mathbf{u}_t\|_2 \|\sqrt{\rho t} \mu_t\|_2^\frac{1}{2} \|\nabla\sqrt{ t} \mu_t\|_2^\frac{1}{2} \|\nabla \phi\|_4\\
&\leq C_{\rho^* T} \|\nabla \phi\|_4^2 \|\sqrt{\rho_t} u_t\|_2^2 + \varepsilon \|\nabla\sqrt{t} \mu_t\|_2^2+\varepsilon \|\sqrt{\rho t} \mu_t\|_2^2.
\end{align*}
Then summing \eqref{eq-31}, \eqref{eq-32}, and \eqref{eq-33} together, and combining with   estimates of $I_i$, $H_i$, $ K_i$ and $L_i$, we deduce that
\begin{equation}\label{3.38}
\begin{split}
\frac{d}{dt} &\left( \|\sqrt{\rho t}\, \mathbf{u}_t\|_2^2 +\|\sqrt{t}\, \mathbf{b}_t\|_2^2 + \|\nabla\sqrt{t}\ \phi_t\|_2^2 + \|\sqrt{\rho_t}\, \phi\, \phi_t\|_2^2+\|\sqrt{\rho t}\phi_t\|_2^2 \right)
+ \int_0^t \rho\tau \|\mu_t\|_2^2 \, dt
\\&+ \int_0^t \nu(\phi)\, \tau \|\nabla \mathbf{u}_t\|_2^2 \, dt+ \int_0^t \|\nabla\sqrt{\tau} \mathbf{b}_t\|_2^2 \, dt
+ \int_0^t \|\nabla\sqrt{\tau}\, \mu_t\|_2^2 \, dt\\
&\leq h_2(t) \left( 1 + \|\sqrt{\rho t}\, \mathbf{u}_t\|_2^2 +\|\sqrt{ t}\, \mathbf{b}_t\|_2^2 + \|\nabla\sqrt{t}\, \phi_t\|_2^2 + \|\sqrt{\rho t}\, \phi\, \phi_t\|_2^2+\|\sqrt{\rho t}\phi_t\|_2^2 \right),
\end{split}
\end{equation}
where
\begin{align*}
&h_2(t)=(1 + C + C_{\rho^*} + C_{\rho^*\rho_*T} + C_T)
\Big(\big( \|\phi\|_{12}^3 + \|\phi\|_4 \big)^2
\big( \|\nabla^2 \phi\|_4^2 + \|\nabla \phi\|_4^2 + \|\mathbf{u}\|_4^2 \big) \\
&+ \big( \|\nabla \phi\|_6 + \|\phi\|_\infty^2 \|\nabla\phi\|_6 \big)^2
\big( \|\mathbf{u}\|_6^2 + \|\nabla\phi\|_4^2  + \|\mathbf{u}\|_4^2 \big) + \|\mathbf{u}\|_4^2 \big( \|\nabla^2 \phi\|_6 + \|\nabla \phi\|_{12}^2 \|\phi\|_\infty + \|\nabla^2 \phi\|_6 \|\phi\|_\infty^2 \big)^2 \\
&+ \|\nabla \mathbf{u}\|_2^2 + \| \mu\|_4^2 \|\nabla \phi\|_4^2 + \|\nabla \phi\|_2^2 + \|\phi\|_\infty^4 + \|\nabla \mu\|_2^2 + \|\nabla \phi\|_\infty^2
 + \|\nabla^2 \phi\|_2^2 + \|\nabla \phi\|_4^2 + \|\phi\|_8^4 \|\nabla\phi\|_4^2 \\
&+ \|\nabla \phi\|_4^8 + \big( \|\phi\|_\infty^2 + 1 \big)^2 \|\nabla \phi\|_4^2 + \|\mathbf{u}\|_4^4 + \|\mathbf{u}\|_6^4 + \|\phi\|_6^4 + \|\nabla \mathbf{u}\|_2^2  + \|\mathbf{u}\|_4 \|\nabla \phi\|_4 \|\phi\|_6 + \|\nabla \mu\|_2^2 + \|\nabla \phi\|_4^2 \\
&+ \|\phi\|_\infty^4 + \|\nabla \phi\|_4^2 + \|\mathbf{u}\|_6^2 + \|\mathbf{u}\|_2^2 + \|\nabla \mathbf{u}\|_2^2
+ \|\mathbf{u}\|_4^2 \big( \|\nabla \phi\|_4^2 + \|\nabla^2 \phi\|_4^2 + \|\mu\|_4^2 \big)
\\&+ \|\nabla^2 \phi\|_4^2 \|\mu\|_4^2 + \|\nabla \phi\|_8^2 \|\phi\|_8^2
+ \big( \|\phi\|_\infty^3 + \|\phi\|_\infty\big)^2 \Big)\times\Big( \|\mathbf{u}\|_{\infty}^2 + \|\mathbf{u}\|_{\infty}^4 + \|\mu\|_{\infty}^2 + \|\phi_t\|_2^2 \\
&+ \| \phi_t\|_4^2 + \| \phi_t\|_6^2 + \|\phi_t\|_2 + \|\mathbf{u}_t\|_2^2 + \|\nabla \phi_t\|_2^2 +  \|\nabla^2 \mathbf{u}\|_2^2 + \|\nabla \mathbf{u}\|_3^2 + \|\nabla \mu\|_4^2 + \|\nabla \mu\|_6^2\\
& + \|\nabla^2 \mu\|_6^2 + \|\nabla^2 \mu\|_4^2 + \|\nabla \mu\|_{w^{1,6}}^2 \Big)
\end{align*}
only depends on, \( \rho^* \), $T$ and the initial value.
Combining with  Theorem \ref{1.1}, and the 2-D Ladyzhenskaya inequality: \( \| \mathbf{u} \|_{\infty}^4 \leq \| \mathbf{u} \|_2^2 \| \nabla^2 \mathbf{u} \|_2^2 \) yields that
$h_2(t) \in L^1_{loc}(\mathbb{R}^+)$. Then applying Gronwall's inequality to \eqref{3.38} implies  that \eqref{2D} holds for \( t \geq 0 \).
\end{proof}

\vskip .2in
\subsection{The estimate  for quantity $\int_{0}^{T}\|(\nabla \mathbf{u}, \nabla \mathbf{b}, \nabla \mu, \nabla\phi)(\tau)\|_{L^{\infty}}d\tau$}
\label{int}
In  this section, our main  goal is to achieve the bound  estimate  for quantity $\int_{0}^{T}\|(\nabla \mathbf{u}, \nabla \mathbf{b}, \nabla \mu, \nabla\phi)(\tau)\|_{L^{\infty}}d\tau$ in terms of the initial data and of $T$ by performing the shift of integrability method. This is given by the following two lemmas in two and three dimensions, respectively.
\begin{lemma}\label{3DDD}
Assume \( d = 3 \), then for all \( T > 0 \), \( p \in [2,\infty] \), we have
\begin{equation}\label{3DD}
\left\lVert (\nabla^2 \sqrt{t} \mathbf{u}, \nabla^2 \sqrt{t} \mathbf{b}, \nabla^2 \sqrt{t}\phi, \nabla^2 \sqrt{t} \mu) \right\rVert_{L^p(0,T;L^r)} + \left\lVert \nabla \sqrt{t} P \right\rVert_{L^p(0,T;L^r)} \leq C_{0,T}, \quad \text{for} \quad 2 \leq r \leq \frac{6p}{3p - 4},
\end{equation}
where \( C_{0,T} \) depends only on \( \rho^* \), \( p \) and the initial value. Furthermore, for \( s \in \left[1,\frac{4}{3}\right) \), then for some \( \theta > 0 \), we have
\begin{equation}
\int_0^T \left( \left\lVert \nabla \mathbf{u} \right\rVert_\infty^s +\left\lVert \nabla \mathbf{b} \right\rVert_\infty^s + \left\lVert \nabla \mu\right\rVert_\infty^s + \left\lVert \nabla\phi \right\rVert_\infty^s \right) dt \leq C_{0,T} T^\theta.\label{eq-s}
\end{equation}
\end{lemma}
\begin{proof}
From \eqref{1.1}, we have
\begin{align}
\begin{cases}
&-\mathrm{div}\left( \sqrt{t}\, v(\phi)\, D\mathbf{u} \right) + \nabla \sqrt{t}\, p = -\mathrm{div}\sqrt{t}\, (\nabla \phi \otimes \nabla \phi) - \rho \sqrt{t}\, (\mathbf{u}_t + \mathbf{u} \cdot \nabla \mathbf{u})+  \mathbf{b} \cdot \nabla \mathbf{b}
,\label{eq-n} \\
&\mathrm{div}\sqrt{t}\, \mathbf{u} = 0,\\
&\Delta\sqrt{t}\mathbf{b}=\sqrt{t}\mathbf{b}_t + \sqrt{t}\mathbf{u} \cdot \nabla \mathbf{b}-\sqrt{t}\mathbf{b} \cdot \nabla \mathbf{u},\\
&\Delta \sqrt{t}\mu = \rho \sqrt{t}\, \phi_t + \rho \sqrt{t}\, \mathbf{u} \nabla \phi, \\
&-\Delta \sqrt{t}\, \phi + \rho \sqrt{t}\, \Phi'(\phi) = \rho \sqrt{t}\, \mu.
\end{cases}
\end{align}
Obviously,
\begin{equation*}
\begin{split}
\Phi'(\phi)&=\phi^3-\phi,\\
-\operatorname{div}(\nabla \phi \otimes \nabla \phi) &= -\Delta \phi \nabla \phi - \nabla \Big( \frac{1}{2} |\nabla \phi|^2 \Big) \\
&= \rho \mu \nabla \phi - \rho \Phi'(\phi) \nabla \phi - \nabla \Big( \frac{1}{2} |\nabla \phi|^2 \Big) \\
&= \rho \mu \nabla \phi - \rho \nabla \Phi(\phi) - \nabla \Big( \frac{1}{2} |\nabla \phi|^2 \Big)\\
&= \rho \mu \nabla \phi - \rho (\phi^3-\phi)\nabla\phi - \nabla\phi \nabla^2\phi.
\end{split}
\end{equation*}
From Theorem \ref{1.1} and \eqref{3D}, we get
\begin{equation*}
\rho\sqrt{t}\mathbf{u}_t, \sqrt{t}\mathbf{b}_t, \rho\sqrt{t}\phi_t \in L^\infty(0,T;L^2) \cap L^2(0,T;H^1),
\end{equation*}
which together with $\dot{H}^1(\Omega) \hookrightarrow L^6(\Omega)$, yields that
\begin{equation*}
\rho\sqrt{t}\mathbf{u}_t, \sqrt{t}\mathbf{b}_t, \rho\sqrt{t}\phi_t \in L^\infty(0,T;L^2) \cap L^2(0,T;L^q) \quad \text{with} \quad q \leq 6.
\end{equation*}
This along with the interpolation inequality  gives rise to
\begin{equation*}
\|\rho\sqrt{t}\mathbf{u}_t\|_{L^p(0,T;L^r)} \leq \|\rho\sqrt{t}\mathbf{u}_t\|_{L^\infty(0,T;L^2)}^{1 - \frac{2}{p}} \|\rho\sqrt{t}\mathbf{u}_t\|_{L^2(0,T;L^q)}^{\frac{2}{p}},
\end{equation*}
\begin{equation*}
\|\sqrt{t}\mathbf{b}_t\|_{L^p(0,T;L^r)} \leq \|\sqrt{t}\mathbf{b}_t\|_{L^\infty(0,T;L^2)}^{1 - \frac{2}{p}} \|\sqrt{t}\mathbf{b}_t\|_{L^2(0,T;L^q)}^{\frac{2}{p}},
\end{equation*}
and
\begin{equation*}
\|\rho\sqrt{t}\phi_t\|_{L^p(0,T;L^r)} \leq \|\rho\sqrt{t}\phi_t\|_{L^\infty(0,T;L^2)}^{1 - \frac{2}{p}} \|\rho\sqrt{t}\phi_t\|_{L^2(0,T;L^q)}^{\frac{2}{p}},
\end{equation*}
with \(\frac{1}{r} = \frac{p - 2}{2p} + \frac{2}{pq}\). Here, when \(q\) takes \(6\), then \(r\) may take the maximum value of \(\frac{6p}{3p - 4}\). Then we readily get
\begin{equation}
\|(\rho\sqrt{t}\mathbf{u}_t, \sqrt{t}\mathbf{b}_t, \rho\sqrt{t}\phi_t)\|_{L^p(0,T;L^r)} \leq C_{0,T} \quad \text{for} \quad p \in [2,\infty], \quad r \in \Big[2, \frac{6p}{3p - 4}\Big].\label{eq-11}
\end{equation}
It follows from Theorem \ref{1.1} and $0<\rho_* \leq \rho(x, t) \leq \rho^*$, that
\begin{equation*}
\rho\sqrt{t}\phi, \rho\sqrt{t}\mu, \rho\sqrt{t}\phi^3 \in L^\infty(0,T;L^2) \cap L^2(0,T;H^1).
\end{equation*}
Using $\dot{H}^1(\Omega) \hookrightarrow L^6(\Omega)$, we have
\begin{equation*}
\rho\sqrt{t}\phi, \rho\sqrt{t}\mu, \rho\sqrt{t}\phi^3 \in L^\infty(0,T;L^2) \cap L^2(0,T;L^q) \quad \text{with} \quad q \leq 6.
\end{equation*}
Employing  interpolation inequality, we conclude that
\begin{align*}
\|\rho\sqrt{t}\phi\|_{L^p(0,T;L^r)} \leq \|\rho\sqrt{t}\phi\|_{L^\infty(0,T;L^2)}^{1 - \frac{2}{p}} \|\rho\sqrt{t}\phi\|_{L^2(0,T;L^q)}^{\frac{2}{p}},\\
\|\rho\sqrt{t}\mu\|_{L^p(0,T;L^r)} \leq \|\rho\sqrt{t}\mu\|_{L^\infty(0,T;L^2)}^{1 - \frac{2}{p}} \|\rho\sqrt{t}\mu\|_{L^2(0,T;L^q)}^{\frac{2}{p}},\\
\|\rho\sqrt{t}\phi^3\|_{L^p(0,T;L^r)} \leq \|\rho\sqrt{t}\phi^3\|_{L^\infty(0,T;L^2)}^{1 - \frac{2}{p}} \|\rho\sqrt{t}\phi^3\|_{L^2(0,T;L^q)}^{\frac{2}{p}}
\end{align*}
with \(\frac{1}{r} = \frac{p - 2}{2p} + \frac{2}{pq}\). Here, when \(q\) takes \(6\), then \(r\) may take the maximum value of \(\frac{6p}{3p - 4}\). Then we have
\begin{equation}
\|(\rho\sqrt{t}\phi, \rho\sqrt{t}\mu,\rho\sqrt{t}\phi^3)\|_{L^p(0,T;L^r)} \leq C_{0,T} \quad \text{for} \quad p \in [2,\infty], \quad r \in \Big[2, \frac{6p}{3p - 4}\Big].\label{eq-12}
\end{equation}
According to Theorem \ref{1.1}, we have $\nabla \mathbf{u}, \nabla \mathbf{b}, \nabla \phi \in L^\infty(0,T;L^2) \cap L^2(0,T;H^1)$, and from $\dot{H}^1(\Omega) \hookrightarrow L^6(\Omega)$, we get $\nabla \mathbf{u}, \nabla \phi \in L^\infty(0,T;L^2) \cap L^2(0,T;L^q)\) with \(q \leq 6$. By interpolation inequality, we obtain
\begin{equation*}
\|\nabla \mathbf{u}\|_{L^p(0,T;L^r)} \leq \|\nabla \mathbf{u}\|_{L^\infty(0,T;L^2)}^{1 - \frac{2}{p}} \|\nabla \mathbf{u}\|_{L^2(0,T;L^q)}^{\frac{2}{p}},
\end{equation*}
\begin{equation*}
\|\nabla \mathbf{b}\|_{L^p(0,T;L^r)} \leq \|\nabla \mathbf{b}\|_{L^\infty(0,T;L^2)}^{1 - \frac{2}{p}} \|\nabla \mathbf{b}\|_{L^2(0,T;L^q)}^{\frac{2}{p}},
\end{equation*}
\begin{equation*}
\|\nabla \phi\|_{L^p(0,T;L^r)} \leq \|\nabla \phi\|_{L^\infty(0,T;L^2)}^{1 - \frac{2}{p}} \|\nabla \phi\|_{L^2(0,T;L^q)}^{\frac{2}{p}}
\end{equation*}
with \(\frac{1}{r} = \frac{p - 2}{2p} + \frac{2}{pq}\). Then we have
\begin{equation*}
\|(\nabla \mathbf{u}, \nabla \mathbf{b}, \nabla \phi)\|_{L^p(0,T;L^r)} \leq C_{0,T} \quad \text{for} \quad p \in [2,\infty], \quad r \in \Big[2, \frac{6p}{3p - 4}\Big],
\end{equation*}
which means that
\begin{equation}
\nabla \mathbf{u}, \nabla \mathbf{b}, \nabla \phi \in L^4(0, T; L^3).\label{eq-15}
\end{equation}
On the other hand, using Gagliardo-Nirenberg interpolation inequality $\|v\|_{L^\infty}^4 \leq C\|\nabla v\|_{L^2}^2 \|\nabla^2 v\|_{L^2}^2$ leads to
\begin{align*}
\|\mathbf{u}\|_{L^4(0,T;L^\infty)} &\leq \|\nabla \mathbf{u}\|_{L^\infty(0,T;L^2)}^{\frac{1}{2}} \|\nabla^2 \mathbf{u}\|_{L^2(0,T;L^2)}^{\frac{1}{2}},\label{eq-16}  \\
\|\mathbf{b}\|_{L^4(0,T;L^\infty)} &\leq \|\nabla \mathbf{b}\|_{L^\infty(0,T;L^2)}^{\frac{1}{2}} \|\nabla^2 \mathbf{b}\|_{L^2(0,T;L^2)}^{\frac{1}{2}},  \\
\|\phi\|_{L^4(0,T;L^\infty)} &\leq \|\nabla \phi\|_{L^\infty(0,T;L^2)}^{\frac{1}{2}} \|\nabla^2 \phi\|_{L^2(0,T;L^2)}^{\frac{1}{2}},\\
\|\phi^3\|_{L^4(0,T;L^\infty)}&=\|\phi\|^3_{L^{12}(0,T;L^\infty)}\leq \|\nabla \phi\|_{L^\infty(0,T;L^2)}^{\frac{3}{2}} \|\nabla^2 \phi\|_{L^{6}(0,T;L^2)}^{\frac{3}{2}},\\
\|\mu\|_{L^4(0,T;L^\infty)} &\leq \|\nabla \mu\|_{L^\infty(0,T;L^2)}^{\frac{1}{2}} \|\nabla^2 \mu\|_{L^2(0,T;L^2)}^{\frac{1}{2}}.
\end{align*}
Thanks to Theorem  \ref{1.1}, we conclude that
\begin{align}
\sqrt{t}\rho \mathbf{u}, \sqrt{t} \mathbf{b}, \sqrt{t}\rho \phi,\sqrt{t}\rho |\phi|^3,\sqrt{t}\rho \mu\in L^4(0,T;L^\infty).
\end{align}
Using Holder's inequality, and then combining with \eqref{eq-15} and \eqref{eq-16}, give rise to
\begin{align*}
\sqrt{t}\rho \mathbf{u} \cdot \nabla \mathbf{u}, \sqrt{t}\mathbf{u} \cdot \nabla \mathbf{b}, \sqrt{t}\mathbf{b} \cdot \nabla \mathbf{u}, \sqrt{t}\mathbf{b} \cdot \nabla \mathbf{b}, \sqrt{t}\rho \mathbf{u}\cdot  \nabla \phi,\sqrt{t}\rho |\phi|^3\nabla \phi,\sqrt{t}\rho \mu  \nabla \phi,\sqrt{t}\rho\phi\nabla \phi \in L^2(0,T;L^3).
\end{align*}
Similarly, we also have
\begin{align*}
\sqrt{t}\rho \mathbf{u}, \sqrt{t}\mathbf{b}, \sqrt{t}\rho \phi,\sqrt{t}\rho |\phi|^3,\sqrt{t}\rho\mu\in L^\infty(0,T;L^6);
\nabla \mathbf{u}, \nabla \mathbf{b}, \nabla \phi \in L^\infty(0,T;L^2),
\end{align*}
which implies that
\begin{equation*}
\sqrt{t}\rho \mathbf{u}\cdot \nabla\mathbf{u}, \sqrt{t}\mathbf{u} \cdot \nabla \mathbf{b}, \sqrt{t}\mathbf{b} \cdot \nabla \mathbf{u}, \sqrt{t}\mathbf{b} \cdot \nabla \mathbf{b}, \sqrt{t}\rho \mathbf{u} \nabla \phi ,\sqrt{t}\rho |\phi|^3\nabla \phi,\sqrt{t}\rho \mu \nabla \phi,\sqrt{t}\rho\phi\nabla \phi\in L^\infty(0, T; L^{3/2}).
\end{equation*}
It follows from the interpolation inequality and Holder's inequality that
\begin{align*}
\|\sqrt{t}\rho \mathbf{u} \cdot \nabla \mathbf{u}\|_{L^p(0,T;L^r)} &\leq \|\sqrt{t}\rho \mathbf{u} \cdot \nabla \mathbf{u}\|_{L^2(0,T;L^3)}^{\frac{2}{p}} \|\sqrt{t}\rho \mathbf{u} \cdot \nabla \mathbf{u}\|_{L^\infty(0,T;L^{3/2})}^{1 - \frac{2}{p}}, \\
\|\sqrt{t}\mathbf{u} \cdot \nabla \mathbf{b}\|_{L^p(0,T;L^r)} &\leq \|\sqrt{t} \mathbf{u} \cdot \nabla \mathbf{b}\|_{L^2(0,T;L^3)}^{\frac{2}{p}} \|\sqrt{t} \mathbf{u} \cdot \nabla \mathbf{b}\|_{L^\infty(0,T;L^{3/2})}^{1 - \frac{2}{p}}, \\
\|\sqrt{t}\mathbf{b} \cdot \nabla \mathbf{u}\|_{L^p(0,T;L^r)} &\leq \|\sqrt{t} \mathbf{b} \cdot \nabla \mathbf{u}\|_{L^2(0,T;L^3)}^{\frac{2}{p}} \|\sqrt{t} \mathbf{b} \cdot \nabla \mathbf{u}\|_{L^\infty(0,T;L^{3/2})}^{1 - \frac{2}{p}}, \\
\|\sqrt{t}\mathbf{b} \cdot \nabla \mathbf{b}\|_{L^p(0,T;L^r)} &\leq \|\sqrt{t} \mathbf{b} \cdot \nabla \mathbf{b}\|_{L^2(0,T;L^3)}^{\frac{2}{p}} \|\sqrt{t} \mathbf{b} \cdot \nabla \mathbf{b}\|_{L^\infty(0,T;L^{3/2})}^{1 - \frac{2}{p}}, \\
\|\sqrt{t}\rho \mathbf{u} \cdot\nabla \phi\|_{L^p(0,T;L^r)} &\leq \|\sqrt{t}\rho \mathbf{u}\cdot\nabla \phi\|_{L^2(0,T;L^3)}^{\frac{2}{p}} \|\sqrt{t}\rho \mathbf{u} \cdot\nabla \phi\|_{L^\infty(0,T;L^{3/2})}^{1 - \frac{2}{p}},\\
\|\sqrt{t}\rho \mu\nabla \phi\|_{L^p(0,T;L^r)} &\leq \|\sqrt{t}\rho \mu\nabla \phi\|_{L^2(0,T;L^3)}^{\frac{2}{p}} \|\sqrt{t}\rho \mu\nabla \phi\|_{L^\infty(0,T;L^{3/2})}^{1 - \frac{2}{p}},\\
\|\sqrt{t}\rho \phi\nabla \phi\|_{L^p(0,T;L^r)} &\leq \|\sqrt{t}\rho \phi\nabla \phi\|_{L^2(0,T;L^3)}^{\frac{2}{p}} \|\sqrt{t}\rho \phi\nabla \phi\|_{L^\infty(0,T;L^{3/2})}^{1 - \frac{2}{p}},\\
\|\sqrt{t}\rho \phi^3\nabla \phi\|_{L^p(0,T;L^r)} &\leq \|\sqrt{t}\rho \phi^3\nabla \phi\|_{L^2(0,T;L^3)}^{\frac{2}{p}} \|\sqrt{t}\rho \phi^3\nabla \phi\|_{L^\infty(0,T;L^{3/2})}^{1 - \frac{2}{p}}.
\end{align*}
By Theorem  \ref{1.1}, we conclude that
$\sqrt{t}\nabla^2\phi$,$\nabla\phi \in L^4(0, T; L^6)$,
thus
$\sqrt{t}\nabla^2\phi\nabla\phi \in L^2(0, T; L^3)$.
Similarly, we have
\begin{align*}
\sqrt{t}\nabla\phi\in L^\infty(0,T;L^6) ,
\nabla^2\phi\in L^\infty(0,T;L^2),
\end{align*}
which implies that
\begin{equation*}
\sqrt{t}\nabla\phi\nabla^2\phi\in L^\infty(0, T; L^{3/2}).
\end{equation*}
Hence
\begin{align*}
\|\sqrt{t} \nabla \phi\nabla^2\phi\|_{L^p(0,T;L^r)} &\leq \|\sqrt{t} \nabla \phi\nabla^2\phi\|_{L^2(0,T;L^3)}^{\frac{2}{p}} \|\sqrt{t} \nabla \phi\nabla^2\phi\|_{L^\infty(0,T;L^{3/2})}^{1 - \frac{2}{p}},
\end{align*}
with \(\frac{2}{p} + \frac{3}{r} = 2,\, p \geq 2\). Using the maximal regularity estimate for the Stokes equations and the standard estimate of elliptic equations for \eqref{eq-n} yields that
\begin{equation}
\left\lVert (\nabla^2 \sqrt{t} \mathbf{u}, \nabla^2 \sqrt{t} \mathbf{b}, \nabla^2 \sqrt{t} \phi,  \nabla^2 \sqrt{t} \mu) \right\rVert_{L^p(0,T;L^r)} + \left\lVert \nabla \sqrt{t} P \right\rVert_{L^p(0,T;L^r)} \leq C_{0,T} ,\label{eq-v}
\end{equation}
for $p \geq 2 $ and $ \frac{2}{p} + \frac{3}{r} = 2 $. Furthermore, from \eqref{eq-v} and the embedding \( W^1_r(\Omega) \hookrightarrow L^q(\Omega) \) with \( \frac{3}{q} = \frac{3}{r} - 1 \) if \( 1 \leq r < 3 \), we have
\begin{equation}
\nabla \sqrt{t} \mathbf{u}, \nabla \sqrt{t} \mathbf{b}, \nabla \sqrt{t} \phi, \nabla \sqrt{t} \mu \in L^p(0, T; L^q)  \quad \text{with} \quad \frac{2}{p} + \frac{3}{r} = 2.\label{eq-t}
\end{equation}
Employing bounds  of \( \rho \mathbf{u}, \mathbf{b}, \rho\mu, \rho\phi, \rho\phi^3, \nabla^2\phi \in L^\infty(0, T; L^6) \), \eqref{eq-t} and H\"older's inequality, yield that
\begin{equation}
\begin{split}
\sqrt{t} \rho \mathbf{u} \cdot \nabla \mathbf{u}, \sqrt{t} \mathbf{u} \cdot \nabla \mathbf{b}, \sqrt{t} \mathbf{b} \cdot \nabla \mathbf{u}, \sqrt{t} \mathbf{b} \cdot \nabla \mathbf{b}, \sqrt{t} \rho \mathbf{u} \cdot \nabla \phi,\label{eq-13}\\
\sqrt{t} \rho \mu \nabla \phi, \sqrt{t} \rho \phi \nabla \phi, \sqrt{t} \rho \phi^3 \nabla \phi, \sqrt{t} \nabla \phi\nabla^2\phi, \in L^p(0, T; L^r),
\end{split}
\end{equation}
for all $(p, r)$, such that $p \geq 2$, $2 \leq r \leq \frac{6p}{3p - 4}$. From \eqref{eq-11}, \eqref{eq-12} and \eqref{eq-13}, we deduce that
\begin{equation}
\bigl\| (\nabla^2 \sqrt{t} \mathbf{u}, \nabla^2 \sqrt{t} \mathbf{b}, \nabla^2 \sqrt{t} \phi, \nabla^2 \sqrt{t} \mu) \bigr\|_{L^p(0, T; L^r)} + \bigl\| \nabla \sqrt{t} P \bigr\|_{L^p(0, T; L^r)} \leq C_{0, T},\label{eq-14}
\end{equation}
for all \( (p, r) \) such that $ p \geq 2, 2\leq r \leq \frac{6p}{3p - 4} $. Fix $p \in (2,4)$ such that $ps < 2p - 2s$,  and taking  $r=\frac{6p}{3p - 4}$,  thanks to $W_r^1(\Omega) \hookrightarrow L^\infty(\Omega)$(since $r > 3$ for $2 < p < 4$) and \eqref{eq-14}, we have
\begin{equation*}
\begin{split}
&\big( \int_0^T \big\lVert \nabla \mathbf{u} \big\rVert_\infty^s dt \big)^{\frac{1}{s}} +\big( \int_0^T \big\lVert \nabla \mathbf{b} \big\rVert_\infty^s dt \big)^{\frac{1}{s}} + \big( \int_0^T \big\lVert \nabla \phi \big\rVert_\infty^s dt \big)^{\frac{1}{s}}
+\big( \int_0^T \big\lVert \nabla\mu \big\rVert_\infty^s dt \big)^{\frac{1}{s}}  \\
&\leq C \big( \int_0^T \big\lVert \sqrt{t} \nabla \mathbf{u} \big\rVert_{W_r^1}^s \frac{dt}{(\sqrt{t})^s} \big)^{\frac{1}{s}} + C \big( \int_0^T \big\lVert \sqrt{t} \nabla \mathbf{b} \big\rVert_{W_r^1}^s \frac{dt}{(\sqrt{t})^s} \big)^{\frac{1}{s}} \\&\quad+ C \big( \int_0^T \big\lVert \sqrt{t} \nabla \phi \big\rVert_{W_r^1}^s \frac{dt}{(\sqrt{t})^s} \big)^{\frac{1}{s}}
+C \big( \int_0^T \big\lVert \sqrt{t} \nabla \mu \big\rVert_{W_r^1}^s \frac{dt}{(\sqrt{t})^s} \big)^{\frac{1}{s}} \\
&\leq C \big( \int_0^T t^{-\frac{ps}{2p - 2s}} dt \big)^{\frac{1}{s} - \frac{1}{p}} \big( \big\lVert \nabla \sqrt{t} \mathbf{u} \big\rVert_{L_p(0,T;W_r^1)} + \big\lVert \nabla \sqrt{t} \mathbf{b} \big\rVert_{L_p(0,T;W_r^1)} \\&\quad+ \big\lVert \nabla \sqrt{t} \phi \big\rVert_{L_p(0,T;W_r^1)}+\big\lVert \nabla \sqrt{t} \mu \big\rVert_{L_p(0,T;W_r^1)}  \big) \\
&\leq C_{0} T^{-\frac{2p - 2s - ps}{2ps}},
\end{split}
\end{equation*}
which concludes that \eqref{eq-s} holds.
\end{proof}
\begin{lemma}
Assume \( d = 2 \), then for all \( T > 0 \), \( p \in [2,\infty] \),  we have
\begin{equation}
\left\lVert (\nabla^2 \sqrt{t} \mathbf{u}, \nabla^2 \sqrt{t} \mathbf{b}, \nabla^2 \sqrt{t}\phi, \nabla^2 \sqrt{t} \mu) \right\rVert_{L^p(0,T;L^{r})} + \left\lVert \nabla \sqrt{t} P \right\rVert_{L^p(0,T;L^{r})} \leq C_{0,T},\label{eq-r}
\end{equation}
where $\ r \in [2,\frac{6p}{3p - 4}]$, \( C_{0,T} \) depends only on \( \rho^* \), \( p \) and the initial value. Furthermore, for \( s \in \left[1,2\right) \), then for some \( \theta > 0 \), we have
\begin{equation}
\int_0^T \left( \left\lVert \nabla \mathbf{u} \right\rVert_\infty^s +\left\lVert \nabla \mathbf{b} \right\rVert_\infty^s + \left\lVert \nabla \mu\right\rVert_\infty^s + \left\lVert \nabla\phi \right\rVert_\infty^s \right) dt \leq C_{0,T} T^\theta.\label{eq-x}
\end{equation}
\end{lemma}
\begin{proof}
Using \eqref{2D} and $0<\rho_* \leq \rho(x, t) \leq \rho^*$  yields that
\begin{equation*}
(\rho\sqrt{t}\mathbf{u}_t, \sqrt{t}\mathbf{b}_t, \rho\sqrt{t}\phi_t) \in L^\infty(0, T; L^2)\bigcap L^2(0, T; H^1).
\end{equation*}
 We have
\begin{equation*}
(\rho\sqrt{t}\mathbf{u}_t, \sqrt{t}\mathbf{b}_t, \rho\sqrt{t}\phi_t) \in L^2(0, T; L^q) \quad \text{for} \quad q < \infty.
\end{equation*}
Then, it follows from the interpolation inequality, that
\begin{equation*}
\|\rho\sqrt{t}\mathbf{u}_t\|_{L^p(0,T;L^r)} \leq \|\rho\sqrt{t}\mathbf{u}_t\|_{L^\infty(0,T;L^2)}^{1 - \frac{2}{p}} \|\rho\sqrt{t}\mathbf{u}_t\|_{L^2(0,T;L^q)}^{\frac{2}{p}},
\end{equation*}
\begin{equation*}
\|\rho\sqrt{t}\mathbf{b}_t\|_{L^p(0,T;L^r)} \leq \|\rho\sqrt{t}\mathbf{b}_t\|_{L^\infty(0,T;L^2)}^{1 - \frac{2}{p}} \|\rho\sqrt{t}\mathbf{b}_t\|_{L^2(0,T;L^q)}^{\frac{2}{p}},
\end{equation*}
\begin{equation*}
\|\rho\sqrt{t}\phi_t\|_{L^p(0,T;L^r)} \leq \|\rho\sqrt{t}\phi_t\|_{L^\infty(0,T;L^2)}^{1 - \frac{2}{p}} \|\rho\sqrt{t}\phi_t\|_{L^2(0,T;L^q)}^{\frac{2}{p}},
\end{equation*}
with $\frac{1}{r} = \frac{p - 2}{2p} + \frac{2}{pq}$, $2 \leq r < p^*$, $p^*=\frac{2p}{p-2}$.
Thus
\begin{equation}
\|(\rho\sqrt{t}\mathbf{u}_t, \sqrt{t}\mathbf{b}_t, \rho\sqrt{t}\phi_t)\|_{L^p(0,T;L^r)} \leq C_{0,T} \quad \text{for} \quad p \in [2, \infty], \quad r \in [2, p^*).
\end{equation}
And similarly, we have
\begin{equation}
\|(\rho\sqrt{t}\mathbf{u}, \sqrt{t}\mathbf{b}, \rho\sqrt{t}\phi, \rho\sqrt{t}\mu, \rho\sqrt{t}\phi^3)\|_{L^p(0,T;L^r)} \leq C_{0,T} \quad \text{for} \quad p \in [2, \infty], \quad r \in [2, p^*).
\end{equation}
It is known from Theorem \ref{1.1} that $(\nabla \mathbf{u}, \nabla \mathbf{b}, \nabla \phi)$ is bounded in  $L^\infty(0, T; L^2) \cap L^2(0, T; H^1)$. Employing the interpolation inequality, it is easy to get,  for $\frac{1}{r} = \frac{p - 2}{2p} + \frac{2}{pq}$, $2 \leq r < p^*$, that
\begin{align*}
\|\nabla \mathbf{u}\|_{L^p(0,T;L^r)} &\le \|\nabla \mathbf{u}\|_{L^\infty(0,T;L^2)}^{1 - \frac{2}{p}} \|\nabla \mathbf{u}\|_{L^2(0,T;L^q)}^{\frac{2}{p}} \\
&\le \|\nabla \mathbf{u}\|_{L^\infty(0,T;L^2)}^{1 - \frac{2}{p}} \|\nabla \mathbf{u}\|_{L^2(0,T;H^1)}^{\frac{2}{p}}, \\
\|\nabla \mathbf{b}\|_{L^p(0,T;L^r)} &\le \|\nabla \mathbf{b}\|_{L^\infty(0,T;L^2)}^{1 - \frac{2}{p}} \|\nabla \mathbf{b}\|_{L^2(0,T;L^q)}^{\frac{2}{p}} \\
&\le \|\nabla \mathbf{b}\|_{L^\infty(0,T;L^2)}^{1 - \frac{2}{p}} \|\nabla \mathbf{b}\|_{L^2(0,T;H^1)}^{\frac{2}{p}}, \\
\|\nabla \phi\|_{L^p(0,T;L^r)} &\le \|\nabla \phi\|_{L^\infty(0,T;L^2)}^{1 - \frac{2}{p}} \|\nabla \phi\|_{L^2(0,T;L^q)}^{\frac{2}{p}} \\
&\le \|\nabla \phi\|_{L^\infty(0,T;L^2)}^{1 - \frac{2}{p}} \|\nabla \phi\|_{L^2(0,T;H^1)}^{\frac{2}{p}},
\end{align*}
which implies that
\begin{equation}
\|(\nabla \mathbf{u}, \nabla \mathbf{b}, \nabla \phi)\|_{L^p(0,T;L^r)} \le C_{0,T} \quad \text{for} \quad p \ge 2, \ r < p^*.
\end{equation}
Similarly, since $\nabla^2\phi$ is bounded in $L^\infty(0, T; L^2) \cap L^2(0, T; L^6)$, we get
\begin{align*}
\|\nabla ^2\phi\|_{L^p(0,T;L^r)} &\le \|\nabla ^2\phi\|_{L^\infty(0,T;L^2)}^{1 - \frac{2}{p}} \|\nabla ^2\phi\|_{L^2(0,T;L^q)}^{\frac{2}{p}},
\end{align*}
with \(\frac{1}{r} = \frac{p - 2}{2p} + \frac{2}{pq}\). Here, when \(q\) takes \(6\), then \(r\) may take the maximum value of \(\frac{6p}{3p - 4}\). Then we easily get
\begin{equation}
\|(\nabla^2\phi)\|_{L^p(0,T;L^r)} \le C_{0,T} \quad \text{for} \quad p \ge 2,\quad 2\leq r \leq \frac{6p}{3p - 4}.
\end{equation}
As obvious, we conclude that
\begin{align*}
&\| (\sqrt{t} \rho \mathbf{u} \cdot \nabla \mathbf{u}, \sqrt{t} \mathbf{u} \cdot \nabla \mathbf{b}, \sqrt{t}\mathbf{b} \cdot \nabla \mathbf{u}, \sqrt{t}\mathbf{b} \cdot \nabla \mathbf{b}) \|_{L^p(0,T; L^r)} \leq C_{0,T}\,  \text{for}\,  p \in [2,\infty],  r \in [2,p^*),\\
&\|( \sqrt{t} \rho \mathbf{u} \cdot\nabla \phi, \sqrt{t} \rho \phi \nabla \phi, \sqrt{t}\rho \phi^3 \nabla \phi, \sqrt{t} \rho \mu \nabla \phi ) \|_{L^p(0,T; L^r)} \leq C_{0,T}\, \text{for}\,  p \in [2,\infty],  r \in [2,p^*),\\
&\| (\sqrt{t} \nabla \phi\nabla^2 \phi) \|_{L^p(0,T; L^r)} \leq C_{0,T} \quad \text{for} \quad p \in [2,\infty], \ r \in \big[2, \frac{6p}{3p - 4}\big].
\end{align*}
Applying the maximal regularity estimate for the Stokes equations and the standard estimate for elliptic equations for \eqref{eq-n} yields that
\begin{equation}
\| (\nabla^2 \sqrt{t} \mathbf{u}, \nabla^2 \sqrt{t} \mathbf{b}, \nabla^2 \sqrt{t} \phi, \nabla^2 \sqrt{t} \phi, \nabla \sqrt{t} P ) \|_{L^p(0,T; L^r)} \leq C_{0,T} \quad \text{for} \quad p \in [2,\infty], \ r \in \Big[2, \frac{6p}{3p - 4}\Big].\label{eq-m}
\end{equation}
Furthermore, using the bound for $(\rho \mathbf{u},  \mathbf{b}, \rho \phi, \rho\mu) \in L^\infty(0,T;L^6)$ and when $d =2$, the embedding $W_r^1(\Omega) \hookrightarrow L^q(\Omega)$ with $\frac{3}{q} = \frac{3}{r} - 1$ if $2 \leq r < 3$,  which implies that $(\nabla \sqrt{t} \mathbf{u}, \nabla \sqrt{t} \phi,\nabla \sqrt{t} \mu)$ is bounded in $L^p(0,T;L^q)$, we get \eqref{eq-m} for the full range of indices.
Fix \( p \in [2,\infty) \) so that \( ps < 2(p - s) \) and \( 1 \leq s < 2 \), which means that \( \big( \int_0^T t^{-\frac{ps}{2p - 2s}} dt \big)^{\frac{1}{s} - \frac{1}{p}} \leq C_{0,T} \).
Taking \( r \in \big[2, \frac{6p}{3p - 4}\big] \) such that the embedding \( W_r^1 \hookrightarrow L^\infty \), we conclude that
\begin{equation*}
\begin{split}
&\big( \int_0^T \big\lVert \nabla \mathbf{u} \big\rVert_\infty^s dt \big)^{\frac{1}{s}} + \big( \int_0^T \big\lVert \nabla \mathbf{b} \big\rVert_\infty^s dt \big)^{\frac{1}{s}} + \big( \int_0^T \big\lVert \nabla \phi \big\rVert_\infty^s dt \big)^{\frac{1}{s}}
+\big( \int_0^T \big\lVert \nabla\mu \big\rVert_\infty^s dt \big)^{\frac{1}{s}}  \\
&\leq C \big( \int_0^T \big\lVert \sqrt{t} \nabla \mathbf{u} \big\rVert_{W_r^1}^s \frac{dt}{(\sqrt{t})^s} \big)^{\frac{1}{s}} + C \big( \int_0^T \big\lVert \sqrt{t} \nabla \mathbf{b} \big\rVert_{W_r^1}^s \frac{dt}{(\sqrt{t})^s} \big)^{\frac{1}{s}}\\&\quad + C \big( \int_0^T \big\lVert \sqrt{t} \nabla \phi \big\rVert_{W_r^1}^s \frac{dt}{(\sqrt{t})^s} \big)^{\frac{1}{s}}
+C \big( \int_0^T \big\lVert \sqrt{t} \nabla \mu \big\rVert_{W_r^1}^s \frac{dt}{(\sqrt{t})^s} \big)^{\frac{1}{s}} \\
&\leq C \big( \int_0^T t^{-\frac{ps}{2p - 2s}} dt \big)^{\frac{1}{s} - \frac{1}{p}} \big( \big\lVert \nabla \sqrt{t} \mathbf{u} \big\rVert_{L_p(0,T;W_r^1)} +big\lVert \nabla \sqrt{t} \mathbf{b} \big\rVert_{L_p(0,T;W_r^1)} \\&\quad+ \big\lVert \nabla \sqrt{t} \phi \big\rVert_{L_p(0,T;W_r^1)}+\big\lVert \nabla \sqrt{t} \mu \big\rVert_{L_p(0,T;W_r^1)}  \big) \\
&\leq C_{0} T^{-\frac{2p - 2s - ps}{2ps}},
\end{split}
\end{equation*}
which deduces that \eqref{eq-x} holds.
\vskip .2in
\subsection{Lagrangian approach}
\label{proofT}

In order to prove  the  uniqueness in  Theorems \ref{1.2} and \ref{1.3}, we first show the Lagrangian formulation of the system  \eqref{1.1}.
Now, we introduce the flow \( X : \mathbb{R}^+ \times \Omega \to \Omega\) of \( \mathbf{u} \) by
\begin{equation*}
\partial_t X(t, y) = \mathbf{u}(t, X(t, y)), \quad X(0, y) = y.
\end{equation*}
Note that
\begin{equation*}
X(t, y) = y + \int_0^t \mathbf{u}(\tau, X(\tau, y)) d\tau,
\end{equation*}
and
\begin{equation*}
\nabla_y X(t, y) = Id + \int_0^t \nabla_y \mathbf{u}(\tau, X(\tau, y)) d\tau.
\end{equation*}
In Lagrangian coordinates \((t, y)\), a solution \((\rho, \mathbf{u}, \mathbf{b}, \phi,\mu, P)\) to the system \eqref{1.1} recasts in \((\bar{\rho}, \bar{\mathbf{u}}, \bar{\mathbf{b}}, \bar{\phi},\bar{\mu}, \bar{P})\) with
\begin{equation}
\begin{split}
\bar{\rho}(t, y) &= \rho(t, X(t, y)), \quad \bar{\mathbf{u}}(t, y) = \mathbf{u}(t, X(t, y)), \quad \bar{\mathbf{b}}(t, y) = \mathbf{b}(t, X(t, y)),\\
\bar{\phi}(t, y) &= \phi(t, X(t, y)),\quad \bar{\mu}(t, y) = \mu(t, X(t, y)),\quad \bar{P}(t, y) = P(t, X(t, y)),
\end{split}
\end{equation}
and the triplet \((\bar{\rho}, \bar{\mathbf{u}}, \bar{\mathbf{b}}, \bar{\phi},\bar{\mu},\bar{P})\) thus satisfies
\begin{equation}
\begin{cases}
\bar{\rho} \bar{\mathbf{u}}_t - div_\mathbf{u}(\nu(\bar{\phi})\nabla_\mathbf{u}\bar{\mathbf{u}}+ \nabla_\mathbf{u} \bar{P} = \bar{\mathbf{b}}\cdot\nabla_\mathbf{u}\bar{\mathbf{b}}
-div_\mathbf{u}(\nabla_\mathbf{u}\bar{\mathbf{u}}\otimes\nabla_\mathbf{u}\bar{\mathbf{u}}), \label{eq-z}
\\
\operatorname{div}_\mathbf{u} \bar{\mathbf{u}} = 0 , \\
\bar{\mathbf{b}}_t-\bar{\mathbf{b}}\cdot\nabla_\mathbf{u}\bar{\mathbf{u}} = \Delta_\mathbf{u}\bar{\mathbf{b}},\\
\bar{\rho} \bar{\phi}_t = -\Delta_\mathbf{u}\bar{\phi},&  \\
\bar{\rho}_t = 0 , \\
\bar{\rho} \bar{\mu} = -\Delta_\mathbf{u}\bar{\phi}+\bar{\rho}(\bar{\phi}^3-\bar{\phi}),\\
\bar{\rho}(y, 0) = \rho_0(y), \bar{\mathbf{u}}(y, 0) = u_0(y),  \bar{\mathbf{b}}(y, 0) = b_0(y), \bar{\phi}(y, 0) = \phi_0(y) ,
\bar{\mu}(y, 0) = \mu_0(y),&
\end{cases}
\end{equation}
where operators \(\nabla_\mathbf{u}, \Delta_\mathbf{u}, \nabla_\mathbf{u} \operatorname{div}_\mathbf{u}\) and \(\operatorname{div}_\mathbf{u}\) correspond to the original operators \(\nabla, \Delta, \nabla \operatorname{div}\), respectively, after performing the change to the Lagrangian coordinates.
As pointed out in \cite{RB1, B2, PZZ}, in our regularity framework, that latter system \eqref{eq-z} is equivalent to the system \eqref{1.1}. Thanks to \eqref{eq-s} and \eqref{eq-x}, we can take the time $T$ to be small enough so that
\begin{equation}
\int_0^T \|\nabla \mathbf{u}\|_\infty d\tau \leq \frac{1}{2}.\label{eq-8}
\end{equation}
Set
\begin{align*}
A &= (\nabla X)^{-1} \text{(inverse of deformation tensor)}, \\
J &= \det \nabla X \text{(Jacobian determinant)}, \\
a &= JA \text{(transpose of cofactor matrix)}.
\end{align*}
Thus, in the \((t, y)\)-coordinates, operators \(\nabla\), \(\operatorname{div}\) and \(\Delta\) translate into
\begin{equation}
\nabla_u := {}^T\!A \nabla_y, \quad \operatorname{div}_\mathbf{u} := \operatorname{div}_y(A \cdot), \quad \text{and} \quad \Delta_\mathbf{u} := \operatorname{div}_\mathbf{u}\nabla_\mathbf{u}. \label{eq-6}
\end{equation}
Moreover, given some matrix \( N \), we define the divergence operator (acting on vector fields \( v \)) by the formulation
\begin{equation}
\operatorname{div}_\mathbf{u} N v = \operatorname{div}_y(N \cdot v) \stackrel{\text{def}}{=} {}^T\!N : \nabla v,
\end{equation}
where \( N : B = \sum_{i,j} N_{ij} B_{ji} \) for \( N = (N_{ij})_{1 \leq i,j \leq d} \) and \( B = (B_{ij})_{1 \leq i,j \leq d} \) two \( d \times d \) matrices.
Of course, if the condition \eqref{eq-8} is fulfilled then we have
\begin{equation}
A = \Big( Id + (\nabla_y X - Id) \Big)^{-1} = \sum_{k=0}^{+\infty} (-1)^k \big( \int_0^t \nabla_y \bar{\mathbf{u}}(\tau, \cdot) d\tau \big)^k,\label{eq-9}
\end{equation}
which yields that
\begin{equation}
\delta A = \big( \int_0^t \nabla \delta \mathbf{u} d\tau \big) \cdot \big( \sum_{k \geq 1} \sum_{0 \leq j < k} C_1^j C_2^{k - 1 - j} \big) \quad \text{with} \quad C_i(t) = \int_0^t \nabla \bar{\mathbf{u}}^i d\tau, \label{eq-7}
\end{equation}
where \( \delta A \stackrel{\text{def}}{=} A_2 - A_1 \) and \( \delta \mathbf{u} \stackrel{\text{def}}{=} \bar{\mathbf{u}}^2 - \bar{\mathbf{u}}^1 \). From \eqref{eq-8} and the convergence property of series \eqref{eq-9}, we known that $A$ is bounded.
Finally, we shall prove   Theorem \ref{1.2}. Let $(\rho^1, \mathbf{u}^1, \mathbf{b}^1, \phi^1,\mu^1, P^1)$  and $(\rho^2, \mathbf{u}^2, \mathbf{b}^2, \phi^2,\mu^2,P^2)$ be two solutions of the system \eqref{1.1} fulfilling the properties of Theorem \ref{1.1} with the same initial data, and denote by \((\bar{\rho}^1, \bar{\mathbf{u}}^1, \bar{\mathbf{b}}^1, \bar{\phi}^1,\bar{\mu}^1,  \bar{P}^1)\) and \((\bar{\rho}^2, \bar{\mathbf{u}}^2, \bar{\mathbf{b}}^2, \bar{\phi}^2, \bar{\mu}^2, \bar{P}^2)\) in Lagrangian coordinates. Of course, we have \(\bar{\rho}^1 = \bar{\rho}^2 = \rho_0\), which explains the choice of our approach here. In what follows, we shall use repeatedly the fact that for \( i = 1,2 \),
\begin{equation}
\begin{split}\label{eq-5}
&t^{\frac{1}{2}} \nabla \bar{\mathbf{u}}^i \in L^2(0,T; L^\infty), t^{\frac{1}{2}} \nabla \bar{\mathbf{b}}^i \in L^2(0,T; L^\infty),
t^{\frac{1}{2}} \nabla \bar{\phi}^i \in L^2(0,T; L^\infty), \\
&t^{\frac{1}{2}} \nabla \bar{\mu}^i \in L^2(0,T; L^\infty),
t^{\frac{1}{2}} \nabla \bar{P}^i \in L^2(0,T; L^4),
t^{\frac{1}{2}} \bar{\mathbf{u}}_t^i \in L^{4/3}(0,T; L^6), \\
&\nabla \bar{\mathbf{u}}^i \in L^1(0,T; L^\infty) \cap L^2(0,T; L^6) \cap L^4(0,T; L^3),\, \nabla \bar{\phi}^i \in L^1(0,T; L^\infty) \cap L^\infty(0,T; w^{1,6}) ,\\
&\nabla \bar{\mu}^i \in L^1(0,T; L^\infty) \cap L^2(0,T; w^{1,6}) ,
\bar{\mathbf{u}}^i \in L^4(0,T; L^\infty),\bar{\phi}^i \in L^4(0,T; L^\infty),\bar{\mu}^i \in L^4(0,T; L^\infty),\\
&t^{\frac{1}{2}} \bar{\mathbf{b}}_t^i \in L^{4/3}(0,T; L^6), \, \nabla \bar{\mathbf{b}}^i \in L^1(0,T; L^\infty) \cap L^2(0,T; L^6) \cap L^4(0,T; L^3), \bar{\mathbf{b}}^i \in L^4(0,T; L^\infty).
\end{split}
\end{equation}
It should be noted that the terms in \eqref{eq-5} is less than or equal to \( c(T) \), where \( c(T) \) designates a nonnegative continuous increasing function of \( T \), with \( c(0) = 0 \) and \( c(T) \to 0 \) when \( T \to 0 \). By Theorem \ref{1.1}, \eqref{eq-s} and $\|A_i\|_{\infty} < \infty\ (i = 1, 2)$, we get
$$\nabla \bar{\phi}^i \in L^1(0,T; L^\infty) \cap L^\infty(0,T; w^{1,6}) ,\quad
\nabla \bar{\mu}^i \in L^1(0,T; L^\infty) \cap L^2(0,T; w^{1,6}).$$
Denoting
\begin{align*}
\delta u \stackrel{\text{def}}{=} \bar{\mathbf{u}}^2 - \bar{\mathbf{u}}^1, \quad
\delta b \stackrel{\text{def}}{=} \bar{\mathbf{b}}^2 - \bar{\mathbf{b}}^1, \quad
\delta \phi \stackrel{\text{def}}{=} \bar{\phi}^2 - \bar{\phi}^1, \quad \delta\mu \stackrel{\text{def}}{=} \bar{\mu}^2 - \bar{\mu}^1
\quad\text{and} \quad \delta P \stackrel{\text{def}}{=} \bar{P}^2 - \bar{P}^1,
\end{align*}
we get
\begin{equation}
\begin{cases}
\rho_0 \, \delta \mathbf{u}_t - \mathop{\mathrm{div}_{\mathbf{u}^1}} \left( \nu(\bar{\phi}^1) \cdot \nabla_{\mathbf{u}^1} \delta \mathbf{u} + \nabla_{\mathbf{u}^1} \delta P \right.
- \rho_0 \, \bar{\mu}^1 \, \nabla_{\mathbf{u}^1} \delta \phi-\bar{\mathbf{b}^1}\cdot\nabla_{\mathbf{u}^1} \delta \mathbf{b}\\
+ \rho_0 \left( |\bar{\phi}^1|^3 - \bar{\phi}^1 \right) \, \nabla_{\mathbf{u}^1} \delta \phi
+ \nabla_{\mathbf{u}^1}^2 \bar{\phi}^1 \cdot \nabla_{\mathbf{u}^1} \delta \phi = \delta f_1, \\
\mathop{\mathrm{div}_{\mathbf{u}^1}} \delta \mathbf{u} = (\mathop{\mathrm{div}_{\mathbf{u}^1}} - \mathop{\mathrm{div}_{\mathbf{u}^2}}) \bar{\mathbf{u}}^2, \\
\delta \mathbf{b}_t -\bar{\mathbf{b}^1}\cdot\nabla_{\mathbf{u}^1} \delta \mathbf{u}-\Delta_{\mathbf{u}^1}\delta \mathbf{b}=\delta f_4,\\
\rho_0 \, \delta \phi_t - \Delta_{\mathbf{u}^1} \delta \mu = \delta f_2, \\
\rho_0 \, \delta \mu + \Delta_{\mathbf{u}^1} \delta \phi - \rho_0 \left( \delta \phi^3 - \delta \phi \right) = \delta f_3, \\
(\delta \mathbf{u}, \delta \mathbf{b}, \delta \phi, \delta \mu) \big\vert_{t=0}= (0, 0)
\end{cases}
\end{equation}
with
\begin{align*}
&\delta f_1 \stackrel{\text{def}}{=} \big[ \mathrm{div}_{\mathbf{u}^2} (\nu(\bar{\phi}^1) \nabla _{\mathbf{u}^2} - \mathrm{div}_{\mathbf{u}^1}(\nu(\bar{\phi}^1) \nabla_{\mathbf{u}^1}) \big] \bar{\mathbf{u}}^2 - (\nabla_{\mathbf{u}^2} - \nabla_{\mathbf{u}^1}) \bar{P}^2 + (\rho_0 \bar{\mu}^2 \nabla_{\mathbf{u}^2} - \rho_0 \bar{\mu}^1 \nabla_{\mathbf{u}^1}) \bar{\phi}^2\\
&+(\bar{\mathbf{b}}^2 \nabla_{\mathbf{u}^2}-\bar{\mathbf{b}}^1 \nabla_{\mathbf{u}^1})\bar{\mathbf{b}}^2 - \rho_0 \big( |\bar{\phi}^2|^3 \nabla_{\mathbf{u}^2} - |\bar{\phi}^1|^3 \nabla_{\mathbf{u}^1} \big) \bar{\phi}^2 + \rho_0 \big( \bar{\phi}^2 \nabla_{\mathbf{u}^2} - \bar{\phi}^1 \nabla_{\mathbf{u}^1}\big) \bar{\phi}^2 - \big( \nabla^2_{\mathbf{u}^2} \bar{\phi}^2 \nabla_{\mathbf{u}^2} - \nabla^2_{\mathbf{u}^1} \bar{\phi}^1 \nabla_{\mathbf{u}^1}\big) \bar{\phi}^2
\end{align*}
\begin{align*}
\delta f_2 \stackrel{\text{def}}{=} (\Delta_{\mathbf{u}^2} - \Delta_{\mathbf{u}^1}) \bar{\mu}^2, \quad \delta f_3 \stackrel{\text{def}}{=} -(\Delta_{\mathbf{u}^2} - \Delta_{\mathbf{u}^1}) \bar{\phi}^2,  \quad \delta f_4 \stackrel{\text{def}}{=} (\bar{\mathbf{b}}^2 \nabla_{\mathbf{u}^2}-\bar{\mathbf{b}}^1 \nabla_{\mathbf{u}^1})\bar{\mathbf{u}}^2+(\Delta_{\mathbf{u}^2} - \Delta_{\mathbf{u}^1}) \bar{\mathbf{b}}^2.
\end{align*}
We claim for sufficiently small \( T > 0 \)
\begin{equation*}
\begin{split}
\int_0^T &\int_{\Omega} \bigl(|\delta \mathbf{u}(t, y)|^2 +|\delta \mathbf{b}(t, y)|^2 +  |\delta \phi(t, y)|^2+|\delta \mu(t, y)|^2 \\&+ |\nabla \delta \mathbf{u}(t, y)|^2 +|\nabla \delta \mathbf{b}(t, y)|^2  + |\nabla \delta \phi(t, y)|^2\bigr) \, dy \, dt = 0.
\end{split}
\end{equation*}
To prove our claim, we first decompose \( \delta \mathbf{u} \) into:
\begin{align}
\delta \mathbf{u} = \varphi + \psi,
\end{align}
where \( \varphi \) is the solution given by Lemma \ref{lemma:4.1} to the following problem:
\begin{align}
\mathop{\mathrm{div}_{\mathbf{u}^1}} \varphi = (\mathop{\mathrm{div}_{\mathbf{u}^1}} - \mathop{\mathrm{div}_{\mathbf{u}^2}}) \bar{\mathbf{u}}^2 \nonumber= \mathop{\mathrm{div}}(\delta A \bar{\mathbf{u}}^2).
\end{align}
Then \eqref{4.9} and \eqref{eq-9} ensure that there exist two universal positive constants \( c \) and \( C \) such that if
\begin{align}
\|\nabla \bar{\mathbf{u}}^1\|_{L^1(0,T; L^\infty)} + \|\nabla \bar{\mathbf{u}}^1\|_{L^2(0,T; L^6)} &\leq c,\label{eq-q}
\end{align}
and then the following inequalities hold true:
\begin{align}
&\|\varphi\|_{L^4(0,T;L^2)} \leq C\|\delta A \bar{\mathbf{u}}^2\|_{L^4(0,T;L^2)}, \|\nabla \varphi\|_{L^2(0,T;L^2)} \leq C\|^T \delta A : \nabla \bar{\mathbf{u}}^2\|_{L^2(0,T;L^2)}, \notag\label{eq-a}, \\
&\|\varphi_t\|_{L^{4/3}(0,T;L^{3/2})} \leq C\|\delta A \bar{\mathbf{u}}^2\|_{L^4(0,T;L^2)} + C\|(\delta A \bar{\mathbf{u}}^2)_t\|_{L^{4/3}(0,T;L^{3/2})}.
\end{align}
Now, let us bound these terms in the right hand side  of \eqref{eq-a}. Regarding $^T \delta A : \nabla \bar{\mathbf{u}}^2$, it follows from H\"{o}lder's inequality, \eqref{eq-7} and \eqref{eq-q} that
\begin{align}
\sup_{t \in [0,T]} \|t^{-1/2} \delta A\|_2 &\leq C \sup_{t \in [0,T]} \big\|t^{-1/2} \int_0^t \nabla \delta \mathbf{u} d\tau \big\|_2 \leq C\|\nabla \delta \mathbf{u}\|_{L^2(0,T;L_2)}.\label{eq-4}
\end{align}
According to \eqref{eq-5} and \eqref{eq-4}, we obtain
\begin{align}
\|^T \delta A : \nabla \bar{\mathbf{u}}^2\|_{L^2(0,T \times \mathbb{T}^d)} &\leq \sup_{t \in [0,T]} \|t^{-1/2} \delta A\|_2 \|t^{1/2} \nabla \bar{\mathbf{u}}^2\|_{L^2(0,T;L^\infty)} \notag \\
&\leq c(T)\|\nabla \delta \mathbf{u}\|_{L^2(0,T;L^2)}.
\end{align}
Similarly, we also have
\begin{align}
\|\delta A \bar{\mathbf{u}}^2\|_{L^4(0,T;L^2)} &\leq \|t^{-1/2} \delta A\|_{L^\infty(0,T;L^2)} \|t^{1/2} \bar{\mathbf{u}}^2\|_{L^4(0,T;L^\infty)}.
\end{align}
Using \eqref{eq-5}, \eqref{eq-a} and \eqref{eq-4} yields that
\begin{align}
\|\nabla \varphi\|_{L^2(0,T;L^2)} &\leq c(T)\|\nabla \delta \mathbf{u}\|_{L^2(0,T;L^2)},\label{eq-b}
\end{align}
and
\begin{align}
\|\varphi\|_{L^4(0,T;L^2)} &\leq c(T)\|\nabla \delta \mathbf{u}\|_{L^2(0,T;L^2)}.
\label{eq-c}
\end{align}
In order to bound $\varphi_t$, it suffices to derive an appropriate estimate in $L^{4/3}(0,T;L^{3/2})$ for
\begin{align*}
(\delta A \bar{\mathbf{u}}^2)_t = \delta A \bar{\mathbf{u}}_t^2 + (\delta A)_t \bar{\mathbf{u}}^2.
\end{align*}
Thanks to \eqref{eq-5} and \eqref{eq-4}, we have
\begin{align*}
\|\delta A \bar{\mathbf{u}}_t^2\|_{L^{4/3}(0,T;L^{3/2})} &\leq \|t^{-1/2} \delta A\|_{L^\infty(0,T;L^2)} \|t^{1/2} \nabla \bar{\mathbf{u}}_t^2\|_{L^{4/3}(0,T;L^6)} \notag \\
&\leq c(T)\|\nabla \delta \mathbf{u}\|_{L^2(0,T;L^2)}.
\end{align*}
For the term $(\delta A)_t \bar{\mathbf{u}}^2$, it follows from Holder's inequality that
\begin{align*}
\|(\delta A)_t \bar{\mathbf{u}}^2\|_{L^{4/3}(0,T;L^{3/2})} &\leq \|(\delta A)_t\|_{L^2(0,T \times \mathbb{T}^d)} \|\bar{\mathbf{u}}^2\|_{L^4(0,T;L^6)}.
\end{align*}
Furthermore, differentiating \eqref{eq-7} with respect to \( t \) and using \eqref{eq-q} for \( \bar{\mathbf{u}}^1 \) and \( \bar{\mathbf{u}}^2 \) yield that
\begin{align*}
\| (\delta A)_t \|_2 &\leq C \Big( \| \nabla \delta \mathbf{u} \|_2 + \big\| t^{-1/2} \int_0^t \nabla \delta \mathbf{u} \, dr \big\|_2 \big( \| t^{1/2} \nabla \bar{\mathbf{u}}^1 \|_\infty + \| t^{1/2} \nabla \bar{\mathbf{u}}^2 \|_\infty \big) \Big),
\end{align*}
which implies that  $$ \| (\delta A)_t \|_{L^2(0,T \times \mathbb{T}^d)} \leq C \| \nabla \delta \mathbf{u} \|_{L^2(0,T \times \mathbb{T}^d)}.$$
Combining this with \eqref{eq-5} yields  that
\begin{align}
\| (\delta A)_t \bar{\mathbf{u}}^2 \|_{L^{4/3}(0,T; L^{3/2})} &\leq c(T) \| \nabla \delta \mathbf{u} \|_{L^2(0,T \times \mathbb{T}^d)},
\label{eq-3}\\ \notag
\text{and thus} \quad \| \varphi_t \|_{L^{4/3}(0,T; L^{3/2})} &\leq c(T) \| \nabla \delta \mathbf{u} \|_{L^2(0,T; L^2)}.
\end{align}
Collecting results from \eqref{eq-a}, \eqref{eq-b}, \eqref{eq-c} and \eqref{eq-3}, we obtain
\begin{equation}
\| \varphi \|_{L^4(0,T; L^2)} + \| \nabla \varphi \|_{L^2(0,T \times \mathbb{T}^d)} + \| \varphi_t \|_{L^{4/3}(0,T; L^{3/2})} \leq c(T) \| \nabla \delta \mathbf{u} \|_{L^2(0,T \times \mathbb{T}^d)}.\label{eq-1}
\end{equation}
Next, let us restate the equations for $(\delta \mathbf{u}, \delta \mathbf{b}, \delta \phi,\delta \mu, \delta P)$ as the following system for $(\psi, \delta \mathbf{b}, \delta \phi,\delta\mu, \delta P)$:
\begin{equation}
\begin{cases}
\rho_0 \,\psi_t - \mathop{\mathrm{div}_{\mathbf{u}^1}} \left( \nu(\bar{\phi}^1) \, \nabla_{\mathbf{u}^1} \psi + \nabla_{\mathbf{u}^1} \delta P \right.
= \delta f_1-\rho_0 \,\varphi_t+\mathop{\mathrm{div}_{\mathbf{u}^1}} \left( \nu(\bar{\phi}^1) \, \nabla_{\mathbf{u}^1} \varphi \right)+\rho_0 \, \bar{\mu}^1 \, \nabla_{\mathbf{u}^1} \delta \phi, \label{eq-10}\\
\bar{\mathbf{b}}^1\nabla_{\mathbf{u}^1} \delta\mathbf{b}- \rho_0 \left( |\bar{\phi}^1|^3 + \bar{\phi}^1 \right) \, \nabla_{\mathbf{u}^1} \delta \phi
-\nabla_{\mathbf{u}^1}^2 \bar{\phi}^1 \, \nabla_{\mathbf{u}^1} \delta \phi, \\
\mathop{\mathrm{div}_{\mathbf{u}^1}} \psi = 0, \\
\delta \mathbf{b}_t -\bar{\mathbf{b}^1}\cdot\nabla_{\mathbf{u}^1} \mathbf{\psi}-\Delta_{\mathbf{u}^1}\delta \mathbf{b}=\bar{\mathbf{b}}^1\nabla_{\mathbf{u}^1} \mathbf{\psi}+\delta f_4,\\
\rho_0 \, \delta \phi_t - \Delta_{\mathbf{u}^1} \delta \mu = \delta f_2, \\
\rho_0 \, \delta \mu + \Delta_{\mathbf{u}^1} \delta \phi - \rho_0 \left( \delta \phi^3 - \delta \phi \right) = \delta f_3, \\
(\Delta \mathbf{u}, \delta \phi, \delta \mu) \big\vert_{t=0}= (0, 0).
\end{cases}
\end{equation}
Due to $\operatorname{div}_{\mathbf{u}^1} \psi = 0$, we have
\[
\int_{\Omega} (\nabla_{\mathbf{u}^1} \delta P) \cdot \phi \, dx = - \int_{\Omega} \operatorname{div}_{\mathbf{u}^1} \phi \cdot \delta P \, dx = 0.
\]
Taking the \( L^2 \)-scalar product of the first equation in system \eqref{eq-10} with \( \psi \) , the third equation with \( \delta\mathbf{b} \), the fourth equation with \( \delta \phi \) and the fifth equation with \( \delta \mu \) respectively, finally adding the three equations together, we infer that
\begin{equation}
\begin{split}
&\frac{1}{2} \frac{d}{dt} \int_{\Omega}|\delta\mathbf{b}|^2+ \rho_0 \big( |\psi|^2 + |\delta \phi|^2 \big) dx + \int_{\Omega} \big( \nu(\bar{\phi}^1) |\nabla_{\mathbf{u}^1} \psi|^2 +|\nabla_{\mathbf{u}^1}\delta\mathbf{b}|^2 + \rho_0 |\ \delta \mu|^2 \big) dx \\
&\leq -\int_\Omega \rho_0 \partial_t \varphi \cdot \psi \, dx + \int_\Omega \Dv_{\mathbf{u}^1} (\nu(\bar{\phi^1}) \nabla_{\mathbf{u}^1} \varphi) \cdot \psi \, dx + \int_\Omega \rho_0 \bar{\mu}_1 \nabla_{\mathbf{u}^1} \delta\phi \, \psi \, dx \\
&\quad-\int_\Omega \rho_0 (|\bar{\phi}^1|^3 - \bar{\phi}^1) \nabla_{\mathbf{u}^1} \delta\phi \, \psi \, dx - \int_\Omega\nabla^2_{\mathbf{u}^1} \bar{\phi^1} \nabla_{\mathbf{u}^1} \delta\phi\, \psi \, dx\\
&\quad+ \int_\Omega (\nu(\bar{\phi^1}) \nabla_{\mathbf{u}^2} \bar{\mathbf{u}}^2 \nabla_{\mathbf{u}^2}\psi - \nu(\bar{\phi^1}) \nabla_{\mathbf{u}^1} \bar{u}^2 \nabla_{u^1}\psi) \, dx \\
&\quad-\int_\Omega (\nabla_{\mathbf{u}^2} - \nabla_{\mathbf{u}^1}) \bar{P}^2 \, \psi \, dx + \int_\Omega (\rho_0 \bar{\mu}^2 \nabla_{\mathbf{u}^2} - \rho_0 \bar{\mu}^1 \nabla_{\mathbf{u}^1}) \bar{\phi}^2 \, \psi \, dx \\
&\quad-\int_\Omega \rho_0 (|\bar{\phi}^2|^3 \nabla_{\mathbf{u}^2} - |\bar{\phi}^1|^3 \nabla_{\mathbf{u}^1}) \bar{\phi}^2 \, \psi \, dx + \int_\Omega \rho_0 (\bar{\phi}^2 \nabla_{\mathbf{u}^2} - \bar{\phi}^1 \nabla_{\mathbf{u}^1}) \bar{\phi}^2 \, \psi \, dx \\
&\quad-\int_\Omega (\nabla_{\mathbf{u}^2}^2 \bar{\phi}^2 \nabla_{\mathbf{u}^2} - \nabla_{\mathbf{u}^1}^2 \bar{\phi}^1 \nabla_{\mathbf{u}^1}) \bar{\phi}^2 \, \psi \, dx + \int_\Omega \Delta_{\mathbf{u}^1} \delta\mu \delta\phi \, dx + \int_\Omega (\Delta_{\mathbf{u}^2} - \Delta_{\mathbf{u}^1}) \bar{\mu}^2 \delta\phi \, dx \\
&\quad-\int_\Omega \Delta_{\mathbf{u}^1} \delta\phi \delta \mu \, dx + \int_\Omega \rho_0 (\delta\phi^3 - \delta\phi) \delta \mu \, dx - \int_\Omega (\Delta_{\mathbf{u}^2} - \Delta_{\mathbf{u}^1}) \bar{\phi}^2 \delta \mu \, dx\\
&\quad+\int_\Omega (\bar{\mathbf{b}}^2\nabla_{\mathbf{u}^2}-\bar{\mathbf{b}}^1\nabla_{\mathbf{u}^1}) \bar{\mathbf{b}}^2\cdot \mathbf{\psi}dx +\int_\Omega (\bar{\mathbf{b}}^2\nabla_{\mathbf{u}^2}-\bar{\mathbf{b}}^1\nabla_{\mathbf{u}^1}) \bar{\mathbf{u}}^2\cdot \mathbf{\psi}dx\\
&\quad+\int_\Omega \bar{\mathbf{b}}^1\cdot\nabla_{\mathbf{u}^1}\mathbf{\psi}\cdot\mathbf{\psi}dx+
\int_\Omega\nabla_{\mathbf{u}^1}\bar{\mathbf{b}}^2\cdot\nabla_{\mathbf{u}^1} \delta\mathbf{b}dx-\int_\Omega\nabla_{\mathbf{u}^2}\bar{\mathbf{b}}^2
\cdot\nabla_{\mathbf{u}^2} \delta\mathbf{b}dx
\\&\triangleq\sum_{i = 1}^{21} M_i
\end{split}
\end{equation}
Here and in what follows, we shall estimate term by term above. For \( M_1 \), it follows from Holder's inequality that
\begin{align*}
\int_0^T M_1(t)dt \leq \| \rho_0 \|_\infty^{3/4} \| \varphi_t \|_{L^{4/3}(0,T; L^{3/2})} \| \rho_0^{1/4} \psi \|_{L^4(0,T; L^3)}.
\end{align*}
Using Holder's inequality and the Sobolev embedding \( H^1(\Omega) \hookrightarrow L^6(\Omega) \) yields that
\begin{align*}
\| \rho_0^{1/4} \psi \|_{L^4(0,T; L^3)} \leq \| \sqrt{\rho_0} \psi \|_{L^\infty(0,T; L^2)}^{1/2} \| \psi \|_{L^2(0,T; L^6)}^{1/2} \leq C \| \sqrt{\rho_0} \psi \|_{L^\infty(0,T; L^2)}^{1/2} \| \psi \|_{L^2(0,T; H^1)}^{1/2}.
\end{align*}
Taking advantage of \eqref{eq-7} and \eqref{eq-1}, we conclude that
\begin{align*}
\int_0^T M_1(t)dt \leq c(T) \big( \| \sqrt{\rho_0} \psi \|_{L^\infty(0,T; L^2)} + \| \nabla \psi \|_{L^2(0,T \times \Omega} \big)^{1/2} \| \sqrt{\rho_0} \psi \|_{L^\infty(0,T; L^2)}^{1/2} \| \nabla \delta \mathbf{u} \|_{L^2(0,T \times \Omega)}.
\end{align*}
For \( M_2 \), it follows from integrating by parts and using \eqref{eq-1}, that
\begin{align*}
\int_0^T M_2(t)dt &\leq \nu^* \int_0^T \big| \int_{\Omega} \nabla_{\mathbf{u}^1} \varphi \nabla_{\mathbf{u}^1} \psi dx \big| dt \\
&\leq \nu^* \int_0^T \int_{\Omega} |\nabla_{u^1} \varphi| |\nabla_{\mathbf{u}^1} \psi| dxdt \\
&\leq \frac{\nu^*}{2} \int_0^T \|\nabla_{\mathbf{u}^1}\psi\|_2^2 dt + \frac{\nu^*}{2} \int_0^T \|\nabla_{\mathbf{u}^1} \varphi\|_2^2 dt \\
&\leq \frac{\nu^*}{2} \int_0^T \|\nabla_{\mathbf{u}^1} \psi\|_2^2 dt + c(T) \int_0^T \|\nabla \delta \mathbf{u}\|_2^2 dt.
\end{align*}
For \( M_3 \) and \( M_4 \),  it follows from H\"{o}lder's inequality that
\begin{align*}
\int_{0}^{t} M_{3}(t) dt &\leq \|\sqrt\rho_0 \psi\|_{L^{\infty}(0,T; L^2(\Omega)} \big( \|\nabla \bar{\phi^1}\|_{L^2(0,T; L^4(\Omega)} + \|\nabla \bar{\phi^2}\|_{L^2(0,T; L^4(\Omega)} \big) \|\bar{\mu^1}\|_{L^2(0,T; L^4(\Omega)}, \\
\int_{0}^{t} M_{4}(t) dt &\leq \|\sqrt\rho_0 \psi\|_{L^{\infty}(0,T; L^2(\Omega))} \big( \|\nabla \bar{\phi^1}\|_{L^2(0,T; L^4(\Omega))} + \|\nabla \bar{\phi^2}\|_{L^2(0,T; L^4(\Omega)} \big)\\
&\quad\times\big( \||\nabla \bar{\phi^1}|^{3}\|_{L^2(0,T; L^4(\Omega)} + \|\bar{\phi^1}\|_{L^2(0,T; L^4(\Omega)} \big).
\end{align*}
Thanks to \eqref{eq-8} and \eqref{eq-9}, we get
\begin{align*}
\| \nabla A \|_{L^\infty (0,T; L^3(\Omega)} &\leq C \| \nabla^2 \mathbf{u}(t, X(\tau, \cdot)) \|_{L^1(0,T; L^3(\Omega)} \| \nabla_y X \|_{L_t^\infty (L_y^\infty)} \leq C t^{\frac{1}{4}},\\
\int_{0}^{t} M_{5}(t) dt &\leq\int_{0}^{t} \bigl\lvert (A_1^T \nabla A_1^T \bar{\phi^1} + |A_1^T|^2 \nabla^2 \bar{\phi^1}) \, |A_1^T \psi \, (|\nabla \bar{\phi^2}| + |\nabla \bar{\phi^1}|) \bigr\rvert \, dt \\
&\leq C \rho_0 \,\|\sqrt\rho_0 \psi\|_{L^\infty(0,T; L^2(\Omega)}\big(\|\nabla A_1^T\|_{L^\infty(0,T; L^3(\Omega)} \,\|A_1^T\|_{L^\infty(0,T; L^\infty(\Omega)}^2
\|\bar{\phi^1}\|_{L^\infty(0,T\times\Omega)}\\
&\quad+\|A_1^T\|_{L^\infty(0,T; L^\infty(\Omega)}^3\,
\|\nabla^2 \bar{\phi^1}\|_{L^2(0,T; L^3(\Omega)} \,\big)\big( \|\nabla \bar{\phi^2}\|_{L^2(0,T; L^{6}(\Omega)} + \|\nabla \bar{\phi^1}\|_{L^2(0,T; L^{6}(\Omega)} \big).
\end{align*}
For \( M_6 \), using \eqref{eq-6} and \eqref{eq-7}, we deduce that
\begin{align*}
M_6 &\leq C \big| \int_{\Omega}\Dv\big((\delta A^T A_2 + A_1^T \delta A)\nabla \bar{\mathbf{u}}^2\big) \cdot \psi \, dx \big| \\
&\leq C \int_{\Omega} \big| \delta A^T A_2 + A_1^T \delta A \big| \| \nabla \bar{\mathbf{u}}^2 \| |\nabla \psi| \, dx \\
&\leq C \| t^{-1/2} \delta A \|_2 \| t^{1/2} \nabla \bar{\mathbf{u}}^2 \|_\infty \| \nabla \psi \|_2,
\end{align*}
which together with \eqref{eq-5} and \eqref{eq-4} implies that
\begin{align*}
\int_0^T M_6(t) dt &\leq \| t^{-1/2} \delta A \|_{L^\infty(0,T; L^2)} \| t^{1/2} \nabla \bar{\mathbf{u}}^2 \|_{L^2(0,T; L^\infty)} \| \nabla \psi \|_{L^2(0,T \times \Omega)} \\
&\leq c(T) \| \nabla \delta \mathbf{u} \|_{L^2(0,T; L^2)} \| \nabla \psi \|_{L^2(0,T \times \Omega)}.
\end{align*}
For \( M_7 \), using Hl\"{o}der's inequality, we obtain
\begin{align*}
M_7(t) \leq \big| \int_{\Omega} \delta A \nabla \bar{P}^2 \psi \, dx \big| \leq C \| t^{-1/2} \delta A \|_2 \| t^{1/2} \nabla \bar{P}^2 \|_3 \| \psi \|_6.
\end{align*}
It then follows from \eqref{eq-5}, \eqref{eq-4} and Sobolev embedding that
\begin{align*}
\int_0^T M_7(t) dt &\leq \| t^{-1/2} \delta A \|_{L^\infty(0,T; L^2)} \| t^{1/2} \nabla \bar{P}^2 \|_{L^2(0,T; L^3)} \| \psi \|_{L^2(0,T; H^1)} \\
&\leq C(T) \| \nabla \delta \mathbf{u}\|_{L^2(0,T; L^2)} \big( \| \sqrt{\rho_0} \psi \|_{L^\infty(0,T; L^2(\Omega))} + \| \nabla \psi \|_{L^2(0,T \times \Omega)} \big),\\
\int_{0}^{T} M_8(t) \, dt &\leq \|\rho_0 \psi\|_{L^\infty(0,T; L^2(\Omega)} \, (\|\bar{\mu}^2\|_{L^2(0,T; L^4(\Omega))} \|\nabla A_2^T\|_{L^\infty(0,T\times \Omega)}
\\ &\quad+ \|\bar{\mu}^1\|_{L^2(0,T; L^4(\Omega))} \|\nabla A_1^T\|_{L^\infty(0,T\times \Omega)})\|\nabla\bar{\phi}^2\|_{L^2(0,T; L^4(\Omega)},\\
\int_0^T M_9(t) \, dt &\leq \| \rho_0 \psi \|_{L^\infty(0,T;L^2(\Omega)} \big( \| A_2^T \|_{L^\infty(0,T\times\Omega)} \| |\bar{\phi^2}|^3 \|_{L^2(0,T;L^4(\Omega)} \big). \\
&\quad+ \| A_1^T \|_{L^\infty(0,T\times\Omega)} |\bar{\phi^1}|^3 \|_{L^2(0,T;L^4(\Omega)} \big( \| \nabla \bar{\phi^2}\|_{L^2(0,T;L^4(\Omega)} \big) \\
\int_0^T M_{10}(t) \, dt &\leq \|\rho_0 \psi \|_{L^\infty(0,T;L^2(\Omega)} \big( \| A_2^T \|_{L^\infty(0,T\times\Omega)} \| \bar{\phi^2} \|_{L^2(0,T;L^4(\Omega)} \big).\\
&\quad+ \| A_1^T \|_{L^\infty(0,T\times\Omega)} \|\bar{\phi^1} \|_{L^2(0,T;L^4(\Omega)} \big( \| \nabla \bar{\phi^2} \|_{L^2(0,T;L^4(\Omega)} \big), \\
\int_0^T M_{11}(t) \, dt &\leq C_0 \| \rho_0 \psi \|_{L^\infty(0,T;L^2(\Omega)}\| \nabla \bar{\phi^2} \|_{L^1(0,T;L^\infty(\Omega)} \\
 &\quad \times\big( \| \nabla A_2^T \|_{L^\infty(0,T;L^3(\Omega)} \| A_2^T \|_{L^\infty(0,T\times\Omega)}^2 | \nabla \bar{\phi^2} \|_{L^\infty(0,T;L^6(\Omega)}
 \\&\quad + \| \nabla A_1^T \|_{L^\infty(0,T;L^3(\Omega)}\| A_1^T \|_{L^\infty(0,T\times\Omega)}^3 \| \nabla^2 \bar{\phi}^1 \|_{L^\infty(0,T;L^6(\Omega)} \big),\\
\int_{0}^{T} M_{12}(t) \, dt &\leq \| A_{1}^T \|_{L^{\infty}(0,T \times \Omega)}^2 \big( \| \nabla \bar{\mu}^2 \|_{L^2(0,T \times \Omega)} + \| \nabla \bar{\mu}^1 \|_{L^2(0,T \times \Omega)} \big) \\
&\quad \times \big( \| \nabla \bar{\phi}^2 \|_{L^2(0,T \times \Omega)} + \| \nabla \bar{\phi}^1 \|_{L^2(0,T \times \Omega)} \big),\\
\int_{0}^{T} M_{13}(t) \, dt &\leq \big( \| A_2^T \|_{L^{\infty}(0,T \times \Omega)}^2 + \| A_1^T \|_{L^{\infty}(0,T \times \Omega)}^2 \big) \| \nabla \bar{\mu}^2 \|_{L^2(0,T \times \Omega)}\\&\quad\times \big( \| \nabla \bar{\phi}^2 \|_{L^2(0,T \times \Omega)} + \| \nabla \bar{\phi}^1 \|_{L^2(0,T \times \Omega)} \big), \\
\int_{0}^{T} M_{14}(t) \, dt &\leq \| A_1^T \|_{L^{\infty}(0,T \times \Omega)}^2 \big( \| \nabla \bar{\phi}^2 \|_{L^2(0,T \times \Omega)} + \| \nabla \bar{\phi}^1 \|_{L^2(0,T \times \Omega)} \big)\\&\quad\times \big( \| \nabla \bar{\mu}^2 \|_{L^2(0,T \times \Omega)} + \| \nabla \bar{\mu}^1 \|_{L^2(0,T \times \Omega)} \big), \\
\int_{0}^{T} M_{15}(t) dt &\leq \big( \| |\bar{\phi}^2|^3 \|_{L^2(0,T \times \Omega)} + \| |\bar{\phi}^1|^3 \|_{L^2(0,T \times \Omega)} \big) \| \sqrt{\rho_0} \delta \mu \|_{L^2(0,T \times \Omega)} \\
&\quad + (\| \sqrt{\rho_0}\bar{\phi}^2 \|_{L^{\infty}; L^2(\Omega)}+\| \sqrt{\rho_0}\bar{\phi}^1 \|_{L^{\infty}; L^2(\Omega)})+ \big( \| \bar{\mu}^2 \|_{L^1(0,T; L^2(\Omega))} + \| \bar{\mu}^1 \|_{L^1(0,T; L^2(\Omega))} \big), \\
\int_{0}^{T} M_{16}(t) dt &\leq \big( \| A_2^T \|_{L^{\infty}(0,T \times \Omega)}^2 + \| A_1^T \|_{L^{\infty}(0,T \times \Omega)}^2 \big)
\| \nabla \bar{\phi}^2 \|_{L^2(0,T \times \Omega)} \\&\quad+ \big(\| \nabla \bar{\mu}^2 \|_{L^2(0,T \times \Omega)}+\| \nabla \bar{\mu}^1 \|_{L^2(0,T \times \Omega)} \big),\\
\int_{0}^{T} M_{17}(t) dt &\leq \big( \| A_2^T \|_{L^{\infty}(0,T \times \Omega)}^2 + \| A_1^T \|_{L^{\infty}(0,T \times \Omega)}^2 \big)
\| \nabla \bar{\mathbf{b}}^2 \|_{L^1(0,T; L^\infty(\Omega))} \\&\quad\times C_{\rho_0}
\| \sqrt{\rho_0} \phi \|_{L^{\infty}; L^2(\Omega)}
\big(\|\bar{\mathbf{b}}^2 \|_{L^{\infty}; L^2(\Omega)}+\|\bar{\mathbf{b}}^1 \|_{L^{\infty}; L^2(\Omega)} \big),\\
\int_{0}^{T} M_{18}(t) dt &\leq \big(\| A_2^T \|_{L^{\infty}(0,T \times \Omega)}^2 + \| A_1^T \|_{L^{\infty}(0,T \times \Omega)}^2 \big)
\| \nabla \bar{\mathbf{u}}^2 \|_{L^1(0,T; L^\infty(\Omega))} \\&\quad\times C_{\rho_0}
\| \sqrt{\rho_0} \phi \|_{L^{\infty}; L^2(\Omega)}
\big(\|\bar{\mathbf{b}}^2 \|_{L^{\infty}; L^2(\Omega)}+\|\bar{\mathbf{b}}^1 \|_{L^{\infty}; L^2(\Omega)} \big),\\
\int_{0}^{T} M_{19}(t) dt &\leq C_{\rho_0}
\| \sqrt{\rho_0} \phi \|_{L^{\infty}; L^2(\Omega)}
\|\bar{\mathbf{b}}^1 \|_{L^{\infty}; L^2(\Omega)}\| \nabla_{ \mathbf{u}^1}\mathbf{\psi} \|_{L^2(0,T \times \Omega)},\\
\int_{0}^{T} M_{20}(t) dt &\leq
\| A_1^T \|_{L^{\infty}(0,T \times \Omega)}^2
\|\nabla\bar{\mathbf{b}}^2 \|_{L^2(0,T \times \Omega)}(\|\nabla\bar{\mathbf{b}}^2 \|_{L^2(0,T \times \Omega)}\|+\|\nabla\bar{\mathbf{b}}^1 \|_{L^2(0,T \times \Omega)}\|),\\
\int_{0}^{T} M_{21}(t) dt &\leq
\| A_2^T \|_{L^{\infty}(0,T \times \Omega)}^2
\|\nabla\bar{\mathbf{b}}^2 \|_{L^2(0,T \times \Omega)}(\|\nabla\bar{\mathbf{b}}^2 \|_{L^2(0,T \times \Omega)}\|+\|\nabla\bar{\mathbf{b}}^1 \|_{L^2(0,T \times \Omega)}\|).
\end{align*}
So altogether, and using \eqref{eq-3}, this gives for all small enough \( T > 0 \),
\begin{align}
\sup_{t \in [0,T]}& \big\| (\sqrt{\rho_0} \psi, \delta \mathbf{b}, \sqrt{\rho_0} \delta\phi) \big\|_2^2 + \big\| (\nabla \delta \mathbf{u}, \nabla \delta \mathbf{b}, \sqrt{\rho_0} \delta \mu) \big\|_{L^2(0,T; L^2)}^2  \label{eq-2}\\ \notag
&\leq c(T) \big\| (\nabla \delta \mathbf{u}, \sqrt{\rho_0} \delta \mu) \big\|_{L^2(0,T; L^2)}^2.
\end{align}
By \eqref{eq-1}, we conclude that
\begin{equation*}\begin{split}
&\big\| (\nabla \delta \mathbf{u}, \nabla \delta \mathbf{b}, \sqrt{\rho_0} \delta \mu) \big\|_{L^2(0,T; L^2)}^2
\\&\leq c(T) \big\| (\nabla \delta \mathbf{u}, \nabla \delta \mathbf{b}, \sqrt{\rho_0} \delta \mu) \big\|_{L^2(0,T; L^2)}^2.
\end{split}
\end{equation*}
Hence \( \nabla \delta \mathbf{u} =  \nabla \delta \mathbf{b}=\delta \mu \equiv 0 \) on \( [0,T] \times \Omega \) if \( T \) is small enough. Plugging that information into \eqref{eq-2} yields
that
\begin{align*}
\big\| (\sqrt{\rho_0} \psi, \delta \mathbf{b}, \sqrt{\rho_0} \delta \phi) \big\|_{L^\infty(0,T; L^2)}^2 + \big\| (\nabla \psi, \nabla \delta \mathbf{b}, \sqrt{\rho_0} \delta \mu) \big\|_{L^2(0,T \times \Omega)}^2 = 0.
\end{align*}
It finally implies that \( \psi \equiv 0 \) on \( [0,T] \times \mathbb{T}^d \), and \eqref{eq-1} clearly yields \( \varphi \equiv 0 \). Therefore, for small enough \( T > 0 \),  we  conclude,  by \eqref{eq-1}, that
\begin{align*}
\bar{\mathbf{u}}^1 = \bar{\mathbf{u}}^2, \quad\bar{\mathbf{b}}^1 = \bar{\mathbf{b}}^2, \quad \bar{\phi}^1 = \bar{\phi}^2, \quad \bar{\mu}^1 = \bar{\mu}^2 \quad \text{on} \quad [0,T] \times \Omega.
\end{align*}
Reverting to Eulerian coordinates, we finally deduce that  two solutions of the system \eqref{1.1} coincide on \( [0,T] \times \Omega\).
\end{proof}
\section{Declarations}



\noindent{\bf Competing interests  }\

On behalf of all authors, the corresponding author states that there is no potential conflicts of interest with respect to the research of this article.

\noindent{\bf Authors' contributions   }\

Lingxin Jiang and Fuyi  Xu contributed equally to this work.

\noindent{\bf Funding   }\

Jiang and Xu were partially supported by the
National Natural Science Foundation of China (12326430), the  Natural Science Foundation of Shandong Province (ZR2026MS0022).

\noindent{\bf Availability of data and materials  }\

Data and materials  sharing not applicable to this article as no data and  materials
 were generated or analyzed during the current study.

\begin{center}

\end{center}

\end{document}